\documentclass[11pt,a4paper]{amsart}
\usepackage{amsfonts}
\usepackage{}
\usepackage{amsfonts}
\usepackage{amsmath,xypic}
\usepackage{mathrsfs}
\usepackage{amssymb}

\usepackage{CJK,CJKnumb}
\usepackage[CJKbookmarks,colorlinks,
            linkcolor=black,
            anchorcolor=black,
            citecolor=blue]{hyperref}
\usepackage{color}              
\usepackage{indentfirst}        
\usepackage{latexsym,bm}        
\usepackage{amsmath,amssymb}    
\usepackage{pstricks}
\usepackage{pst-node}
\usepackage{pst-tree}
\usepackage{pst-plot}
\usepackage{pst-text}
\usepackage{graphicx}
\usepackage{cases}
\usepackage{pifont}
\usepackage{txfonts}
\usepackage[all,knot,poly]{xy}

\allowdisplaybreaks[4]  

\CJKtilde   

\newcommand{\R}{\mathscr{R}}

\newtheorem{corollary}{Corollary}[section]
\newtheorem{theorem}{Theorem}[section]
\newtheorem{lemma}{Lemma}[section]
\newtheorem{remark}{Remark}[section]
\newtheorem{proposition}{Proposition}[section]
\newtheorem{definition}{Definition}[section]
\definecolor{black}{rgb}{0.00,0.00,0.45}

\allowdisplaybreaks[4]
                               { \end{spacing}
                               } 
\begin{document}
\title[Double-bosonization and Majid's Conjecture, (II)]{Double-bosonization and Majid's Conjecture, (II):
cases of irregular $R$-matrices and type-crossings of $F_4$, $G_2$}
\author[H. Hu]{Hongmei Hu}
\address{School of Mathematical Sciences, University of Science and Technology of China, Jin Zhai Road 96,
Hefei 230026,
PR China}
\email{hmhu@ustc.edu.cn}

\author[N. Hu]{Naihong Hu$^{\ast\ast}$}
\address{Department of Mathematics,  Shanghai Key Laboratory of Pure Mathematics and Mathematical Practice,
East China Normal University,
Minhang Campus,
Dong Chuan
Road 500,
Shanghai 200241,
PR China}
\email{nhhu@math.ecnu.edu.cn}
\subjclass{Primary 16S40, 16W30, 17B37, 18D10; Secondary  17B10, 20G42, 81R50}
\date{August 12, 2015}


\keywords{Double-bosonization, braided category, braided groups (Nichols algebras), type-crossing construction, normalized $R$-matrix,  representations}

\thanks{$^{\ast\ast}$ N.~H., Corresponding author, supported by the NNSFC (Grant No.
 11271131).}

\date{}
\maketitle

\newcommand*{\abstractb}[3]{ %
                             \begingroup%
                             \leftskip=8mm \rightskip=8mm
                             \fontsize{11pt}{\baselineskip}\noindent{\textbf{Abstract} ~}{#1}\\ 
                             {\textbf{Keywords} ~}{#2}\\ 
                             {\textbf{MR(2010) Subject Classification} ~}{#3}\\ 
                                                          \endgroup
                           }
\begin{abstract}
The purpose of the paper is to build up the related theory of weakly
quasitriangular dual pairs suitably for non-standard $R$-matrices (irregular), and
establish the generalized double-bosonization construction theorem for irregular $R$,
which generalize Majid's results for regular $R$ in \cite{majid1}. As an
application, the type-crossing construction for the exceptional
quantum groups of types $F_{4}$, $G_{2}$ is obtained. This affirms
the Majid's expectation that the tree structure of nodes diagram
associated with quantum groups can be grown out of the node
corresponding to $U_q(\mathfrak{sl}_2)$ by double-bosonization procedures. Notably from a
representation perspective, we find an effective method to get the
minimal polynomials for the non-standard $R$-matrices involved.
\end{abstract}

\section{Introduction}

\noindent{\rm 1.1.}
The charming power of quantum group theory always attracts many
mathematicians to find some better way in a suitable framework to
understand the structure of quantum groups defined initially by
generators and relations. Specially for the quasitriangular Hopf
algebras (with universal $R$-matrices), there are several well-known
general constructions, among which main representatives are: $(1)$
The Drinfeld's quantum double theory of any Hopf algebra introduced
in \cite{dri1}. $(2)$ The FRT-construction in \cite{FRT1} based on
the (standard) $R$-matrices. $(3)$ The Majid's double-bosonization
theory in \cite{majid1} (also see  Sommerh\"auser's construction
\cite{somm} independently working in a slightly different braided
category, i.e., the Yetter-Drinfeld category), which improved the
FRT-construction via extending Drinfeld's quantum double to the
generalized quantum double associated with dually-paired braided
groups (named by Majid) coming from a braided category consisting of
the (co-)modules of a (co-)quasitriangular Hopf algebra. Based on
his series of earlier work \cite{majid2}--\cite{majid5}, Majid
\cite{majid1} in 1995 developed the double-bosonization theory to
yield a direct construction for $U_q(\mathfrak g)$ as an application
of which one can view the Lusztig's algebra $\mathfrak f$
(\cite{lus}) as the braided group in a special braided category
$\mathfrak{M}_H$ (or ${}_H\mathfrak{M}$), where
($H=kQ=k[K_i^{\pm1}]$, $A=kQ^\vee$) is a weakly quasitriangular dual
pair (see p. 169 of \cite{majid1}), where $Q$ (resp. $Q^\vee$) is a
(dual) root lattice of $\mathfrak g$. $(4)$ The Rosso's quantum shuffle construction \cite{rosso} (axiomatic approach) (for the entire construction see \cite{HLR})
based on the braidings, which together with Nichols' earlier work \cite{N} has been closely related to the currently fruitful developments of classifications of the
pointed Hopf algebras or Nichols algebras (or braided groups in terms of Majid), for instance, see \cite{AS}, \cite{AHS}, \cite{HS}, etc.
Let us say more about the double-bosonization below.

\noindent{\rm 1.2.}
Roughly speaking, Majid's double-bosonization theory is as
follows. Associated to any mutually dual braided groups $B^{\star},
B$ covariant under a underlying quasitriangular Hopf algebra H,
there is a new quantum group on the tensor space $B^{\star}\otimes
H\otimes B$ by double bosonization in \cite{majid1}, consisting of
$H$ extended by $B$ as additional ’positive roots’ and its dual
$B^{\star}$ as additional ’negative roots’. Specially, Majid
regarded $U_{q}(\mathfrak{n}^{\pm})$ as the mutually dual braided
groups in a braided category of right (left) $H$-modules with Hopf
``Cartan subalgebra" $H$. On the other hand, based on a couple of
examples of lower rank given in \cite{majid1,majid6,majid7}, Majid
claimed that many novel quantum groups, as well as the
rank-inductive or type-crossing construction (we named) of
$U_q(\mathfrak g)$ can be obtained in principle by this theory,
since he expected his double-bosonization construction would allow
to generate a tree of quantum groups and at each node of the tree,
there should be many choices to adjoin a certain dually-paired
braided groups covariant under a certain quantum subgroup at that
node. In fact, it is a challenge to elaborate the full tree
structure because it involves some rather complicated technical
points needed to be solved from representation theory. Moreover,
it is difficult to explore the information encoded in the $R$-matrix
corresponding to a certain representation of $U_{q}(\mathfrak g)$,
for example, the spectral decomposition of $R$-matrix, etc. By
exploiting the well-known information for standard $R$-matrices (we mean those associated to the vector representations of classical types $ABCD$), the
authors in \cite{HH} obtained the rank-inductive construction of
quantum groups $U_q(\mathfrak g)$ for classical types in the quantum tree of Majid.
Furthermore, can we elucidate the tree structure at the exceptional
nodes? That is,
{\color{black}\it How to construct the exceptional quantum groups via double-bosonization theory?}

\noindent{\rm 1.3.}
Recall that the starting point of Majid's framework \cite{majid1} is
to work with the weakly quasitriangular dual pairs of bialgebras
$(\widetilde{U(R)}, A(R))$, where $\widetilde{U(R)}:=A(R)\Join A(R)$
is the double cross product defined by \cite{majid5}, $A(R)$ is the
FRT-bialgebra defined in \cite{FRT1}. He could seek the weakly
quasitriangular dual pair $(U_q(\mathfrak g),\mathcal{O}_{q}(G))$ as
the quotient Hopf algebras of $(\widetilde{U(R)}, A(R))$ to
establish his double-bosonization Theorem  (see Corollary 5.5
\cite{majid1}) when $R$ is regular (for definition, see p. 174,
\cite{majid1}). In \cite{HH}, when $R$ is one of standard
$R$-matrices which is regular in the sense of Majid, we can work
with a slight different weakly quasitriangular dual pair
$(U_q^{\textrm{ext}}(\mathfrak g),\mathcal{O}_{q}(G))$ that fits into Majid's
original framework (Corollary 5.5 \cite{majid1}), where
$U_q^{\textrm{ext}}(\mathfrak g)$ is the extended Hopf algebra of
$U_q(\mathfrak g)$ of classical type,
$\mathcal{O}_{q}(G)$ is the quotient Hopf algebra of $A(R)$, the
quantum coordinate algebra on the associated simple Lie group $G$ of
classical type (cf. \cite{klim}), so that the braided vector algebra
 $B=V(R',R)\in{}^{\mathcal{O}_{q}(G)}\mathfrak M$, the braided co-vector algebra
 $B^{\ast}=V^{\vee}(R',R_{21}^{-1})\in\mathfrak M^{\mathcal{O}_{q}(G)}$ can be transferred into the originally-required
objects in the braided categories $\mathfrak M_{U_q^{\textrm{ext}}(\mathfrak g)}$, ${}_{U_q^{\textrm{ext}}(\mathfrak
g)}\mathfrak M$,
respectively.

\noindent{\rm 1.4.} While for those non-standard $R$-matrices we
encountered when we deal with the type-crossing construction of the
exceptional quantum groups, the weakly quasitriangular dual pair of
bialgebras $(\widetilde{U(R)}, A(R))$, previously served as the
starting point of Maijd's work, has to be replaced. The new weakly
quasitriangular dual pair needs to be created since the quantum
coordinate algebra $\mathcal O_q(G)$ doesn't meet our requirement
and its substitute is no longer the quotient of $A(R)$. Such
$R$-matrices are not regular in the sense of Majid, we name them
{\it irregular}. This means that it is desirable to establish the
related theory of the weakly quasitriangular dual pairs suitably for
irregular $R$-matrices. Roughly speaking, thanks to Theorem 8
\cite{FRT1}, we can exploit their defining bialgebra $H_R$ as a
suitable candidate, which will be a Hopf algebra when the $R$
involved satisfies the FRT-condition. Also, we can still use the
double cross product $\widetilde{U(R)}$ to yield a suitable quotient
Hopf algebra $\widehat{U(R)}$ we define. Fortunately, we prove that
such a pair $(\widehat{U(R)}, H_R)$ forms a weakly quasitriangular
dual pair in the sense of Majid as we expected. When $R$ is
standard, $\widehat{U(R)}$ coincides with
$U_q^{\textrm{ext}}(\mathfrak g)$ by \cite{klim}, and $H_R$ is
isomorphic to $\mathcal O_q(G)$ due to the Remark of Theorem 8
\cite{FRT1}. For convenience, we still write
$U_q^{\textrm{ext}}(\mathfrak g)$ instead of $\widehat{U(R)}$. On
the other hand, we further argue that $H_R$ is coquasitriangular
when $R$ satisfies our more assumption (than the FRT-condition), so
that the categories of left (right) comodules over $H_R$ are braided
in the sense of Majid and $V(R',R)\in {}^{H_R}\mathfrak M$,
$V^{\vee}(R',R_{21}^{-1})\in \mathfrak M^{H_R}$.

\noindent{\rm 1.5.}
After clarifying the rough destination above, in order to deduce the
generalized version of double-bosonization construction theorem
suitably for general $R$-matrices, especially for the
double-bosonization inductive construction of the exceptional
quantum groups, we have to refine the above framework. Firstly, we
need to replace $U_q^{\textrm{ext}}(\mathfrak g)$ by its central
extension object $\widetilde{U_q^{\textrm{ext}}(\mathfrak
g)}:=U_q^{\textrm{ext}}(\mathfrak g)\otimes \mathbb C[c, c^{-1}]$,
and $H_{R_{VV}}$ by
$\widetilde{H_{R_{VV}}}:=H_{R_{VV}}\otimes\mathbb C[g, g^{-1}]$,
where $R_{VV}$ is the $R$-matrix corresponding to a certain chosen
irreducible representation of $U_q(\mathfrak g)$. Secondly, to
normalize $R_{VV}$ at its certain eigenvalue to gain a quantum group
normalization constant $\lambda$ such that $\langle
c,g\rangle=\lambda$ and $(\widetilde{U_q^{\textrm{ext}}(\mathfrak
g)},\widetilde{H_{R_{VV}}})$ forms a new weakly quasitriangular dual
pair of Hopf algebras as we desired. Thirdly, to get the minimal
polynomial of the braiding $PR_{VV}$, from which we achieve the pair
$(R,R')$, where $R$ is determined by the above normalization of
$R_{VV}$, subsequently, $R'$ can be figured out from the minimal
polynomial. Notice that $A(R)\cong A(R_{VV})\subset
H_{R_{VV}}\subset \widetilde{H_{R_{VV}}}$ as bialgebras, we can
regard the braided objects $B=V(R',
R)\in{}^{\widetilde{A(R)}}\mathfrak M$,
$B^*=V^{\vee}(R',R_{21}^{-1})\in  \mathfrak
M^{\widetilde{A(R)}}$ as the braided objects in ${}^{\widetilde{H_{R_{VV}}}}\mathfrak M$ and $\mathfrak
M^{\widetilde{H_{R_{VV}}}}$, respectively. Via the dual pairing between $\widetilde{U_q^{\textrm{ext}}(\mathfrak g)}$ and $\widetilde{H_{R_{VV}}}$, we get $B=V(R',
R)\in{\mathfrak M}_{\widetilde{U_q^{\textrm{ext}}(\mathfrak g)}}$
and
$B^*=V^{\vee}(R',R_{21}^{-1})\in{}_{\widetilde{U_q^{\textrm{ext}}(\mathfrak
g)}}\mathfrak M$. This affords us a prerequisite for building the generalized double-bosonization construction Theorem for the irregular $R$'s.

\noindent{\rm 1.6.} The paper is organized as follows. In Section 2,
we recall some basic facts about the FRT-construction \cite{FRT1}
and the Majid's double-bosonization \cite{majid1}. In Section 3, for
the irregular $R$-matrices, we start with  the weakly
quasitriangular dual pair of bialgebras $(\widetilde{U(R)}, H_R)$ to
get the required weakly quasitriangular dual pair of Hopf algebras
$(\widehat{U(R)}, H_R)$ when the $R$ involved meets certain strict
conditions. Finally, we establish the generalized
double-bosonization construction theorem for irregular $R$'s.
Section 4 is devoted to applying this theorem to give the
type-crossing double-bosonization constructions of $U_{q}(F_{4})$
and $U_{q}(G_{2})$. First of all, in order to work
out the construction of $U_{q}(F_{4})$, we start from the nodes
diagram associated to $U_{q}(B_{3})$ and its $8$-dimensional spin
representation $T_{V}$, and prove that the braiding $\hat R_{VV}$ is of diagonal type and find an ingenious method to
figure out its minimal polynomial, which
captures/depends on some features of the defining representation
rather than the whole information of $R$-matrix itself. Consequently, making the normalization for $R_{VV}$ at a certain eigenvalue of its braiding to get the pair $(R, R')$ from which
determine our required dually-paired braided groups, we prove the quantum group we constructed is just $U_{q}(F_{4})$. Next, for constructing $U_{q}(G_{2})$,
we begin with $U_{q}(A_{1})$ and its spin $\frac{3}{2}$
representation, which is built on the $4$-dimensional homogeneous
submodule of degree $3$ in the $U_{q}(sl_{2})$-module algebra (cf.
\cite{H}) $\mathcal A_q=\mathbb{C}_q\langle x, y\rangle$ with
$yx=qxy$ (namely, the Manin's quantum plane). Similar information necessary for $U_q(G_2)$ is provided.
We believe that the results and methods of this paper will be useful for seeking new quantum groups or Hopf algebras arising from various Nichols algebras currently focused on by
many Hopf algebraists (for instance, see \cite{AHS}, \cite{AS}, \cite{CL}, \cite{HS}, etc. and references therein).

\section{Preliminaries}
Let us fix some general notation which will be kept throughout this paper.
The letters $\mathbb{C}, \mathbb{Z}$ always stand for the complex field,
the set of integer numbers,
respectively.
$0\neq q\in \mathbb{C}$ and $n\in \mathbb{Z}_+$,
then
$$[n]_q=\frac{q^n-q^{-n}}{q-q^{-1}},\quad
[n]_q!=\prod_{k=1}^n [k]_q,\quad
\begin{bmatrix}
n\\ k
\end{bmatrix}_q
=\frac{[n]_q!}{[n-k]_q![k]_q!}.$$
Let $\mathbb{R}$ be the real field,
$E$ the Euclidean space $\mathbb{R}^{n}$ or a suitable
subspace. Denote by $\varepsilon_{i}$'s the usual orthogonal unit
vectors in $\mathbb{R}^{n}$.
Let $\mathfrak{g}$ be a finite dimensional complex simple Lie algebra with simple roots $\alpha_{i}$'s,
$\lambda_{i}$ the fundamental weight corresponding to simple root $\alpha_{i}$.
Cartan matrix of $\mathfrak g$ is $(a_{ij})$,
where $a_{ij}=\frac{2(\alpha_{i},\alpha_{j})}{(\alpha_{i},\alpha_{i})}$,
and $d_{i}=\frac{(\alpha_{i},\alpha_{i})}{2}$.
Let $(H,\R)$ be a quasitriangular Hopf algebra,
where $\R$ is the universal $R$-matrix,
$\R=\R^{(1)}\otimes \R^{(2)}$,
$\R_{21}=\R^{(2)}\otimes \R^{(1)}$.
Denote by $\Delta, \eta, \epsilon, S$ its coproduct,
counit,
unit,
and antipode of $H$, respectively.
We shall use Sweedler's notation:
$\Delta(h)=h_{1}\otimes h_{2}$, for $h\in H$.
$H^{\textrm{op}} \ (H^{\textrm{cop}})$ denotes the opppsite (co)algebra structure of $H$.
$\mathfrak{M}_{H} \ ({}_{H}\mathfrak{M}$) denotes the braided category consisting of right (left) $H$-modules.
If $(A,\mathfrak{r})$ is a coquasitriangular Hopf algebra $A$,
where $\mathfrak{r}$ is the universal r-form of $A$,
then ${}^{A}\mathfrak{M} \ (\mathfrak{M}^{A})$ denotes the braided category  consisting of left (right) $A$-comodules.
For a detailed description of these theories,
we left the readers to refer to Drinfeld's and Majid's papers \cite{dri2}, \cite{majid2}, \cite{majid3}, and so on.
By a braided group we mean a braided bialgebra or Hopf algebra in some braided category.
In order to distinguish from the ordinary Hopf algebras,
we denote by
$\underline{\Delta},\underline{S}$ its coproduct and antipode, respectively.

\subsection{FRT-construction}
An invertible solution of the quantum Yang-Baxter equation (in brief, QYBE) $R_{12}R_{13}R_{23}=R_{23}R_{13}R_{12}$ is called $R$-matrix.
There is a bialgebra $A(R)$ \cite{FRT1} corresponding to any invertible $R$-matrix,
called the FRT-bialgebra.
\begin{definition}
$A(R)$ is generated by $1$ and $t=\{t^{i}_{j}\}$,
having the following structure:
$$RT_{1}T_{2}=T_{2}T_{1}R,\quad
\Delta(T)=T\otimes T,\quad
\epsilon(T)=I~(\mbox{unit~matrix}).$$
$A(R)$ is a coquasitriangular bialgebra with
$\R:~A(R)\otimes A(R) \longrightarrow k$ such that $\R(t^{i}_{j}\otimes t^{k}_{l})=R^{ik}_{jl}$.
Here
$T_{1}=T\otimes I$,
$T_{2}=I\otimes T$, $R^{ik}_{jl}$ denotes the entry at row $(ik)$ and column $(jl)$ in $R$.
\end{definition}
There is a general procedure to construct the Hopf algebra $H_{R}$ associated with the FRT-bialgebra $A(R)$ in \cite{FRT1}.
The algebra $H_{R}$ is generated by  $1$ and $t^{i}_{j},\widetilde{t}^{i}_{j},i,j=1,\cdots,n$,
with the relations
\begin{gather}
RT_{1}T_{2}=T_{2}T_{1}R,\quad
R^{t}\widetilde{T}_{1}\widetilde{T}_{2}=\widetilde{T}_{2}\widetilde{T}_{1}R^{t},\quad
(R^{t_{2}})^{-1}T_{1}\widetilde{T}_{2}=\widetilde{T}_{2}T_{1}(R^{t_{2}})^{-1},\label{tilde}
\end{gather}
where $(R^t)^{ij}_{kl}=R^{kl}_{ij}$,
$(R^{t_1})^{ij}_{kl}=R^{kj}_{il}$,
$(R^{t_2})^{ij}_{kl}=R^{il}_{kj}$. Specially, assume that
$R^{t_{2}}$ is nonsingular and
\begin{equation}\label{chap21}
(R^{-1})^{t_{1}}P(R^{t_{2}})^{-1}PK_{0}=\textrm{const}\cdot K_{0},
\end{equation}
where $P$ is the permutation matrix with the entry
$P^{ij}_{kl}=\delta_{il}\delta_{jk}$, ``\text{const}" means a
diagonal matrix consisting of constants, and
$(K_{0})^{ij}_{kl}:=\delta_{ij}\delta_{kl}$, then $H_{R}$ is a Hopf
algebra with the structure maps
$$\Delta(T)=T\otimes T,\quad\Delta(\widetilde{T})=\widetilde{T}\otimes \widetilde{T},\quad
\epsilon(T)=\epsilon(\widetilde{T})=I;\quad
S(T)=(\tilde{T})^{t},\quad
S(\tilde{T})=DT^{t}D^{-1}.
$$
Here
$D=\textrm{tr}_{2}(P((R^{t_{2}})^{-1})^{t_{1}})\in M_{n}(\mathbb{C}),$
and $\textrm{tr}_{2}$ denotes matrix trace in the second factor in the tensor product $\mathbb{C}^{n}\otimes \mathbb{C}^{n}$.
In particular,
for the standard $R$-matrices,
$H_{R}$'s are isomorphic to quantum coordinate functions algebras $\mathcal{O}_{q}(G)$'s on the corresponding Lie groups $G$ (see the Remark of Theorem 8
\cite{FRT1}).
We will consider those Hopf algebras $H_{R}$ for the general $R$-matrices in the next section.

In the braided category ${}^{A(R)}\mathfrak{M} \
(\mathfrak{M}^{A{(R)}})$, there are two classical braided groups
$V(R',R)$ and $V^{\vee}(R',R_{21}^{-1})$ in \cite{majid4}, called
the braided (co-)vector algebras, where $R^{\prime}$ is another
matrix satisfying
\begin{gather}
 R_{12}R_{13}R^{\prime}_{23}=R^{\prime}_{23}R_{13}R_{12},\qquad
R_{23}R_{13}R^{\prime}_{12}=R^{\prime}_{12}R_{13}R_{23},\label{*}\\
 (PR+1)(PR^{\prime}-1)=0,\qquad
R_{21}R^{\prime}_{12}=R^{\prime}_{21}R_{12}.\label{**}
\end{gather}
The braided vector algebra $V(R^{\prime},R)$ defined as a quadratic
algebra with generators $1$, $\{e^{i}~|~i=1,\cdots,n\}$ and
relations $e^{i}e^{j}=\sum\limits_{a,b} R'{}^{ji}_{ab}e^{a}e^{b}$,
forms a braided group with $\underline{\Delta}(e^{i})=e^{i}\otimes
1+1\otimes e^{i}, \underline{\epsilon}(e^{i})=0,
\underline{S}(e^{i})=-e^{i}, \Psi(e^{i}\otimes
e^{j})=\sum\limits_{a,b}R^{ji}_{ab}e^{a}\otimes e^{b}$ in braided
category ${}^{A(R)}\mathfrak{M}$. Under the duality $\langle
f_{j},e^{i}\rangle=\delta_{ij}$, the braided co-vector algebra
$V^{\vee}(R^{\prime},R_{21}^{-1})$ defined by $1$, $\{f_{j}\mid
j=1,\cdots,n\}$, and relations
$f_{i}f_{j}=\sum\limits_{a,b}f_{b}f_{a}R'{}^{ab}_{ij}$ forms another
braided group with $\underline{\Delta}(f_{i})=f_{i}\otimes
1+1\otimes f_{i},$ $\underline{\epsilon}(f_{i})=0,$
$\underline{S}(f_{i})=-f_{i},$ $\Psi(f_{i}\otimes
f_{j})=\sum\limits_{a,b}f_{b}\otimes f_{a}R^{ab}_{ij}$ in braided
category $\mathfrak{M}^{A(R)}$.

\begin{remark}\label{rem1}
In fact, when we know the minimal polynomial equation $ \prod_{i}(PR-x_{i})=0
$ of the braiding $PR$, we can normalize $R$ at a certain eigenvalue
$x_{i}$ of $PR$ so that $x_{i}=-1$, then set $
R'=P+P\prod\limits_{j\neq i}(PR-x_{j})$. This gives us a way to get the pair $(R, R')$ that satisfy conditions
\eqref{*} and \eqref{**}.
\end{remark}

\subsection{Majid's double-bosonization}
Let
$C, B$ be a pair of braided groups in $\mathfrak{M}_{H}$,
which are called dually paired if there is an intertwiner
$\textrm{ev}: C\otimes B \longrightarrow k$ with
$
\textrm{ev}(cd,b)=\textrm{ev}(d,b_{\underline{(1)}})\textrm{ev}(c,b_{\underline{(2)}}),$
$
\textrm{ev}(c,ab)=\textrm{ev}(c_{\underline{(2)}},a)\textrm{ev}(c_{\underline{(1)}},b),$
$
\forall a,b\in B,c,d\in C.
$
Then $C^{\textrm{op}/\textrm{cop}}$ (with opposite product and coproduct) is a Hopf algebra in $_{H}\mathfrak{M}$,
which is dual to $B$ in the sense of an ordinary duality pairing $\langle~,~\rangle$ being $H$-bicovariant.
Let $\overline{C}=(C^{\textrm{op}/\textrm{cop}})^{\underline{\textrm{cop}}}$,
then $\overline{C}$ is a braided group in $_{\overline{H}}\mathfrak{M}$,
where $\overline{H}$ is $(H,\R_{21}^{-1})$.
With these, Majid gave the following double-bosonization theorem and some results in \cite{majid1}:
\begin{theorem}\label{ml1} $(${\rm\bf Majid}$)$
On the tensor space $\bar{C}\otimes H \otimes B$,
there is a unique Hopf algebra structure $U=U(\bar{C},H,B)$ such that
$H\ltimes B$ (bosonization) and $\bar{C}\rtimes H$ (bosonization) are sub-Hopf algebras by the canonical inclusions and
$$
\left.
\begin{array}{rl}
bc=&(\R_{1}^{(2)}\rhd c_{\overline{(2)}})
\R_{2}^{(2)}\R_{1}^{-(1)}(b_{\underline{(2)}}\lhd\R_{2}^{-(1)})\,
\langle\R_{1}^{(1)}\rhd c_{\overline{(1)}}, b_{\underline{(1)}}\lhd
\R_{2}^{(1)}\rangle\,\cdot\\
&\qquad\qquad\qquad\quad\,\langle\R_{1}^{-(2)}\rhd \overline{S}c_{\overline{(3)}}, b_{\underline{(3)}}\lhd \R_{2}^{-(2)}\rangle,
\end{array}
\right.
$$
for all $ b\in B, \,c\in\overline{C}$ viewed in $U$.
The product, coproduct of $U$ are given by
$$
\left.
\begin{array}{rl}
(c\otimes h\otimes b)\cdot(d\otimes g\otimes a)=&c(h_{(1)}\R_{1}^{(2)}\rhd d_{\overline{(2)}})
\otimes h_{(2)}\R_{2}^{(2)}\R_{1}^{-(1)}g_{(1)}\otimes (b_{\underline{(2)}}\lhd\R_{2}^{-(1)}g_{(2)})a\\
&\langle\R_{1}^{(1)}\rhd d_{\overline{(1)}},b_{\underline{(1)}}\lhd\R_{2}^{(1)}\rangle
\langle\R_{1}^{-(2)}\rhd\overline{S}d_{\overline{(3)}},b_{\underline{(3)}}\lhd\R_{2}^{-(2)}\rangle;
\end{array}
\right.
$$
$$
\Delta(c\otimes h\otimes b)=c_{\overline{(1)}}\otimes \R^{-(1)}h_{(1)}\otimes b_{\underline{(1)}}\lhd\R^{(1)}
\otimes \R^{-(2)}\rhd c_{\overline{(2)}}\otimes h_{(2)}\R^{(2)}\otimes b_{\underline{(2)}}.
$$
Moreover,
the antipodes of $H\ltimes B$ and $\bar{C}\rtimes H$ can be extended to an antipode $S: U \rightarrow U$ by
the two extensions
$S(chb)=(Sb)\cdot(S(ch))$
and
$S(chb)=(S(hb))\cdot(Sc)$.

Here,
$\rhd, \lhd$ refer to left, right actions respectively,
$\R_{1},\R_{2}$
are distinct copies of the quasitriangular structure $\R$ of $H$.
\end{theorem}
Moreover,
Majid \cite{majid1} proposed the concept of a weakly quasitriangular dual pair via his insight on more examples on
matched pairs of bialgebras or Hopf algebras in \cite{majid5}.
\begin{definition}
Let $(H,A)$ be a pair of Hopf algebras equipped with a dual pairing $\langle ,\,\rangle$
and convolution-invertible algebra\,/\,anti-coalgebra maps $\R,\bar{\R}: A \rightarrow H$ obeying
$$\langle\bar{\R}(a),b\rangle=\langle\R^{-1}(b),a\rangle,\quad
\partial^{R}h=\R \ast(\partial^{L}h)\ast\R^{-1},\quad
\partial^{R}h=\bar{\R}\ast(\partial^{L}h)\ast\bar{\R}^{-1}
$$
for
$a, b\in A, \ h\in H$.
Here $\ast$ is the convolution product in $\text{hom}(A,H)$ and
$(\partial^{L}h)(a)=\langle h_{(1)},a\rangle h_{(2)},
$
$
(\partial^{R}h)(a)=h_{(1)}\langle h_{(2)},a\rangle$
are left, right ``differentiation operators" regarded as maps $A \rightarrow H$ for any fixed $h$.
\end{definition}
For convenience,
we give in brief the definition of coquasitriangular Hopf algebra.
\begin{definition}\label{decoquasi}
A coquasitriangular Hopf algebra is a Hopf algebra $A$ equipped with a linear form $\mathfrak{r}: A\otimes A \longrightarrow \mathbb{C}$ such that the following conditions hold:

$(i)$ $\mathfrak{r}$ is invertible with respect to the convolution,
that is,
there exists another linear form $\bar{\mathfrak{r}}: A\otimes A \longrightarrow \mathbb{C}$
such that $\mathfrak{r}\bar{\mathfrak{r}}=\bar{\mathfrak{r}}\mathfrak{r}=\epsilon\otimes\epsilon$ on $A\otimes A$,

$(ii)$ $
\mathfrak{r}\circ (\cdot\otimes \text{id})=\mathfrak{r}_{13}\ast\mathfrak{r}_{23},\quad
\mathfrak{r}\circ(\text{id}\otimes\cdot)=\mathfrak{r}_{13}\ast\mathfrak{r}_{12},\quad
\cdot\circ\tau=\mathfrak{r}\ast\cdot\ast\mathfrak{r}^{-1}.
$
\end{definition}
\begin{remark}\label{Cross}
If there exists a coquasitriangular Hopf algebra $A$ such that $(H,A)$ is a weakly quasitriangular dual pair,
and
$b,
c$
are primitive elements.
Then many cross relations in the above theorem can be simplified,
for example,
\begin{gather*}
[b,c]=\R(b^{\overline{(1)}})\langle c,b^{\overline{(2)}}\rangle-
\langle c^{\overline{(1)}},b\rangle\bar{\R}(c^{\overline{(2)}});
\\
\Delta b=b^{\overline{(2)}}\otimes \R(b^{\overline{(1)}})+1\otimes b,\quad
\Delta c=c\otimes 1+\bar{\R}(c^{\overline{(2)}})
\otimes c^{\overline{(1)}}.
\end{gather*}
\end{remark}

\section{Generalized Double-bosonization Construction Theorem}
Throughout the following sections,
let $T_{V}$ be an irreducible representation of $U_{q}(\mathfrak g)$ or the $h$-adic Drinfeld-Jimbo algebra $U_{h}(\mathfrak g)$ with
$V$ a finite-dimensional vector space.
\subsection{The universal $R$-matrix and $L$-functionals of $U_{h}(\mathfrak g)$ }
In order to provide an explicit expression for the universal $R$-matrix $\mathcal{R}$ of $U_{h}(\mathfrak g)$,
the key step is to determine dual bases of the $h$-adic vector spaces $U_h(\mathfrak b_+)$ and
$\tilde{U}_h(\mathfrak b_-):=\mathbb{C}[[h]]\cdot1+hU_h(\mathfrak b_-)$ with respect to the bilinear form
$\langle~,~\rangle '$:
\begin{equation}\label{imp0}
\langle H_i, \tilde{H}_j\rangle '=d_{i}^{-1}a_{ji},
\quad
\langle H_j, \tilde{F}_i\rangle '=\langle E_{j}, \tilde{H}_i\rangle'=0,
\quad
\langle E_i, \tilde{F}_j\rangle '=\delta_{ij}\frac{h}{e^{d_{i}h}-e^{-d_{i}h}},
\end{equation}
where
$\tilde{H}_i:=hH_i$,
$\tilde{F}_i=hH_i$
and
$\langle aa',b\rangle '=\langle a',b_{(1)}\rangle '\langle a,b_{(2)}\rangle '$,
$
\langle a,bb'\rangle '=\langle a_{(1)},b\rangle '\langle a_{(2)},b'\rangle '
$.
To do this,
there is the following Lemma \ref{orbit}.
\begin{lemma}\label{orbit}
If a positive root $\beta$ and a simple root $\alpha_{i}$ of $\mathfrak g$ belong to the same Weyl group $W$-orbit,
and they satisfy
\begin{equation}\label{imp1}
\langle E_\beta, \tilde{F}_\beta\rangle '=\langle E_i, \tilde{F}_i\rangle '=\frac{h}{q_i-q_i^{-1}},
\end{equation}
then using the above relations \eqref{imp0},
the dual bases of $U_h(\mathfrak b_+)$ and
$\tilde{U}_h(\mathfrak b_-)$ are obtained.
\end{lemma}
These dual bases are denoted by $e_{i}$ resp. $f_{i}$,
and then $\mathcal{R}_{D}:=\sum\limits_i(1\otimes e_i)\otimes(f_i\otimes 1)$ is a universal $R$-matrix of quantum double
$D(U_h(\mathfrak b_+),\tilde{U}_h(\mathfrak b_-))$.
Taking advantage of the canonical homomorphism $\pi$ of $D(U_h(\mathfrak b_+),\tilde{U}_h(\mathfrak b_-))$ to $U_h(\mathfrak g)$,
then $\mathcal{R}:=(\pi\otimes\pi)\mathcal{R}_{D}$ is a universal $R$-matrix of $U_h(\mathfrak g)$.
Especially,
the root elements $E_\beta$ and $F_\beta$,
obtained by Lusztig's automorphism $T_i$'s \cite{lus},
satisfy the condition \eqref{imp1},
and then an corresponding explicit form of the universal $R$-matrix is \cite{klim}:
\begin{equation}\label{imp2}
\mathcal{R}=\textrm{exp}(h\sum\limits_{i,j}B_{ij}(H_{i}\otimes H_{j}))\prod\limits_{\beta\in \Delta_{+}}\textrm{exp}_{q_{\beta}}((1-q_{\beta}^{-2})(E_{\beta}\otimes F_{\beta})).
\end{equation}
Here the matrix $(B_{ij})$ is the inverse of the matrix $(C_{ij})=(d_{j}^{-1}a_{ij})$,
the $q$-exponential function $\textrm{exp}_{q}x$ is defined by $\textrm{exp}_{q}x=\sum\limits_{r=0}^{\infty}\frac{q^{\frac{r(r+1)}{2}}}{[r]_{q}!}x^{r}.$

On the other hand,
corresponding to an irreducible representation $T_V$ of $U_h(\mathfrak g)$,
there are uniquely determined elements $l_{ij}^{\pm}\in U_{h}(\mathfrak g)$ such that
\begin{gather*}
(\textrm{id}\otimes T_{V})(\mathcal{R})(1\otimes v_{j})=\sum\limits_{i}l_{ij}^{+}\otimes v_{i},\\
(T_{V}\otimes \textrm{id})(\mathcal{R}^{-1})(v_{j}\otimes 1)=(T_{V}\otimes \textrm{id})(S\otimes \textrm{id})(\mathcal{R})(v_{j}\otimes 1)=\sum\limits_{i}v_{i}\otimes l_{ij}^{-},
\end{gather*}
These  $l_{ij}^{\pm}$ are called $L$-functionals associated with the representation $T_{V}$ \cite{klim}.
Moreover,
these matrices $\text{L}^{\pm}=(l_{ij}^{\pm})$ generate a bialgebra,
denoted by $U(\text{L}^{\pm})$,
satisfying the following relations
$$\text{L}_{1}^{\pm}\text{L}_{2}^{\pm}R_{VV}=R_{VV}\text{L}_{2}^{\pm}\text{L}_{1}^{\pm},\quad
\text{L}_{1}^{+}\text{L}_{2}^{-}R_{VV}=R_{VV}\text{L}_{2}^{-}\text{L}_{1}^{+},\quad
\Delta(\text{L}^{\pm})=\text{L}^{\pm}\otimes\text{L}^{\pm},\quad
\epsilon(\text{L}^{\pm})=I.
$$

\subsection{$R$-matrices for representations of $U_q(\mathfrak g)$}
By a certain algebra automorphism of the completion
$\bar{U}_q^{+}(\mathfrak g)\bar{\otimes}\bar{U}_q^{+}(\mathfrak g)$ (see P.264 in \cite{klim}),
the above universal $R$-matrix \eqref{imp2} can yield a universal $R$-matrix of $U_q(\mathfrak g)$,
that is,
$$
\mathfrak{R}=\sum\limits_{r_{1},\cdots,r_{n}=0}^{\infty}\prod\limits_{j=1}^{n}
\frac{(1-q_{\beta_{j}}^{-2})^{r_{j}}}{[r_{j}]_{q_{\beta_{j}}}!}q_{\beta_{j}}^{\frac{r_{j}(r_{j}+1)}{2}}E_{\beta_{j}}^{r_{j}}\otimes
F_{\beta_{j}}^{r_{j}}.$$
Then the $R$-matrix datum $R_{VV}$ is obtained by
$R_{VV}=B_{VV}\circ(T_{V}\otimes T_{V})(\mathfrak{R})$, where
$B_{VV}$ denotes the linear operator on $V\otimes V$ given by
$B_{VV}(v\otimes w):=q^{(\mu,\mu^{\prime})}v\otimes w$, for $v\in
V_{\mu},$ $w\in V_{\mu^{\prime}}$.
We can take a basis $\{ v_{i}\}$ of $V$,
$i=1, \cdots, \textrm{dim}V$,
then
\begin{equation}\label{rmatrix}
B_{VV}\circ(T_{V}\otimes
T_{V})(\mathfrak{R})(v_{i}\otimes v_{j})=R_{VV}{}^{mn}_{ji}(v_{n}\otimes v_{m}).
\end{equation}
It need to pay attention to that the $R$-matrix in Majid's paper \cite{majid1} is the $P\circ\cdot\circ P$ of the ordinary $R$-matrix.
In general, $B_{VV}\circ(T_{V}\otimes
T_{V})(\mathfrak{R})(v_{i}\otimes v_{j})=R_{VV}{}^{nm}_{ij}(v_{n}\otimes v_{m})$ in some other references.
However,
$P\circ R\circ P=R^{t}$ when $PR$ is symmetric,
so then we write occasionally
$B_{VV}\circ(T_{V}\otimes
T_{V})(\mathfrak{R})(v_{i}\otimes v_{j})=R_{VV}{}^{ij}_{nm}(v_{n}\otimes v_{m})$ for convenience.
Notice that we will use Majid's form of $R$-matrix in the remaining sections of this paper.

\subsection{The extended Hopf algebra $U_{q}^{\textrm{ext}}(\mathfrak{g})$}
Majid gave the concept of ``weak antipodes" for the dually paired bialgebras in \cite{majid5}.
Dually paired bialgebras $H_{1}, H_{2}$ are said to possess weak antipodes if the canonical maps
$H_{1} \longrightarrow H_{2}^{\ast}$ and $ H_{2} \longrightarrow H_{1}^{\ast}$ defined by
the pairing are invertible in the convolution algebras $\textrm{Hom}(H_{1}, H_{2}^{\ast})$ and $\textrm{Hom}(H_{2}, H_{1}^{\ast})$ respectively.
These weak antipodes are denoted by $\mathbb{S}_{H_{1}}, \mathbb{S}_{H_{2}}$.
On the other hand,
Majid defined $\widetilde{U(R)}=A(R)\Join A(R)$ the double cross product bialgebra constructed in \cite{majid5}.
It consists of two copies of $A(R)$ generated by $1$ and $m^{\pm}$ with the cross relations and coalgebra structure:
$$
Rm^{\pm}_{1}m^{\pm}_{2}=m^{\pm}_{2}m^{\pm}_{1}R,
\quad
Rm_{1}^{+}m_{2}^{-}=m_{2}^{-}m_{1}^{+}R,
\quad
\Delta((m^{\pm})^{i}_{j})=(m^{\pm})_{j}^{a}\otimes (m^{\pm})_{a}^{i},
\quad
\epsilon((m^{\pm})^{i}_{j})=\delta_{ij}.
$$
For any invertible $R$-matrix $R$,
the bialgebras $\widetilde{U(R)}$ and FRT-bialgebra $A(R)$ are dually-paired by
$\langle(m^{+})^{i}_{j},t^{k}_{l}\rangle= R^{i}_{j}{}^{k}_{l},
$
$
\langle (m^{-})^{i}_{j},t^{k}_{l}\rangle =(R^{-1}){}^{k}_{l}{}^{i}_{j}
$ in \cite{majid1},
moreover,
we have the following (revised version of \cite{majid5})
\begin{proposition}
Let $R$ be an invertible matrix solution of the QYBE.
If matrices $R^{t_{2}}$ and $(R^{-1})^{t_{2}}$ are also invertible,
and set $\widetilde{R}=((R^{t_{2}})^{-1})^{t_{2}}$,
$\widetilde{R^{-1}}=[((R^{-1})^{t_{2}})^{-1}]^{t_{2}}$,
then the dual pairing $(\widetilde{U(R)},A(R))$ possesses weak antipodes,
which are defined by
\begin{gather}\label{weak}
(\mathbb{S}_{A(R)}t^{k}_{l})((m^{+})^{i}_{j})
=(\mathbb{S}_{\widetilde{U(R)}}(m^{+})^{i}_{j})(t^{k}_{l})
=\widetilde{R}^{ik}_{jl}=((R^{t_{2}})^{-1})^{il}_{jk},\\
(\mathbb{S}_{A(R)}t^{k}_{l})(m^{-})^{i}_{j})
=(\mathbb{S}_{\widetilde{U(R)}}(m^{-})^{i}_{j})(t^{k}_{l})
=\widetilde{R^{-1}}^{ki}_{lj}=[((R^{-1})^{t_{2}})^{-1}]^{kj}_{li}.
\end{gather}
Their equivalent expressions in matrices are
\begin{gather*}
(\mathbb{S}_{A(R)}t_{2})(m_{1}^{+})=(\mathbb{S}_{\widetilde{U(R)}}m_{1}^{+})(t_{2})=\widetilde{R},\\
(\mathbb{S}_{A(R)}t_{2})(m_{1}^{-})=(\mathbb{S}_{\widetilde{U(R)}}m_{1}^{-})(t_{2})=P\circ\widetilde{R^{-1}}\circ P.
\end{gather*}
\end{proposition}
\begin{proof}
The corresponding canonical maps for the dual pair $(\widetilde{U(R)}, A(R))$ are
$$
j_{A(R)}(t^{i}_{j})=\langle\cdot,t^{i}_{j}\rangle\in (\widetilde{U(R)})^{\ast},
\quad
j_{\widetilde{U(R)}}((m^{\pm})^{i}_{j})=\langle(m^{\pm})^{i}_{j},\cdot\rangle\in (A(R))^{\ast}.
$$
The units in the convolution algebras $\textrm{Hom}(A(R),(\widetilde{U(R)})^{\ast})$ and $\textrm{Hom}(\widetilde{U(R)},(A(R))^{\ast})$
are $\eta_{(\widetilde{U(R)})^{\ast}}\circ\epsilon_{A(R)}$ and $\eta_{(A(R))^{\ast}}\circ\epsilon_{\widetilde{U(R)}}$, respectively,
and
$$
\left.
\begin{array}{rl}
[(\eta_{(\widetilde{U(R)})^{\ast}}\circ\epsilon_{A(R)})(t^{i}_{j})]((m^{\pm})^{k}_{l})
&=\epsilon_{A(R)}(t^{i}_{j})\langle1,(m^{\pm})^{k}_{l}\rangle
=\epsilon_{A(R)}(t^{i}_{j})\epsilon_{\widetilde{U(R)}}((m^{\pm})^{k}_{l})\\
&=\delta_{ij}\delta_{kl}
=I^{ik}_{jl}.
\end{array}
\right.
$$
$
(\eta_{(A(R))^{\ast}}\circ\epsilon_{\widetilde{U(R)}})((m^{\pm})^{i}_{j})(t^{k}_{l})=I^{ik}_{jl}
$
can also be obtained.
For any $t^{i}_{j}$ and $(m^{\pm})^{k}_{l}$,
we have
\begin{gather}
\left.
\begin{array}{rl}
[(\mathbb{S}_{A(R)}\ast j_{A(R)})(t^{i}_{j})]((m^{+})^{k}_{l})
&=[\mathbb{S}_{A(R)}(t^{i}_{a})j_{A(R)}(t^{a}_{j})]((m^{+})^{k}_{l})\\
&=(\mathbb{S}_{A(R)}(t^{i}_{a}))((m^{+})^{b}_{l})(j_{A(R)}(t^{a}_{j}))((m^{+})^{k}_{b})\\
&=(\mathbb{S}_{A(R)}(t^{i}_{a}))((m^{+})^{b}_{l})\langle(m^{+})^{k}_{b},t^{a}_{j}\rangle\\
&=(R^{t_{2}})^{-1}{}^{ba}_{li}R^{ka}_{bj}
=(R^{t_{2}})^{-1}{}^{ba}_{li}R^{t_{2}}{}^{kj}_{ba}\\
&=(R^{t_{2}}(R^{t_{2}})^{-1})^{kj}_{li}
=I^{kj}_{li}=\delta_{ij}\delta_{kl}\\
&=[(\eta_{(\widetilde{U(R)})^{\ast}}\circ\epsilon_{A(R)})(t^{i}_{j})]((m^{+})^{k}_{l}).
\end{array}
\right.\label{weak1}\\
\left.
\begin{array}{rl}
[(\mathbb{S}_{A(R)}\ast j_{A(R)})(t^{i}_{j})]((m^{-})^{k}_{l})
&=[\mathbb{S}_{A(R)}(t^{i}_{a})j_{A(R)}(t^{a}_{j})]((m^{-})^{k}_{l})\\
&=(\mathbb{S}_{A(R)}(t^{i}_{a}))((m^{-})^{b}_{l})(j_{A(R)}(t^{a}_{j})((m^{-})^{k}_{b})\\
&=(\mathbb{S}_{A(R)}(t^{i}_{a}))((m^{-})^{b}_{l})\langle(m^{-})^{k}_{b},t^{a}_{j}\rangle\\
&=((R^{-1})^{t_{2}})^{-1}{}^{ba}_{li}(R^{-1})^{ak}_{jb}
=((R^{-1})^{t_{2}})^{-1}{}^{il}_{ab}(R^{-1})^{t_{2}}{}^{ab}_{jk}\\
&=((R^{-1})^{t_{2}})^{-1}(R^{-1})^{t_{2}})^{il}_{jk}
=I^{il}_{jk}=\delta_{ij}\delta_{kl}\\
&=[(\eta_{(\widetilde{U(R)})^{\ast}}\circ\epsilon_{A(R)})(t^{i}_{j})]((m^{-})^{k}_{l}).
\end{array}
\right.\label{weak2}\\
\left.
\begin{array}{rl}
[(j_{A(R)}\ast \mathbb{S}_{A(R)})(t^{i}_{j})]((m^{+})^{k}_{l})
&=[(j_{A(R)})(t^{i}_{a})\mathbb{S}_{A(R)}(t^{a}_{j})]((m^{+})^{k}_{l})\\
&=(j_{A(R)})(t^{i}_{a}))((m^{+})^{b}_{l})(\mathbb{S}_{A(R)}(t^{a}_{j}))((m^{+})^{k}_{b})\\
&=\langle(m^{+})^{b}_{l},t^{i}_{a}\rangle\mathbb{S}_{A(R)}(t^{a}_{j})((m^{+})^{k}_{b})\\
&=R^{bi}_{la}(R^{t_{2}})^{-1}{}^{kj}_{ba}
=R^{ba}_{li}(R^{t_{2}})^{-1}{}^{kj}_{ba}\\
&=((R^{t_{2}})^{-1}R^{t_{2}})^{kj}_{li}
=I^{kj}_{li}=\delta_{ij}\delta_{kl}\\
&=[(\eta_{(\widetilde{U(R)})^{\ast}}\circ\epsilon_{A(R)})(t^{i}_{j})]((m^{+})^{k}_{l}).
\end{array}
\right.\label{weak3}\\
\left.
\begin{array}{rl}
[(j_{A(R)}\ast \mathbb{S}_{A(R)})(t^{i}_{j})]((m^{-})^{k}_{l})
&=[(j_{A(R)})(t^{i}_{a})\mathbb{S}_{A(R)}(t^{a}_{j})]((m^{-})^{k}_{l})\\
&=(j_{A(R)})(t^{i}_{a}))((m^{-})^{b}_{l})(\mathbb{S}_{A(R)}(t^{a}_{j}))((m^{-})^{k}_{b})\\
&=\langle(m^{-})^{b}_{l},t^{i}_{a}\rangle\mathbb{S}_{A(R)}(t^{a}_{j})((m^{-})^{k}_{b})\\
&=(R^{-1})^{ib}_{al}((R^{-1})^{t_{2}})^{-1}{}^{ab}_{jk}
=((R^{-1})^{t_{2}})^{il}_{ab}((R^{-1})^{t_{2}})^{-1}{}^{ab}_{jk}\\
&=((R^{-1})^{t_{2}}((R^{-1})^{t_{2}})^{-1})^{il}_{jk}
=I^{il}_{jk}=\delta_{ij}\delta_{kl}\\
&=[(\eta_{(\widetilde{U(R)})^{\ast}}\circ\epsilon_{A(R)})(t^{i}_{j})]((m^{-})^{k}_{l}).
\end{array}
\right.\label{weak4}
\end{gather}
In view of equalities \eqref{weak1}--\eqref{weak4},
we obtain
\begin{equation}\label{inverse1}
\mathbb{S}_{A(R)}\ast j_{A(R)}=j_{A(R)}\ast \mathbb{S}_{A(R)}=\eta_{(\widetilde{U(R)})^{\ast}}\circ\epsilon_{A(R)}.
\end{equation}
We obtain the following equalities in a similar way.

$
\left.
\begin{array}{rl}
[(\mathbb{S}_{\widetilde{U(R)}}\ast j_{\widetilde{U(R)}})((m^{+})^{k}_{l})](t^{i}_{j})
&=(\mathbb{S}_{\widetilde{U(R)}}((m^{+})^{b}_{l})(t^{i}_{a}))(j_{\widetilde{U(R)}}((m^{+})^{k}_{b}))(t^{a}_{j})\\
&=(R^{t_{2}})^{-1}{}^{ba}_{li}R^{t_{2}}{}^{kj}_{ba}=I^{kj}_{li}=\delta_{ij}\delta_{kl}\\
&=[(\eta_{(A(R))^{\ast}}\circ\epsilon_{\widetilde{U(R)}})((m^{+})^{k}_{l})](t^{i}_{j}).
\end{array}
\right.$

$
\left.
\begin{array}{rl}
[(j_{\widetilde{U(R)}}\ast \mathbb{S}_{\widetilde{U(R)}})((m^{+})^{k}_{l})](t^{i}_{j})
&=(j_{\widetilde{U(R)}})((m^{+})^{b}_{l}))(t^{i}_{a})(\mathbb{S}_{\widetilde{U(R)}}((m^{+})^{k}_{b}))(t^{a}_{j})\\
&=(R^{t_{2}})^{ba}_{li}(R^{t_{2}})^{-1}{}^{kj}_{ba}
=I^{kj}_{li}=\delta_{ij}\delta_{kl}\\
&=[(\eta_{(A(R))^{\ast}}\circ\epsilon_{\widetilde{U(R)}})((m^{+})^{k}_{l})](t^{i}_{j}).
\end{array}
\right.$

$
\left.
\begin{array}{rl}
[(\mathbb{S}_{\widetilde{U(R)}}\ast j_{\widetilde{U(R)}})((m^{-})^{k}_{l})](t^{i}_{j})
&=(\mathbb{S}_{\widetilde{U(R)}}((m^{-})^{b}_{l}))(t^{i}_{a})(j_{\widetilde{U(R)}}((m^{-})^{k}_{b})(t^{a}_{j})\\
&=((R^{-1})^{t_{2}})^{-1}{}^{il}_{ab}(R^{-1})^{t_{2}}{}^{ab}_{jk}
=I^{il}_{jk}=\delta_{ij}\delta_{kl}\\
&=[(\eta_{(A(R))^{\ast}}\circ\epsilon_{\widetilde{U(R)}})((m^{-})^{k}_{l})](t^{i}_{j}).
\end{array}
\right.$

$
\left.
\begin{array}{rl}
[(j_{\widetilde{U(R)}}\ast \mathbb{S}_{\widetilde{U(R)}})((m^{-})^{k}_{l})](t^{i}_{j})
&=(j_{\widetilde{U(R)}})((m^{-})^{b}_{l}))(t^{i}_{a})(\mathbb{S}_{\widetilde{U(R)}}((m^{-})^{k}_{b}))(t^{a}_{j})\\
&=((R^{-1})^{t_{2}})^{il}_{ab}((R^{-1})^{t_{2}})^{-1}{}^{ab}_{jk}
=I^{il}_{jk}=\delta_{ij}\delta_{kl}\\
&=[(\eta_{(A(R))^{\ast}}\circ\epsilon_{\widetilde{U(R)}})((m^{-})^{k}_{l})](t^{i}_{j}).
\end{array}
\right.
$
Then we get
\begin{equation}\label{inverse2}
\mathbb{S}_{\widetilde{U(R)}}\ast j_{\widetilde{U(R)}}=j_{\widetilde{U(R)}}\ast \mathbb{S}_{\widetilde{U(R)}}=\eta_{(A(R))^{\ast}}\circ\epsilon_{\widetilde{U(R)}}.
\end{equation}
According to relations \eqref{inverse1} and \eqref{inverse2}, we
prove that $\mathbb{S}_{\widetilde{U(R)}}, \mathbb{S}_{A(R)}$
defined by \eqref{weak} in the Proposition are actually the weak
antipodes of the dual pairing $(\widetilde{U(R)},A(R))$.
\end{proof}
Now, we are only concerned with the bialgebra $\widetilde{U(R)}$ generated by the upper (lower) triangular matrices $m^{+}\ (m^{-})$ for
any invertible upper triangular matrix $n^{2}\times n^{2}$ $R$-matrix $R$. Starting from
$\widetilde{U(R)}$, we give the following
\begin{definition}
Define $\widehat{U(R)}$ as the quotient algebra of
$\widetilde{U(R)}$ modulo the biideal $D$ generated by
$$
(m^{+})^{i}_{i}(m^{-})^{i}_{i}-1, \quad (m^{-})^{i}_{i}(m^{+})^{i}_{i}-1,\qquad
i=1,2,\cdots n.
$$
The generators of $\widehat{U(R)}$ are still denoted by $m^{\pm}$ without confusion.
\end{definition}
Specially,
for the quotient algebra $\widehat{U(R)}$,
we have
\begin{theorem}\label{Hopf}
$\widehat{U(R)}$ obviously satisfies the following relations:
\begin{gather}
Rm^{\pm}_{1}m^{\pm}_{2}=m^{\pm}_{2}m^{\pm}_{1}R,\quad
Rm_{1}^{+}m_{2}^{-}=m_{2}^{-}m_{1}^{+}R,\label{new1}\\
(m^{+})^{i}_{i}(m^{-})^{i}_{i}=(m^{-})^{i}_{i}(m^{+})^{i}_{i}=1,\quad i=1,2,\cdots, n.\label{new2}
\end{gather}
Moreover,
$\widehat{U(R)}$ is a Hopf algebra with the comultiplication $\Delta$,
counit $\epsilon$,
and the antipode $S$ determined by
\begin{equation}
\Delta((m^{\pm})^{i}_{j})=(m^{\pm})_{j}^{a}\otimes (m^{\pm})_{a}^{i},\quad
\epsilon((m^{\pm})^{i}_{j})=\delta_{ij},\quad
S((m^{\pm})^{t})=((m^{\pm})^{t})^{-1}.
\end{equation}
\end{theorem}
In order to prove the above theorem, we give the following Lemma,
which can be checked easily.
\begin{lemma}\label{lemcomm}
Suppose $R$ is any invertible $R$-matrix,
$A, B$ are matrices with non-commutative entries.
Then we have the following equivalent relations
\begin{gather}
RA_{1}B_{2}=B_{2}A_{1}R \Longleftrightarrow R^{t}B^{t}_{2}A^{t}_{1}=A^{t}_{1}B^{t}_{2}R^{t},\label{comm1}\\
RA_{2}B_{1}=B_{1}A_{2}R \Longleftrightarrow R^{t}B^{t}_{1}A^{t}_{2}=A^{t}_{2}B^{t}_{1}R^{t}.\label{comm2}
\end{gather}
\end{lemma}
\noindent{\it Proof of Theorem \ref{Hopf}.} The other relations in
Theorem \ref{Hopf} can be proved easily, so we only focus on the
antipode. The lower (upper) triangular matrices $(m^{+})^{t}~
((m^{-})^{t})$ are invertible in virtue of $(m^{+})^{i}_{i}
(m^{-})^{i}_{i}=(m^{-})^{i}_{i}(m^{+})^{i}_{i}=1$ in
$\widehat{U(R)}$, then we denote the corresponding matrices
$[((m^{\pm})^{t})^{-1}]^{t}$ by $\widehat{m^{\pm}}$, and the entry
located at row $i$ and column $j$ by $(\widehat{m^{\pm}})^{i}_{j}$.
We build the following $\mathbb{C}[q,q^{-1}]$-linear map
\begin{equation}\label{hom}
\left.
\begin{array}{rll}
S: \widehat{U(R)} &\longrightarrow &\widehat{U(R)}\\
(m^{\pm})^{i}_{j}&\longrightarrow&(\widehat{m^{\pm}})^{i}_{j}: \quad S(m^{\pm})=\widehat{m^{\pm}}.
\end{array}
\right.
\end{equation}
According to Lemma \ref{lemcomm},
we have
\begin{equation}\label{neww1}
\left.
\begin{array}{rl}
&Rm^{\pm}_{1}m^{\pm}_{2}=m^{\pm}_{2}m^{\pm}_{1}R\\
& \Longleftrightarrow
R^{t}(m^{\pm})^{t}_{2}(m^{\pm})^{t}_{1}=(m^{\pm})^{t}_{1}(m^{\pm})^{t}_{2}R^{t}\\
&\Longleftrightarrow
((m^{\pm})^{t})^{-1}_{2}((m^{\pm})^{t})^{-1}_{1}R^{t}=R^{t}((m^{\pm})^{t})^{-1}_{1}((m^{\pm})^{t})^{-1}_{2}\\
&\Longleftrightarrow
R[((m^{\pm})^{t})^{-1}]^{t}_{2}[((m^{\pm})^{t})^{-1}]^{t}_{1}
=[((m^{\pm})^{t})^{-1}]^{t}_{1}[((m^{\pm})^{t})^{-1}]^{t}_{2}R\\
&\Longleftrightarrow
R\widehat{m^{\pm}}_{2}\widehat{m^{\pm}}_{1}=\widehat{m^{\pm}}_{1}\widehat{m^{\pm}}_{2}R.
\end{array}
\right.
\end{equation}
\begin{equation}\label{neww2}
\left.
\begin{array}{rl}
&Rm^{+}_{1}m^{-}_{2}=m^{-}_{2}m^{+}_{1}R\\
& \Longleftrightarrow
R^{t}(m^{-})^{t}_{2}(m^{+})^{t}_{1}=(m^{+})^{t}_{1}(m^{-})^{t}_{2}R^{t}\\
&\Longleftrightarrow
((m^{-})^{t})^{-1}_{2}((m^{+})^{t})^{-1}_{1}R^{t}=R^{t}((m^{+})^{t})^{-1}_{1}((m^{-})^{t})^{-1}_{2}\\
&\Longleftrightarrow
R[((m^{-})^{t})^{-1}]^{t}_{2}[((m^{+})^{t})^{-1}]^{t}_{1}
=[((m^{+})^{t})^{-1}]^{t}_{1}[((m^{-})^{t})^{-1}]^{t}_{2}R\\
&\Longleftrightarrow
R\widehat{m^{-}}_{2}\widehat{m^{+}}_{1}=\widehat{m^{+}}_{1}\widehat{m^{-}}_{2}R.
\end{array}
\right.
\end{equation}
We obtain the diagonal entries $(\widehat{m^{\pm}})^{i}_{i}=(m^{\mp})^{i}_{i}$ in view of the definition of $\widehat{m^{\pm}}$,
then \begin{equation}\label{neww3}
(\widehat{m^{+}})^{i}_{i}(\widehat{m^{-}})^{i}_{i}=(\widehat{m^{-}})^{i}_{i}(\widehat{m^{+}})^{i}_{i}=1.
\end{equation}
So the map $S$ is the anti-algebra automorphism of $\widehat{U(R)}$ in virtue of \eqref{neww1}--\eqref{neww3}.
By the equalities $(m^{\pm})^{t}((m^{\pm})^{t})^{-1}=((m^{\pm})^{t})^{-1}(m^{\pm})^{t}=I$,
we obtain the following equivalent relations
$$
\left.
\begin{array}{rl}
&((m^{\pm})^{t})^{i}_{a}[((m^{\pm})^{t})^{-1}]^{a}_{j}=[((m^{\pm})^{t})^{-1}]^{i}_{a}((m^{\pm})^{t})^{a}_{j}=\delta_{ij}\\
&\Longleftrightarrow
(m^{\pm})^{a}_{i}(\widehat{m^{+}})^{j}_{a}=(\widehat{m^{+}})^{a}_{i}(m^{\pm})^{j}_{a}=\delta_{ij}\\
&\Longleftrightarrow
(m^{\pm})^{a}_{i}(Sm^{\pm})^{j}_{a}=(Sm^{\pm})^{a}_{i}(m^{\pm})^{j}_{a}=\delta_{ij}\\
&\Longleftrightarrow
[m\circ(\textrm{id}\otimes S)\circ\Delta]((m^{\pm})^{j}_{i})=
[m\circ(S\otimes \textrm{id})\circ\Delta]((m^{\pm})^{j}_{i})=\delta_{ij}=(\eta\circ\epsilon)((m^{\pm})^{j}_{i}).
\end{array}
\right.
$$
Obviously, we obtain $m\circ(\textrm{id}\otimes
S)\circ\Delta=m\circ(S\otimes
\textrm{id})\circ\Delta=\eta\circ\epsilon.$ So we prove that the map
$S$ is the antipode of $\widehat{U(R)}$.
\hfill $\Box$

Let $T_{V}$ be an irreducible representation of $U_{h}(\mathfrak{g})$,
suppose that there exists a basis such that
the corresponding $R$-matrix $R_{VV}$ is upper triangular,
and $L$-functionals $L^{\pm}$ are upper (lower) triangular.
Then from the structures of bialgebras $U(\text{L}^{\pm})$ and $\widetilde{U(R_{VV})}$,
we obtained easily \cite{HH} that the antipode $S$ gives the morphism from the bialgebra  $U(\text{L}^{\pm})$ generated by $\text{L}^{\pm}$ to $\widetilde{U(R_{VV})}$.
On the other hand,
we notice that the equality $l^{+}_{ii}l^{-}_{ii}=l^{-}_{ii}l^{+}_{ii}=1$ can be obtained by the definition of $L$-functionals,
so the above morphism can be induced on Hopf algebra $\widehat{U(R_{VV})}$,
namely,

\begin{center}
\setlength{\unitlength}{1mm}
\begin{picture}(50,30)
\put(5,25){$U(\text{L}^{\pm})$}
\put(25,27){$S$}
\put(15,26){\line(1,0){25}}
\put(39,25){$>$}
\put(40.5,25){$\widetilde{U(R_{VV})}$}
\put(44.2,4.5){\line(0,1){20}}
\put(43,3){$\vee$}
\put(43,5){$\vee$}
\put(40,0){$\widehat{U(R_{VV})}$}
\put(9,23){\vector(3,-2){30}}
\put(20,10){$S$}
\end{picture}
\end{center}

Thence,
we obtain the following
\begin{lemma}\label{m+}
Under the antipode $S$ of $U_{h}(\mathfrak{g})$,
the expression of
generators $m^{\pm}$ for $\widehat{U(R_{VV})}$ can be
obtained by the $L$-functionals of $U_{h}(\mathfrak{g})$ associated
with representation $T_{V}$.
\end{lemma}
From the explicit form \eqref{imp2} of universal $R$-matrix $\mathcal{R}$ for
$U_{h}(\mathfrak{g})$, we know that there exist factors of form
$\textrm{exp}(a_{i}hH_{i})$'s in entries of $m^{\pm}$, and these
rational numbers $a_{i}$'s are not integers in general. When
$R_{VV}$ is one of standard $R$-matrices, $\widehat{U(R_{VV})}$
coincides with the extended Hopf algebra
$U_{q}^{\textrm{ext}}(\mathfrak{g})$ with
$K_{i}^{a_{i}}(=\textrm{exp}(a_{i}hH_{i}))$ adjoined (see
\cite{klim}). In view of this fact, we still use notation
$U_{q}^{\textrm{ext}}(\mathfrak{g})$ instead of
$\widehat{U(R_{VV})}$ for $R_{VV}$ (not necessarily standard).

\subsection{The coquasitriangularity of Hopf algebra $H_{R_{VV}}$}
After introducing Hopf algebra $H_{R}$ in the preceding section, we
will furthermore consider these Hopf algebras for the general
$R$-matrices, not necessarily standard $R$-matrices. As we know, for
the standard $R$-matrices, $H_{R}$'s are the coquasitriangular Hopf
algebras. What about $H_{R}$ for the general $R$-matrices? Firstly,
we give some equalities about $R$-matrix under some conditions.
\begin{proposition}\label{lematrix}
Suppose $R$ is an invertible $R$-matrix,
and the matrices $R^{t_{i}}, (R^{-1})^{t_{i}}, i=1,2$ are also invertible,
then we have the following equalities
\begin{gather}
(R^{t})_{12}(R^{t})_{13}(R^{t})_{23}=(R^{t})_{23}(R^{t})_{13}(R^{t})_{12},\label{R000}\\
R_{12}(R^{t_{2}})^{-1}_{13}(R^{t_{2}})^{-1}_{23}
=(R^{t_{2}})^{-1}_{23}(R^{t_{2}})^{-1}_{13}R_{12},\label{R0}\\
(R^{t_{2}})_{13}(R^{t_{2}})_{23}R_{12}=R_{12}(R^{t_{2}})_{23}(R^{t_{2}})_{13},\label{R1}\\
R_{13}(R^{t_{2}})^{-1}_{12}((R^{-1})^{t_{1}})^{-1}_{23}
=((R^{-1})^{t_{1}})^{-1}_{23}(R^{t_{2}})^{-1}_{12}R_{13},\label{R2}\\
R_{23}((R^{-1})^{t_{1}})^{-1}_{12}((R^{-1})^{t_{1}})^{-1}_{13}
=((R^{-1})^{t_{1}})^{-1}_{13}((R^{-1})^{t_{1}})^{-1}_{12}R_{23},\label{R3}\\
(R^{t})_{23}(R^{t_{2}})^{-1}_{13}(R^{t_{2}})^{-1}_{12}=(R^{t_{2}})^{-1}_{12}(R^{t_{2}})^{-1}_{13}(R^{t})_{23},\label{R4}\\
(R^{t})_{12}((R^{-1})^{t_{1}})^{-1}_{23}((R^{-1})^{t_{1}})^{-1}_{13}
=((R^{-1})^{t_{1}})^{-1}_{13}((R^{-1})^{t_{1}})^{-1}_{23}(R^{t})_{12},\label{R5}\\
(R^{t_{2}})^{-1}_{23}((R^{-1})^{t_{1}})^{-1}_{13}R_{12}
=R_{12}(R^{t_{2}})^{-1}_{13}((R^{-1})^{t_{1}})^{-1}_{23},\label{R6}\\
(R^{t_{2}})^{-1}_{12}((R^{-1})^{t_{1}})^{-1}_{23}R^{-1}_{13}
=R^{-1}_{13}((R^{-1})^{t_{1}})^{-1}_{23}(R^{t_{2}})^{-1}_{12},\label{R7}\\
(R^{t_{2}})^{-1}_{23}((R^{t})_{13}((R^{-1})^{t_{1}})_{12}
=((R^{-1})^{t_{1}})_{12}((R^{t})_{13}(R^{t_{2}})^{-1}_{23}.\label{R00}
\end{gather}
\end{proposition}
\begin{proof}
All equalities follow mainly from the equality $R_{12}R_{13}R_{23}=R_{23}R_{13}R_{12}$ and some skills of matrices.
\end{proof}

\begin{theorem}\label{cotheo}
Suppose $R$ is an invertible $R$-matrix,
the matrices $R^{t_{i}}, (R^{-1})^{t_{i}}, i=1,2$ are also invertible,
if these matrices satisfy the condition \eqref{chap21},
then $H_{R}$ is a coquasitriangular Hopf algebra with its coquasitriangular structure
given by
\begin{gather}
\mathfrak{r}(t^{i}_{j}\otimes t^{k}_{l})=R^{ik}_{jl},\quad
\mathfrak{r}(\tilde{t}^{i}_{j}\otimes \tilde{t}^{k}_{l})=(R^{t})^{ik}_{jl},\label{coquasi}\\
\mathfrak{r}(\tilde{t}^{i}_{j}\otimes t^{k}_{l})=((R^{-1})^{t_{1}})^{ik}_{jl},\quad
\mathfrak{r}(t^{i}_{j}\otimes \tilde{t}^{k}_{l})=((R^{t_{2}})^{-1})^{ik}_{jl}.\label{coquasi1}
\end{gather}
The equivalent statement in matrices is
\begin{gather}
\mathfrak{r}(T_{1}\otimes T_{2})=R_{12},
\quad\mathfrak{r}(\tilde{T}_{1}\otimes \tilde{T}_{2})=(R^{t})_{12},\label{coquasi2}\\
\mathfrak{r}(\tilde{T}_{1}\otimes T_{2})=((R^{-1})^{t_{1}})_{12},\quad
\mathfrak{r}(T_{1}\otimes \tilde{T}_{2})=((R^{t_{2}})^{-1})_{12}.\label{coquasi3}
\end{gather}
\end{theorem}
\begin{proof}
Let $\mathbb{C}\langle t^{i}_{j},\widetilde{t}^{i}_{j}\rangle$ denote the free algebra generated by $1$ and $t^{i}_{j},\widetilde{t}^{i}_{j}$.
We first construct two linear forms $\mathfrak{r}$  and $\bar{\mathfrak{r}}$ on
 $\mathbb{C}\langle t^{i}_{j},\widetilde{t}^{i}_{j}\rangle\otimes\mathbb{C}\langle t^{i}_{j},\widetilde{t}^{i}_{j}\rangle$.
To do so,
we define  $\mathfrak{r}$ and  $\bar{\mathfrak{r}}$ by \eqref{coquasi2}, \eqref{coquasi3}, \eqref{coquasi4} and \eqref{coquasi5} below, respectively.
\begin{gather}
\bar{\mathfrak{r}}(T_{1}\otimes T_{2})=(R_{VV}^{-1})_{12},
\quad\bar{\mathfrak{r}}(\tilde{T}_{1}\otimes \tilde{T}_{2})=(R_{VV}^{t})^{-1}_{12},\label{coquasi4}\\
\bar{\mathfrak{r}}(\tilde{T}_{1}\otimes T_{2})=[((R_{VV}^{-1})^{t_{1}})]^{-1}_{12},\quad
\bar{\mathfrak{r}}(T_{1}\otimes \tilde{T}_{2})=(R_{VV}^{t_{2}})_{12}.\label{coquasi5}
\end{gather}
And we also define
\begin{equation}\label{tip0}
\mathfrak{r}(1\otimes t^{i}_{j})=\mathfrak{r}(1\otimes \tilde{t}^{i}_{j})=\delta_{ij},\quad
\mathfrak{r}(t^{i}_{j}\otimes 1)=\mathfrak{r}(\tilde{t}^{i}_{j}\otimes 1)=\delta_{ij}.
\end{equation}
Since $\mathbb{C}\langle t^{i}_{j},\widetilde{t}^{i}_{j}\rangle$ is
the free algebra with generators $t^{i}_{j},\widetilde{t}^{i}_{j}$
and a coalgebra structure (see $H_{R}$ in Section 2), there is a
unique extension of $\mathfrak{r}$ to a linear form $\mathfrak{r}:
\mathbb{C}\langle
t^{i}_{j},\widetilde{t}^{i}_{j}\rangle\otimes\mathbb{C}\langle
t^{i}_{j},\widetilde{t}^{i}_{j}\rangle \longrightarrow \mathbb{C}$
such that the following property holds
\begin{equation}
\mathfrak{r}(ab\otimes c)=\mathfrak{r}(a\otimes
c_{(1)})\mathfrak{r}(b\otimes c_{(2)}),\quad \mathfrak{r}(a\otimes
bc)=\mathfrak{r}(a_{(1)}\otimes c)\mathfrak{r}(a_{(2)}\otimes
b)\label{co1}.
\end{equation}
Similarly,
$\bar{\mathfrak{r}}$ extends uniquely to a linear form $\bar{\mathfrak{r}}: \mathbb{C}\langle t^{i}_{j},\widetilde{t}^{i}_{j}\rangle\otimes\mathbb{C}\langle t^{i}_{j},\widetilde{t}^{i}_{j}\rangle \longrightarrow \mathbb{C}$
such that \eqref{co1} holds for the coalgebra $\mathbb{C}\langle t^{i}_{j},\widetilde{t}^{i}_{j}\rangle^{\textrm{op}}$.

The crucial step of the proof is to show that $\mathfrak{r}$ and
$\bar{\mathfrak{r}}$ pass to linear forms on $H_{R}\otimes H_{R}$.
Recall that $H_{R}$ is the quotient
of the algebra $\mathbb{C}\langle
t^{i}_{j},\widetilde{t}^{i}_{j}\rangle$ by relations in \eqref{tilde}.
By \eqref{co1} and the definition of $\mathfrak{r}$, we compute
\begin{gather*}
\mathfrak{r}(R_{12}T_{1}T_{2}\otimes T_{3})=R_{12}\mathfrak{r}(T_{1}T_{2}\otimes T_{3})
=R_{12}\mathfrak{r}(T_{1}\otimes T_{3})\mathfrak{r}(T_{2}\otimes T_{3})=R_{12}R_{13}R_{23},\\
\mathfrak{r}(T_{2}T_{1}R_{12}\otimes T_{3})=\mathfrak{r}(T_{2}T_{1}\otimes T_{3})R_{12}
=\mathfrak{r}(T_{2}\otimes T_{3})\mathfrak{r}(T_{1}\otimes T_{3})R_{12}
=R_{23}R_{13}R_{12},
\end{gather*}
so we obtain
\begin{equation}\label{coq1}
\mathfrak{r}((R_{12}T_{1}T_{2}-T_{2}T_{1}R_{12})\otimes T_{3})=0.
\end{equation}
Similarly,
according to \eqref{R0}, \eqref{R5}, \eqref{R000}, \eqref{R7}, \eqref{R4} respectively,
we obtain
\begin{equation}\label{coq2}
\left.
\begin{array}{c}
\mathfrak{r}((R_{12}T_{1}T_{2}-T_{2}T_{1}R_{12})\otimes \tilde{T}_{3})=0,\\\quad
\mathfrak{r}((R^{t}_{12}\tilde{T}_{1}\tilde{T}_{2}-\tilde{T}_{2}\tilde{T}_{1}R^{t}_{12})\otimes T_{3})=0,\quad
\mathfrak{r}((R^{t}_{12}\tilde{T}_{1}\tilde{T}_{2}-\tilde{T}_{2}\tilde{T}_{1}R^{t}_{12})\otimes \tilde{T}_{3})=0,\\
\mathfrak{r}(((R^{t_{2}})^{-1}_{12}T_{1}\tilde{T}_{2}-\tilde{T}_{2}T_{1}(R^{t_{2}})^{-1}_{12})\otimes
T_{3})=0,\quad
\mathfrak{r}(((R^{t_{2}})^{-1}_{12}T_{1}\tilde{T}_{2}-\tilde{T}_{2}T_{1}(R^{t_{2}})^{-1}_{12})\otimes
\tilde{T}_{3})=0.
\end{array}
\right.
\end{equation}
Again, applying \eqref{co1}, and equalities \eqref{R3}, \eqref{R4},
\eqref{R000}, \eqref{R0}, \eqref{R00}, we get that
\begin{equation}\label{coq3}
\left.
\begin{array}{c}
\mathfrak{r}(T_{1}\otimes(R_{23}T_{2}T_{3}-T_{3}T_{2}R_{23}))=0,\quad
\mathfrak{r}(\tilde{T}_{1}\otimes(R_{23}T_{2}T_{3}-T_{3}T_{2}R_{23}))=0,\\
\mathfrak{r}(T_{1}\otimes(R^{t}_{23}\tilde{T}_{2}\tilde{T}_{3}-\tilde{T}_{3}\tilde{T}_{2}R^{t}_{23}))=0,\quad
\mathfrak{r}(\tilde{T}_{1}\otimes(R^{t}_{23}\tilde{T}_{2}\tilde{T}_{3}-\tilde{T}_{3}\tilde{T}_{2}R^{t}_{23}))=0,\\
\mathfrak{r}(T_{1}\otimes((R^{t_{2}})^{-1}_{23}T_{2}\tilde{T}_{3}-\tilde{T}_{3}T_{2}(R^{t_{2}})^{-1}_{23}))=0,\quad
\mathfrak{r}(\tilde{T}_{1}\otimes((R^{t_{2}})^{-1}_{23}T_{2}\tilde{T}_{3}-\tilde{T}_{3}T_{2}(R^{t_{2}})^{-1}_{23}))=0.
\end{array}
\right.
\end{equation}
Furthermore, by \eqref{tip0}, we have $\mathfrak{r}(*\otimes
1)=\mathfrak{r}(1\otimes *)=0,$ where $*$ denotes the relations in
$\eqref{tilde}$. Therefore, $\mathfrak{r}$ induces a well-defined
linear form, again denoted by $\mathfrak{r}$, on $H_{R}\otimes
H_{R}$. Now we repeat the preceding reasoning with $\mathfrak{r}$
replaced by $\bar{\mathfrak{r}}$ and $\mathbb{C}\langle
t^{i}_{j},\widetilde{t}^{i}_{j}\rangle$ by $\mathbb{C}\langle
t^{i}_{j},\widetilde{t}^{i}_{j}\rangle^{\textrm{op}}$. Similarly, we
can obtain a well-defined linear form $\bar{\mathfrak{r}}$ on
$H_{R}\otimes H_{R}$.

Next, we show that two linear forms $\mathfrak{r}$ and $\bar{\mathfrak{r}}$ on $H_{R}\otimes H_{R}$ satisfy the conditions of the Definition \ref{decoquasi}.
By
$\Delta(T)=T\otimes T,
\Delta(\tilde{T})=\tilde{T}\otimes\tilde{T},
$
we obtain that
\begin{gather*}
\Delta(T_{1}\otimes T_{2})=(T_{1}\otimes T_{2})\otimes(T_{1}\otimes T_{2}),\quad
\Delta(\tilde{T}_{1}\otimes T_{2})=(\tilde{T}_{1}\otimes T_{2})\otimes(\tilde{T}_{1}\otimes T_{2}),\\
\Delta(T_{1}\otimes \tilde{T}_{2})=(T_{1}\otimes \tilde{T}_{2})\otimes(T_{1}\otimes \tilde{T}_{2}),\quad
\Delta(\tilde{T}_{1}\otimes \tilde{T}_{2})=(\tilde{T}_{1}\otimes \tilde{T}_{2})\otimes(\tilde{T}_{1}\otimes \tilde{T}_{2}).
\end{gather*}
Then by definition, we can obtain easily that
\begin{equation}\label{tip1}
\mathfrak{r}\ast \bar{\mathfrak{r}}=\bar{\mathfrak{r}}\ast\mathfrak{r}=\epsilon\otimes\epsilon.
\end{equation}
It is obvious that the relations in \eqref{co1} are equivalent to
\begin{equation}\label{tip2}
\mathfrak{r}\circ (\cdot\otimes \text{id})=\mathfrak{r}_{13}\ast\mathfrak{r}_{23},\quad
\mathfrak{r}\circ(\text{id}\otimes\cdot)=\mathfrak{r}_{13}\ast\mathfrak{r}_{12}.
\end{equation}
To prove the relation
$m^{\textrm{op}}=\mathfrak{r}*m*\bar{\mathfrak{r}}$ is equivalent to
prove that
\begin{equation}\label{tip10}
\mathfrak{r}(a_{(1)}\otimes b_{(1)})a_{(2)}b_{(2)}
=\mathfrak{r}(a_{(2)}\otimes b_{(2)})b_{(1)}a_{(1)}
\end{equation}
holds for any $a, b\in H_{R}$.
If equation \eqref{tip10} holds for $a^{\prime}, b^{\prime}, c^{\prime}$,
then by \eqref{co1},
we have
\begin{equation}\label{tip11}
\left.
\begin{array}{rl}
c^{\prime}_{(1)}(a^{\prime}b^{\prime})_{(1)}\mathfrak{r}(a^{\prime}b^{\prime})_{(2)}\otimes c^{\prime}_{(2)})
&=c^{\prime}_{(1)}a^{\prime}_{(1)}b^{\prime}_{(1)}\mathfrak{r}((a^{\prime}_{(2)}\otimes c^{\prime}_{(2)})\mathfrak{r}(b^{\prime}_{(2)}\otimes c^{\prime}_{(3)})\\
&=c^{\prime}_{(1)}a^{\prime}_{(1)}\mathfrak{r}((a^{\prime}_{(2)}\otimes c^{\prime}_{(2)})b^{\prime}_{(1)}\mathfrak{r}(b^{\prime}_{(2)}\otimes c^{\prime}_{(3)})\\
&=\mathfrak{r}(a^{\prime}_{(1)}\otimes c^{\prime}_{(1)})a^{\prime}_{(2)}
c^{\prime}_{(2)}b^{\prime}_{(1)}\mathfrak{r}(b^{\prime}_{(2)}\otimes c^{\prime}_{(3)})\\
&=\mathfrak{r}(a^{\prime}_{(1)}\otimes c^{\prime}_{(1)})a^{\prime}_{(2)}
\mathfrak{r}(b^{\prime}_{(1)}\otimes c^{\prime}_{(2)})b^{\prime}_{(2)}c^{\prime}_{(3)}\\
&=\mathfrak{r}(a^{\prime}_{(1)}\otimes c^{\prime}_{(1)})\mathfrak{r}(b^{\prime}_{(1)}\otimes c^{\prime}_{(2)})
a^{\prime}_{(2)}b^{\prime})_{(2)}c^{\prime}_{(3)}\\
&=\mathfrak{r}(a^{\prime}_{(1)}b^{\prime}_{(1)}\otimes c^{\prime}_{(1)})a^{\prime}_{(2)}b^{\prime})_{(2)}c^{\prime}_{(2)}\\
&=\mathfrak{r}((a^{\prime}b^{\prime})_{(1)}\otimes c^{\prime}_{(1)})(a^{\prime}b^{\prime})_{(2)}c^{\prime}_{(2)}.
\end{array}
\right.
\end{equation}
\begin{equation}\label{tip12}
\left.
\begin{array}{rl}
(b^{\prime}c^{\prime})_{(1)}a^{\prime}_{(1)}\mathfrak{r}(a^{\prime}_{(2)}\otimes(b^{\prime}c^{\prime})_{(2)} )
&=b^{\prime}_{(1)}c^{\prime}_{(1)}a^{\prime}_{(1)}\mathfrak{r}(a^{\prime}_{(2)}\otimes c^{\prime}_{(2)})\mathfrak{r}(a^{\prime}_{(3)}\otimes b^{\prime}_{(2)})\\
&=b^{\prime}_{(1)}\mathfrak{r}((a^{\prime}_{(1)}\otimes c^{\prime}_{(1)})a^{\prime}_{(2)}c^{\prime}_{(2)}\mathfrak{r}(a^{\prime})_{(3)}\otimes b^{\prime}_{(2)})\\
&=\mathfrak{r}(a^{\prime}_{(1)}\otimes c^{\prime}_{(1)})b^{\prime}_{(1)}a^{\prime}_{(2)}\mathfrak{r}(a^{\prime}_{(3)}\otimes b^{\prime}_{(2)})c^{\prime}_{(2)}\\
&=\mathfrak{r}(a^{\prime}_{(1)}\otimes c^{\prime}_{(1)})
\mathfrak{r}(a^{\prime}_{(2)}\otimes b^{\prime}_{(1)})a^{\prime}_{(3)}b^{\prime}_{(2)}c^{\prime}_{(2)}\\
&=\mathfrak{r}(a^{\prime}_{(1)}b^{\prime}_{(1)}\otimes c^{\prime}_{(1)})a^{\prime}_{(2)}b^{\prime}_{(2)}c^{\prime}_{(2)}\\
&=\mathfrak{r}(a^{\prime}_{(1)}\otimes (b^{\prime}c^{\prime})_{(1)})a^{\prime}_{(2)}(b^{\prime}c^{\prime})_{(2)}.
\end{array}
\right.
\end{equation}
Namely,
according to \eqref{tip11} and \eqref{tip12},
we verify that equation \eqref{tip10} holds for $a=a', b=b'c'$ and for $a=a'b', b=c'$ provided that they hold for $a=a', b=b'$,
for $a=a', b=c'$ and for $a=b', b=c'$, respectively.
Thus it suffices to show that equation \eqref{tip10} holds for $1$ and the generators.
By \eqref{tip0},
we obtain
$$\mathfrak{r}(1\otimes t^{i}_{a})t^{a}_{j}=\epsilon(t^{i}_{a})t^{a}_{j}=t^{i}_{j}
=t^{i}_{a}\epsilon(t^{a}_{j}) =\mathfrak{r}(1\otimes
t^{a}_{j})t^{i}_{a},\ \, \mathfrak{r}(1\otimes
\tilde{t}^{i}_{a})\tilde{t}^{a}_{j}=\epsilon(\tilde{t}^{i}_{a})\tilde{t}^{a}_{j}=\tilde{t}^{i}_{j}
=\tilde{t}^{i}_{a}\epsilon(\tilde{t}^{a}_{j}) =\mathfrak{r}(1\otimes
\tilde{t}^{a}_{j})\tilde{t}^{i}_{a}.
$$
Then associated with
$
\mathfrak{r}(1\otimes ab)=\mathfrak{r}(1\otimes b)\mathfrak{r}(1\otimes a)
$,
we have
\begin{equation}\label{tip3}
\mathfrak{r}(1\otimes a_{(1)})a_{(2)}=\mathfrak{r}(1\otimes a_{(2)})a_{(1)}, \quad\text{ for }\ a\in H_{R}.
\end{equation}
Similarly,
we can prove
\begin{equation}\label{tip4}
\mathfrak{r}(a_{(1)}\otimes 1)a_{(2)}=\mathfrak{r}(a_{(2)}\otimes
1)a_{(1)}, \quad\text{ for }\ a\in H_{R}.
\end{equation}\label{comm}
Namely, if $a$ or $b$ is $1$,
equation \eqref{tip10} is proved.
Moreover,
by the definition of $\mathfrak{r}$,
we can obtain that
$$m^{\textrm{op}}\ast\mathfrak{r}(T_{1}\otimes T_{2})
=T_{2}T_{1}\mathfrak{r}(T_{1}\otimes T_{2})=T_{2}T_{1}R_{12},\quad
\mathfrak{r}\ast m(T_{1}\otimes T_{2})=\mathfrak{r}(T_{1}\otimes
T_{2})T_{1}T_{2}=R_{12}T_{1}T_{2},$$ namely, we have the following
equivalent relation
\begin{equation}\label{tip6}
m^{\textrm{op}}\ast\mathfrak{r}(T_{1}\otimes T_{2})
=\mathfrak{r}\ast m(T_{1}\otimes T_{2})
\Longleftrightarrow
R_{12}T_{1}T_{2}=T_{2}T_{1}R_{12}.
\end{equation}
Hence, we can obtain other equivalent relations in a similar way.
\begin{gather}
 m^{\textrm{op}}\ast\mathfrak{r}(T_{1}\otimes \tilde{T}_{2})
=\mathfrak{r}\ast m(T_{1}\otimes \tilde{T}_{2})
\Longleftrightarrow
\tilde{T}_{2}T_{1}(R^{t_{2}})^{-1}=(R^{t_{2}})^{-1}T_{1}\tilde{T}_{2},\label{tip7} \\
m^{\textrm{op}}\ast\mathfrak{r}(\tilde{T}_{1}\otimes T_{2})
=\mathfrak{r}\ast m(\tilde{T}_{1}\otimes T_{2})
\Longleftrightarrow
T_{2}\tilde{T}_{1}(R^{-1})^{t_{1}}=(R^{-1})^{t_{1}}\tilde{T}_{1}T_{2},\label{tip8}\\
m^{\textrm{op}}\ast\mathfrak{r}(\tilde{T}_{1}\otimes \tilde{T}_{2})
=\mathfrak{r}\ast m(\tilde{T}_{1}\otimes \tilde{T}_{2})
\Longleftrightarrow
\tilde{T}_{2}\tilde{T}_{1}R^{t}=R^{t}\tilde{T}_{1}\tilde{T}_{2}.\label{tip9}
\end{gather}
Here,
the relation $(R^{-1})^{t_{1}}\tilde{T}_{1}T_{2}=T_{2}\tilde{T}_{1}(R^{-1})^{t_{1}}$ can be obtained by $T\tilde{T}^{t}=\tilde{T}^{t}T=I$ and Proposition \ref{comm},
namely,
\begin{equation}\label{tip5}
\left.
\begin{array}{l}
R^{t}\tilde{T}_{1}\tilde{T}_{2}=\tilde{T}_{2}\tilde{T}_{1}R^{t}
\Longleftrightarrow
R\tilde{T}^{t}_{2}\tilde{T}^{t}_{1}=\tilde{T}^{t}_{1}\tilde{T}^{t}_{2}R
\Longleftrightarrow
\tilde{T}^{t}_{2}\tilde{T}^{t}_{1}R^{-1}=R^{-1}\tilde{T}^{t}_{1}\tilde{T}^{t}_{2}\\
\Longleftrightarrow
\tilde{T}^{t}_{1}R^{-1}T_{2}=T_{2}R^{-1}\tilde{T}^{t}_{1}
\Longleftrightarrow
(R^{-1})^{t_{1}}\tilde{T}_{1}T_{2}=T_{2}\tilde{T}_{1}(R^{-1})^{t_{1}}.
\end{array}
\right.
\end{equation}
Then equation \eqref{tip10} holds for generators $a$ or $b$ in virtue of \eqref{tip6}--\eqref{tip9}.
Up to now,
with the conditions in the Theorem,
we have proved that
$\mathfrak{r}$ is actually a universal $r$-form of $H_{R}$,
and $H_{R}$ is coquasitriangular.
The proof is complete.
\end{proof}

We find that many conclusions result from some requirements for
$R$-matrix, for instance, $R^{t_{i}}, (R^{-1})^{t_{i}}, i=1,2$ are
invertible, and
$(R^{-1})^{t_{1}}P(R^{t_{2}})^{-1}PK_{0}=\textrm{const}\cdot K_{0}$.
However, it is uneasy to confirm these conditions. To our surprise,
for the $R$-matrix $R_{VV}$ associated to any finite-dimensional
irreducible $U_{q}(\mathfrak{g})$-module, 
we can confirm this fact according to some features of the representation.
\begin{proposition}\label{rvv}
Assume that $V$ is an irreducible finite-dimensional $U_q(\mathfrak
g)$-module,
then the matrices $R_{VV}^{t_{i}},
(R_{VV}^{-1})^{t_{i}}, i=1,2$ are invertible, and they satisfy the
FRT-condition
$$(R_{VV}^{-1})^{t_{1}}P(R_{VV}^{t_{2}})^{-1}PK_{0}=\textrm{const}\cdot
K_{0}.$$
\end{proposition}
\begin{proof}
As well known,
for any irreducible finite-dimensional module with the highest weight $\lambda$,
there exist uniquely crystal bases $(\mathcal{L}(\lambda), \mathcal{B}(\lambda))~(\mathcal{B}(\infty),\mathcal{L}(\infty))$,
which are bases of $U_{q}(\mathfrak g)$-modules at $``q=0"~(\infty)$,
respectively.
Crystal bases have a nice property,
that is,
for any $b, b'\in\mathcal{B}(\lambda)$ and $\alpha$ is a simple root:
$b=\tilde{E}_{\alpha}b' \Longleftrightarrow
b'=\tilde{F}_{\alpha}b$,
where
$\tilde{E}_{\alpha}, \tilde{F}_{\alpha}$
are famous Kashiwara operators.
Moreover,
Kashiwara globalize this notion.
Namely,
with the aid of a crystal base,
Kashiwara constructed a unique base named the global crystal base of any finite-dimensional highest weight irreducible $U_q(\mathfrak g)$-module that has similar properties to crystal bases.
The (global) crystal bases theory for quantum groups was introduced by Kashiwara in
\cite{Ka1},
\cite{Ka2},
\cite{Ka3}.

Based on the global crystal base theory of quantum groups,
for the finite-dimensional irreducible $U_q(\mathfrak g)$-module $V$,
$V$ has a basis $\{ v_{i}\}$ of $V$ with index-raising weights such that the actions of
$T_{V}(E_{i})$'s and $T_{V}(F_{i})$'s raise and descend the indices
of the basis, respectively.
Moreover,
they satisfy
$E_{\alpha}(v_{i})=fv_{k}\Longleftrightarrow F_{\alpha}(v_{k})=gv_{i}$
for some $f, g\in\mathbb{C}[q,q^{-1}]^*$.
In order to analyze the indices of $R$-matrix conveniently,
we substitute the expression \eqref{rmatrix} for $B_{VV}\circ(T_{V}\otimes
T_{V})(\mathfrak{R})(v_{i}\otimes
v_{j})=R_{VV}{}^{ij}_{nm}(v_{n}\otimes v_{m})$,
and also obtain the upper
triangular $R$-matrix $R_{VV}$ by  with respect to the
lexicographic order on the basis of $V^{\otimes 2}$ induced by the
chosen basis $v_i$ of $V$.
Since
$(R_{VV}^{t_{2}})^{ij}_{kl}=R_{VV}{}^{il}_{kj}$,
we have $i\leq k$ and
$j\leq l$ by the defining formula of $R_{VV}$ if
$(R_{VV}^{t_{2}})^{ij}_{kl}\neq 0$. More precisely, by its upper triangularity, we have $(ij)<(kl)$ if
$i<k$; and if $i=k$, we have $l=j$, i.e., $(ij)=(kl)$.
So, we prove that $(ij)\leq(kl)$ if $(R_{VV}^{t_{2}})^{ij}_{kl}\neq 0$. This means $R_{VV}^{t_{2}}$ is also
upper triangular. Furthermore, we obtain
$\textrm{det}(R_{VV}^{t_{2}})=\textrm{det}(R_{VV})$ by
$(R_{VV}^{t_{2}}){}^{ij}_{ij}=R_{VV}{}^{ij}_{ij}$, so
$R_{VV}^{t_{2}}$ is invertible. We can also prove that
$(R^{-1}_{VV})^{t_{2}}$ is invertible owing to
$R_{VV}^{-1}=B_{VV}\circ(T_{V}\otimes T_{V})(S\otimes
\textrm{id})(\mathfrak{R})$. In the same way, we can prove that matrices
$R_{VV}^{t_{1}}$ and $(R_{VV}^{-1})^{t_{1}}$ are invertible.

For the finite-dimensional irreducible $U_{q}(\mathfrak{g})$-module $V$ involved,
according to the equivalent relation
$E_{\alpha}(v_{i})=fv_{k}\Longleftrightarrow F_{\alpha}(v_{k})=gv_{i}$
for some $f, g\in\mathbb{C}[q,q^{-1}]^*$,
we obtain that
$E_{i_{1}}^{r_{1}}E_{i_{2}}^{r_{2}}\cdots E_{i_{s}}^{r_{s}}(v_j)=xv_i
\Longleftrightarrow
F_{i_{s}}^{r_{s}}\cdots F_{i_{2}}^{r_{2}}F_{i_{1}}^{r_{1}}(v_i)=yv_j$
for some $x, y\in\mathbb{C}[q,q^{-1}]^*$.
Thus
$
R_{VV}{}^{ja}_{ai}=\textrm{const}\cdot
\delta_{ij}$
by \eqref{rmatrix},
and
the equality $(R_{VV}^{-1})^{ja}_{ai}=\textrm{const}\cdot
\delta_{ij}$ can be obtained in the similar way,
then $R_{VV}^{t_{2}}{}^{ji}_{aa}=R_{VV}{}^{ja}_{ai}=\textrm{const}\cdot
\delta_{ij}$.
With these,
by the equality $(R^{t_{2}})^{-1}R^{t_{2}}=R^{t_{2}}(R^{t_{2}})^{-1}=I$,
we have
$(R_{VV}^{t_{2}})^{-1}{}^{ji}_{aa}=\textrm{const}\cdot
\delta_{ij}.$
With these analysis, we obtain

$
\left.
\begin{array}{rl}
[(R_{VV}^{-1})^{t_{1}}P(R_{VV}^{t_{2}})^{-1}PK_{0}]^{ij}_{kl}
&=\sum\limits_{m,n,p,q}((R_{VV}^{-1})^{t_{1}})^{ij}_{mn}(P(R_{VV}^{t_{2}})^{-1}P){}^{mn}_{pq}K_{0}{}^{pq}_{kl}\\
&=\sum\limits_{m,n,p,q}((R_{VV}^{-1})^{t_{1}})^{ij}_{mn}(R_{VV}^{t_{2}})^{-1}{}^{nm}_{pq}\delta_{qp}\delta_{kl}\\
&=\sum\limits_{m,n,p}((R_{VV}^{-1})^{t_{1}})^{ij}_{mn}(R_{VV}^{t_{2}})^{-1}{}^{nm}_{pp}\delta_{kl}\\
&=\sum\limits_{m,n,p}((R_{VV}^{-1})^{t_{1}})^{ij}_{mn}(R_{VV}^{t_{2}})^{-1}{}^{nm}_{pp}\delta_{mn}\delta_{kl}\\
&=\sum\limits_{m,p}((R_{VV}^{-1})^{t_{1}})^{ij}_{mm}(R_{VV}^{t_{2}})^{-1}{}^{mm}_{pp}\delta_{kl}\\
&=\sum\limits_{m,p}(R_{VV}^{-1})^{mj}_{im}(R_{VV}^{t_{2}})^{-1}{}^{mm}_{pp}\delta_{kl}\\
&=\sum\limits_{m,p}(R_{VV}^{-1})^{mj}_{jm}(R_{VV}^{t_{2}})^{-1}{}^{mm}_{pp}\delta_{ij}\delta_{kl}
=\textrm{const}\cdot K_{0}.
\end{array}
\right.
$

The proof is complete.
\end{proof}

By Theorem 3.2 and Proposition 3.3, we obtain an important
conclusion as desired:
\begin{corollary} Assume that $V$ is an irreducible
finite-dimensional $U_q(\mathfrak g)$-module, $R_{VV}$ is the
corresponding $R$-matrix, then  $H_{R_{VV}}$ is a coquasitriangular
Hopf algebra, so that the categories of left (right)
$H_{R_{VV}}$-comodules ${}^{H_{R_{VV}}}\mathfrak M$ \ (\,$\mathfrak
M^{H_{R_{VV}}}$\,) are braided.
\end{corollary}

\subsection{The weakly quasitriangular dual pairs for general $R$-matrices}
Starting from any irreducible representation $T_{V}$ of
$U_{q}(\mathfrak{g})$ or $U_{h}(\mathfrak{g})$, we have the extended
Hopf algebra $U_{q}^{\textrm{ext}}(\mathfrak{g})$ generated by
$m^{\pm}$ and the coquasitriangular Hopf algebra $H_{R_{VV}}$
corresponding to the $R$-matrix $R_{VV}$. In general, $R_{VV}$ is no
longer regular in the sense of Majid \cite{majid1} so that the
Majid's double-bosonization construction Theorem is invalid for the
irregular $R$-matrix. In order to realize the inductive
constructions of the exceptional quantum groups by
double-bosonization procedure, we have to try to generalize the
Majid's double-bosonization construction Theorem to the irregular
cases. To this end, we consider the dual pair of bialgebras
$(\widetilde{U(R)},H_R)$ for the irregular $R$-matrices, instead of
the original weakly quasitriangular dual pair of bialgebras
$(\widetilde{U(R)}, A(R))$ for the regular $R$-matrices served as
the starting point of the Majid's framework. Moreover, we can
establish a weakly quasitriangular dual pair of Hopf algebras
$(U_{q}^{\textrm{ext}}(\mathfrak{g}), H_{R_{VV}})$ as we desired,
where $U_{q}^{\textrm{ext}}(\mathfrak{g})=\widehat{U(R)}$ in
subsection 3.3.
\begin{theorem}\label{theoWeak}
Let $T_{V}$ be an irreducible representation of $U_{q}(\mathfrak{g})$,
$R_{VV}$ the corresponding $R$-matrix,
then there exists a weakly quasitriangular dual pair between $U_{q}^{\textrm{ext}}(\mathfrak{g})$ and $H_{R_{VV}}$,
given by
\begin{gather}
\langle (m^{+})^{i}_{j},t^{k}_{l}\rangle=R_{VV}{}^{ik}_{jl},\quad
\langle (m^{-})^{i}_{j},t^{k}_{l}\rangle=(R^{-1}_{VV}){}^{ki}_{lj},\label{tlidew1}\\
\langle (m^{+})^{i}_{j},\tilde{t}^{k}_{l}\rangle=(R_{VV}^{t_{2}})^{-1}{}^{ik}_{jl},\quad
\langle (m^{-})^{i}_{j},\tilde{t}^{k}_{l}\rangle=[(R^{-1}_{VV})^{t_{1}}]^{-1}{}^{ki}_{lj}.\label{tlidew2}
\end{gather}
The convolution-invertible algebra\,/\,anti-coalgebra maps $\R,\bar{\R}$ in
$\textrm{Hom}(H_{\lambda R},U_{q}^{\textrm{ext}}(\mathfrak{g}))$ are
$$
\R(t^{i}_{j})=(m^{+})^{i}_{j},\quad
\R(\tilde{t}^{i}_{j})=(m^{+})^{-1}{}^{j}_{i};\quad
\bar{\R}(t^{i}_{j})=(m^{-})^{i}_{j},\quad
\bar{\R}(\tilde{t}^{i}_{j})=(m^{-})^{-1}{}^{j}_{i}.
$$
Moreover,
the following relations
\begin{equation}\label{three1}
\langle\bar{\R}(a),b\rangle=\langle\R^{-1}(b),a\rangle,\quad
\partial^{R}h=\R \ast(\partial^{L}h)\ast\R^{-1},\quad
\partial^{R}h=\bar{\R}\ast(\partial^{L}h)\ast\bar{\R}^{-1}
\end{equation}
hold for any $a\in H_{R_{VV}},\,h\in
U_{q}^{\textrm{ext}}(\mathfrak{g})$.
\end{theorem}
In order to prove the above theorem,
we give the following Lemma \ref{pat},
which can be proved in virtue of the definition of weakly quasitriangular dual pair.
\begin{lemma}\label{pat}
$(H,A)$ is a weakly quasitriangular dual pair of Hopf algebras,
$\R,\bar{\R}$ are the corresponding convolution-invertible algebra\,/\,anti-coalgebra maps,
and $\partial^{L}h,\,
\partial^{R}h$
are left and right ``differentiation operators" for any fixed $h\in
H$. Then for any $a, b, c\in A,\,g, h\in H$, we have
\begin{gather}
\partial^{L}(gh)=\partial^{L}g\ast\partial^{L}h,\quad
\partial^{R}(gh)=\partial^{R}g\ast\partial^{R}h,\label{patri}\\
\langle\bar{\R}(ab),c\rangle=\langle\R^{-1}(c),ab\rangle,\quad
\langle\bar{\R}(a),bc\rangle=\langle\R^{-1}(bc),a\rangle,\label{patri1}\\
\partial^{R}(gh)=\R \ast\partial^{L}(gh)\ast\R^{-1},\quad
\partial^{R}(gh)=\bar{\R}\ast\partial^{L}(gh)\ast\bar{\R}^{-1}.\label{patri2}
\end{gather}
\end{lemma}

\noindent {\it Proof of Theorem \ref{theoWeak}.} \ For convenience, we denote
$R_{VV}$ by $R$ in the proof, and also give the equivalent
expressions in matrix form for those relations in Theorem
\ref{theoWeak}:
\begin{gather}
\langle m_{1}^{+}, T_{2}\rangle=(R_{VV})_{12},\quad \langle
m_{1}^{-},T_{2}\rangle=(R_{VV}^{-1})_{21}=P\circ
(R_{VV}^{-1})_{12}\circ P,
\\
\langle
m_{1}^{+},\tilde{T}_{2}\rangle=((R_{VV}^{t_{2}})^{-1})_{12},\quad
\langle
m_{1}^{-},\tilde{T}_{2}\rangle=(R_{VV}^{-1})^{t_{1}})^{-1}_{21}=P\circ(R_{VV}^{-1})^{t_{1}})^{-1}\circ
P,\\
\R(t)=m^{+},\quad \R(\tilde{t})=((m^{+})^{-1})^{t};\quad
\bar{\R}(t)=m^{-},\quad
\bar{\R}(\tilde{t})=((m^{-})^{-1})^{t}.\label{eqv00}
\end{gather}

{\it Step 1. The dual pairing.}

$\mathbb{C}\langle (m^{+})^{i}_{j},(m^{-})^{i}_{j}\rangle$ denotes the free algebra generated by $m^{\pm}$.
Obviously,
\eqref{tlidew1} and \eqref{tlidew2} define a dual pairing between $\mathbb{C}\langle t^{i}_{j},\widetilde{t}^{i}_{j}\rangle$ and
$\mathbb{C}\langle (m^{+})^{i}_{j},(m^{-})^{i}_{j}\rangle$.
By the QYBE,
we have
\begin{equation*}
\left.
\begin{array}{l}
R_{12}R_{13}R_{23}
=R_{23}R_{13}R_{12}
\Longleftrightarrow
R_{13}R_{12}R^{-1}_{23}
=R^{-1}_{23}R_{12}R_{13}\\
\Longleftrightarrow
R_{23}R_{12}^{-1}R_{13}^{-1}
=R_{13}^{-1}R_{12}^{-1}R_{23}
\Longleftrightarrow
R_{13}^{-1}R_{23}^{-1}R_{12}=R_{12}R_{23}^{-1}R_{13}^{-1}.
\end{array}
\right.
\end{equation*}
Associated with \eqref{R0}, \eqref{R2} and \eqref{R3} in Proposition
\ref{lematrix}, we obtain
\begin{equation}\label{matrix1}
\left.
\begin{array}{c}
\langle R_{12}m^{+}_{1}m^{+}_{2},T_{3}\rangle
=R_{12}R_{13}R_{23}
=R_{23}R_{13}R_{12}
=\langle m^{+}_{2}m^{+}_{1}R_{12},T_{3}\rangle,\\
\left.
\begin{array}{rl}
\langle R_{12}m^{+}_{1}m^{+}_{2},\tilde{T}_{3}\rangle
&=R_{12}(R^{t_{2}})^{-1}_{13}(R^{t_{2}})^{-1}_{23}
=(R^{t_{2}})^{-1}_{23}(R^{t_{2}})^{-1}_{13}R_{12}\\
&=\langle m^{+}_{2}m^{+}_{1}R_{12},T_{3}\rangle.
\end{array}
\right.
\end{array}
\right.
\end{equation}
\begin{equation}\label{matrix2}
\left.
\begin{array}{c}
\langle R_{23}m^{-}_{2}m^{-}_{3},T_{1}\rangle
=R_{23}R_{12}^{-1}R_{13}^{-1}
=R_{13}^{-1}R_{12}^{-1}R_{23}
=\langle m^{-}_{3}m^{-}_{2}R_{23},T_{1}\rangle,\\
\left.
\begin{array}{rl}
\langle R_{23}m^{-}_{2}m^{-}_{3},\tilde{T}_{1}\rangle
&=R_{23}((R^{-1})^{t_{1}})^{-1}_{12}((R^{-1})^{t_{1}})^{-1}_{13}\\
&=((R^{-1})^{t_{1}})^{-1}_{13}((R^{-1})^{t_{1}})^{-1}_{12}R_{23}
=\langle m^{-}_{3}m^{-}_{2}R_{23},\tilde{T}_{1}\rangle.
\end{array}
\right.
\end{array}
\right.
\end{equation}
\begin{equation}\label{matrix3}
\left.
\begin{array}{c}
\langle R_{13}m^{+}_{1}m^{-}_{3},T_{2}\rangle
=R_{13}R_{12}R^{-1}_{23}
=R^{-1}_{23}R_{12}R_{13}
=\langle m^{-}_{3}m^{+}_{1}R_{13},T_{2}\rangle,\\
\left.
\begin{array}{rl}
\langle R_{13}m^{+}_{1}m^{-}_{3},\tilde{T}_{2}\rangle
&=R_{13}(R^{t_{2}})^{-1}_{12}((R^{-1})^{t_{1}})^{-1}_{23}
=((R^{-1})^{t_{1}})^{-1}_{23}(R^{t_{2}})^{-1}_{12}R_{13}\\
&=\langle m^{-}_{3}m^{+}_{1}R_{13},\tilde{T}_{2}\rangle.
\end{array}
\right.
\end{array}
\right.
\end{equation}
So the dual pairing between $\mathbb{C}\langle
t^{i}_{j},\widetilde{t}^{i}_{j}\rangle$ and $\mathbb{C}\langle
(m^{+})^{i}_{j},(m^{-})^{i}_{j}\rangle$ can passe to the dual
pairing of $U_{q}^{\textrm{ext}}(\mathfrak{g})$ and
$\mathbb{C}\langle t^{i}_{j},\widetilde{t}^{i}_{j}\rangle$. On the
other hand, in view of
$\Delta^{\textrm{cop}}(m^{\pm})=m^{\pm}\otimes m^{\pm}$ and the
equalities \eqref{R4}, \eqref{R5}, \eqref{R6}, \eqref{R7} in
Proposition \ref{lematrix}, we obtain
\begin{equation}\label{matrix4}
\left.
\begin{array}{c}
\langle m^{+}_{1},R_{23}T_{2}T_{3}\rangle
=R_{23}R_{13}R_{12}
=R_{12}R_{13}R_{23}
=\langle m^{+}_{1},T_{3}T_{2}R_{23}\rangle,\\
\langle m^{-}_{3},R_{12}T_{1}T_{2}\rangle
R_{13}^{-1}R_{23}^{-1}R_{12}
=R_{12}R_{23}^{-1}R_{13}^{-1}
=\langle m^{-}_{3},T_{2}T_{1}R_{12}\rangle.
\end{array}
\right.
\end{equation}
\begin{equation}\label{matrix5}
\left.
\begin{array}{c}
\left.
\begin{array}{rl}
\langle m^{+}_{1},R^{t}_{23}\tilde{T}_{2}\tilde{T}_{3}\rangle
&=(R^{t})_{23}(R^{t_{2}})^{-1}_{13}(R^{t_{2}})^{-1}_{12}\\
&=(R^{t_{2}})^{-1}_{12}(R^{t_{2}})^{-1}_{13}(R^{t})_{23}
=\langle m^{+}_{1},\tilde{T}_{3}\tilde{T}_{2}R^{t}_{23}\rangle,
\end{array}
\right.\\
\left.
\begin{array}{rl}
\langle  m^{-}_{3},R^{t}_{12}\tilde{T}_{1}\tilde{T}_{2}\rangle
&=(R^{t})_{12}((R^{-1})^{t_{1}})^{-1}_{23}((R^{-1})^{t_{1}})^{-1}_{13}\\
&=((R^{-1})^{t_{1}})^{-1}_{13}((R^{-1})^{t_{1}})^{-1}_{23}(R^{t})_{12}
=\langle  m^{-}_{3},\tilde{T}_{2}\tilde{T}_{1}R^{t}_{12}\rangle.
\end{array}
\right.
\end{array}
\right.
\end{equation}
\begin{equation}\label{matrix6}
\left.
\begin{array}{c}
\left.
\begin{array}{rl}
\langle m^{+}_{1},(R^{t_{2}})^{-1}_{23}T_{2}\tilde{T}_{3}\rangle
&=(R^{t_{2}})^{-1}_{23}((R^{-1})^{t_{1}})^{-1}_{13}R_{12}\\
&=R_{12}(R^{t_{2}})^{-1}_{13}((R^{-1})^{t_{1}})^{-1}_{23}
=\langle m^{+}_{1},\tilde{T}_{3}T_{2}(R^{t_{2}})^{-1}_{23}\rangle,
\end{array}
\right.\\
\left.
\begin{array}{rl}
\langle m^{-}_{3},(R^{t_{2}})^{-1}_{12}T_{1}\tilde{T}_{2}\rangle
&=(R^{t_{2}})^{-1}_{12}((R^{-1})^{t_{1}})^{-1}_{23}R^{-1}_{13}\\
&=R^{-1}_{13}((R^{-1})^{t_{1}})^{-1}_{23}(R^{t_{2}})^{-1}_{12}
=\langle  m^{-}_{3},\tilde{T}_{2}T_{1}(R^{t_{2}})^{-1}_{23}\rangle.
\end{array}
\right.
\end{array}
\right.
\end{equation}
Then we prove that \eqref{tlidew1} and \eqref{tlidew2} define a dual pairing of $U_{q}^{\textrm{ext}}(\mathfrak{g})$ and $H_{R_{VV}}$.

{\it Step 2. $\R,\; \bar{\R}$ are algebra\,/\,anti-coalgebra maps.}

We set $R=A_{i}\otimes B_{i}$,
and by Lemma \ref{comm},
we obtain that
\begin{equation}\label{eqv0}
\left.
\begin{array}{rl}
Rm^{\pm}_{1}m^{\pm}_{2}=m^{\pm}_{2}m^{\pm}_{1}R
&\Longleftrightarrow
R(m^{\pm})^{-1}_{2}(m^{\pm})^{-1}_{1}=(m^{\pm})^{-1}_{1}(m^{\pm})^{-1}_{2}R\\
&\Longleftrightarrow
R^{t}((m^{\pm})^{-1})^{t}_{1}((m^{\pm})^{-1})^{t}_{2}=((m^{\pm})^{-1})^{t}_{2}((m^{\pm})^{-1})^{t}_{1}R^{t}.
\end{array}
\right.
\end{equation}
\begin{equation}\label{eqv1}
\left.
\begin{array}{l}
Rm^{\pm}_{1}m^{\pm}_{2}=m^{\pm}_{2}m^{\pm}_{1}R
\Longleftrightarrow
m^{\pm}_{1}R(m^{\pm})^{-1}_{2}=(m^{\pm})^{-1}_{2}R(m^{\pm})_{1}\\
\Longleftrightarrow
(m^{\pm}\otimes I)(A_{i}\otimes B_{i})(I\otimes (m^{\pm})^{-1})
=(I\otimes (m^{\pm})^{-1})(A_{i}\otimes B_{i})(m^{\pm}\otimes I)\\
\Longleftrightarrow
m^{\pm}A_{i}\otimes B_{i}(m^{\pm})^{-1}=A_{i}m^{\pm}\otimes (m^{\pm})^{-1}B_{i}\\
\Longleftrightarrow
m^{\pm}A_{i}\otimes [B_{i}(m^{\pm})^{-1}]^{t}=A_{i}m^{\pm}\otimes [(m^{\pm})^{-1}B_{i}]^{t}\\
\Longleftrightarrow
m^{\pm}A_{i}\otimes ((m^{\pm})^{-1})^{t}B_{i}^{t}=A_{i}m^{\pm}\otimes B_{i}^{t}((m^{\pm})^{-1})^{t}\\
\Longleftrightarrow
(m^{\pm}\otimes I)(I\otimes ((m^{\pm})^{-1})^{t})(A_{i}\otimes B_{i}^{t})
=(A_{i}\otimes B_{i}^{t})(I\otimes ((m^{\pm})^{-1})^{t})(m^{\pm}\otimes I)\\
\Longleftrightarrow
m^{\pm}_{1}((m^{\pm})^{-1})^{t}_{2}R^{t_{2}}=R^{t_{2}}((m^{\pm})^{-1})^{t}_{2}m^{\pm}_{1}.
\end{array}
\right.
\end{equation}
Then in virtue of \eqref{eqv0} and \eqref{eqv1},
we obtain that $\R, \bar{\R}$ are algebra homomorphisms.
On the other hand,
since $\Delta$ is an algebra homomorphism,
we have $$
\delta_{jk}\otimes 1=\delta_{jk}\Delta(1)=\Delta(\delta_{jk})=\Delta(I^{j}_{k})
=\Delta(((m^{+})^{-1})^{j}_{i}(m^{+})^{i}_{k})
=\Delta(((m^{+})^{-1})^{j}_{i})((m^{+})^{a}_{k}\otimes(m^{+})^{i}_{a}).$$
Then we obtain that
$
\Delta(((m^{+})^{-1})^{j}_{i})=((m^{+})^{-1})^{j}_{a}\otimes ((m^{+})^{-1})^{a}_{i},
$
so
\begin{gather}
\left.
\begin{array}{rl}
(\R\otimes\R)\Delta(t^{i}_{j})
&=\R(t^{i}_{a})\otimes\R(t^{a}_{j})
=(m^{+})^{i}_{a}\otimes (m^{+})^{a}_{j}\\
&=\Delta^{\textrm{cop}}((m^{+})^{i}_{j})
=\Delta^{\textrm{cop}}(\R(t^{i}_{j})),
\end{array}
\right.
\label{eqv4}
 \\
\left.
\begin{array}{rl}
(\R\otimes\R)\Delta(\tilde{t}^{i}_{j})
&=\R(\tilde{t}^{i}_{a})\otimes\R(\tilde{t}^{a}_{j})
=[((m^{+}))^{-1}]^{a}_{i}\otimes [((m^{+}))^{-1}]^{j}_{a}\\
&=\Delta^{\textrm{cop}}([((m^{+}))^{-1}]^{j}_{i})
=\Delta^{\textrm{cop}}(\R(\tilde{t}^{i}_{j})).
\end{array}
\right.\label{eqv5}
\end{gather}
We prove that $\R$ is an anti-coalgebra map by \eqref{eqv4} and
\eqref{eqv5}. In a similar analysis, we can prove that $\bar{\R}$ is
also an anti-coalgebra map.

{\it Step 3.  $\R, \bar{\R}$ are convolution-invertible.}

Since the upper (lower) triangularity of $m^{+}$ $\ (\,m^{-}\,)$,
the matrices $(m^{+})^{-1}$ $\ (\,(m^{-})^{-1}\,)$ are upper (lower)
triangular, then matrices $((m^{+})^{-1})^{t}$ $\
(((m^{-})^{-1})^{t})$ are lower (upper) triangular, and
$[((m^{+})^{-1})^{t}]^{i}_{i}[((m^{-})^{-1})^{t}]^{i}_{i}=[((m^{-})^{-1})^{t}]^{i}_{i}[((m^{+})^{-1})^{t}]^{i}_{i}=1.$
So the matrices $((m^{\pm})^{-1})^{t}$ are invertible. With these,
we define the following two maps $\R^{-1}, \bar{\R}^{-1}$ in the
convolution algebra
$\textrm{Hom}(H_{R_{VV}},U_{q}^{\textrm{ext}}(\mathfrak{g}))$,
\begin{gather}
\R^{-1}(t^{i}_{j})=((m^{+})^{-1})^{i}_{j},\quad
\R^{-1}(\tilde{t}^{i}_{j})=[((m^{+})^{-1})^{t}]^{-1}{}^{i}_{j};\label{m1}\\
\bar{\R}^{-1}(t^{i}_{j})=((m^{-})^{-1})^{i}_{j},\quad
\bar{\R}^{-1}(\tilde{t}^{i}_{j})=[((m^{-})^{-1})^{t}]^{-1}{}^{i}_{j}.\label{m2}
\end{gather}
By \eqref{m1},
for any $t^{i}_{j}, \tilde{t}^{i}_{j}$,
we have
\begin{gather}
\left.\begin{array}{rl}
(\R\ast\R^{-1})(t^{i}_{j})&=\R(t^{i}_{a})\R^{-1}(t^{a}_{j})
=(m^{+})^{i}_{a}((m^{+})^{-1})^{a}_{j}\\
&=(m^{+}(m^{+})^{-1})^{i}_{j}=I^{i}_{j}
=\delta_{ij}
=\eta\circ\epsilon(t^{i}_{j}),
\end{array}
\right.
\label{e1}
\\
\left.\begin{array}{rl}
(\R^{-1}\ast\R)(t^{i}_{j})&=\R^{-1}(t^{i}_{a})\R(t^{a}_{j})
=((m^{+})^{-1})^{i}_{a}(m^{+})^{a}_{j}\\
&=((m^{+})^{-1}m^{+})^{i}_{j}=I^{i}_{j}
=\delta_{ij}
=\eta\circ\epsilon(t^{i}_{j}),
\end{array}
\right.
\label{e2}
\\
\left.
\begin{array}{rl}
(\R\ast\R^{-1})(\tilde{t}^{i}_{j})&=\R(\tilde{t}^{i}_{a})\R^{-1}(\tilde{t}^{a}_{j})
=[(((m^{+})^{-1}))^{t}]^{i}_{a}[(((m^{+})^{-1})^{t})^{-1}]^{a}_{j}\\
&=[(((m^{+})^{-1}))^{t}(((m^{+})^{-1})^{t})^{-1}]^{i}_{j}=I^{i}_{j}
=\delta_{ij}
=\eta\circ\epsilon(\tilde{t}^{i}_{j}),
\end{array}\right.
\label{e3}
\\
\left.
\begin{array}{rl}
(\R^{-1}\ast\R)(\tilde{t}^{i}_{j})
&=\R^{-1}(\tilde{t}^{i}_{a})\R(\tilde{t}^{a}_{j})
=[(((m^{+})^{-1})^{t})^{-1}]^{i}_{a}[(((m^{+})^{-1}))^{t}]^{a}_{j}\\
&=[(((m^{+})^{-1})^{t})^{-1}(((m^{+})^{-1}))^{t}]^{i}_{j}
=I^{i}_{j}=\delta_{ij}
=\eta\circ\epsilon(\tilde{t}^{i}_{j}).
\end{array}\right.
\label{e4}
\end{gather}
So we obtain that
$\R\ast\R^{-1}=\R^{-1}\ast\R=\eta_{U_{q}^{\textrm{ext}}(\mathfrak{g})}\circ
\epsilon_{H_{R_{VV}}}$ owing to $\eqref{e1}, \eqref{e2}, \eqref{e3},
\eqref{e4}$.
$\bar{\R}\ast\bar{\R}^{-1}=\bar{\R}^{-1}\ast\bar{\R}=\eta_{U_{q}^{\textrm{ext}}(\mathfrak{g})}\circ
\epsilon_{H_{R_{VV}}}$ can be proved in a similar way.

{\it Step 4. The maps $\R, \bar{\R}$ satisfy the relations in \eqref{three1}.}

According to Lemma \ref{pat},
we only need to prove that the relations in \eqref{three1} hold for generators.
In view of
$\R^{-1}\ast \R=\R\ast\R^{-1}=\eta\circ\epsilon$,
then we have
\begin{gather*}
\left.
\begin{array}{rl}
\langle \R(t^{k}_{a})\R^{-1}(t^{a}_{l}),t^{i}_{j}\rangle
&=\langle \R(t^{k}_{a}),t^{i}_{b}\rangle\langle\R^{-1}(t^{a}_{l}),t^{b}_{j}\rangle
=\langle(m^{+})^{k}_{a},t^{i}_{b}\rangle\langle\R^{-1}(t^{a}_{l}),t^{b}_{j}\rangle\\
&=R^{ki}_{ab}\langle\R^{-1}(t^{a}_{l}),t^{b}_{j}\rangle
=\langle\delta_{kl},t^{i}_{j}\rangle
=\delta_{kl}\delta_{ij}\\
&=I^{ki}_{lj}
=\langle \R^{-1}(t^{k}_{a})\R(t^{a}_{l}),t^{i}_{j}\rangle
=\langle \R^{-1}(t^{k}_{a}),t^{i}_{b}\rangle\langle\R(t^{a}_{l}),t^{b}_{j}\rangle\\
&=\langle \R^{-1}(t^{k}_{a}),t^{i}_{b}\rangle\langle(m^{+})^{a}_{l},t^{b}_{j}\rangle
=\langle \R^{-1}(t^{k}_{a}),t^{i}_{b}\rangle R^{ab}_{lj}.
\end{array}
\right.
\end{gather*}
So
$
\langle \R^{-1}(t^{k}_{l}),t^{i}_{j}\rangle=(R^{-1})^{ki}_{lj},
$
then we obtain that
\begin{equation}\label{dual1}
\langle \bar{\R}(t^{i}_{j}),t^{k}_{l}\rangle=\langle (m^{-})^{i}_{j}),t^{k}_{l}\rangle
=(R^{-1})^{ki}_{lj}=\langle \R^{-1}(t^{k}_{l}),t^{i}_{j}\rangle.
\end{equation}
We also obtain the following relations in a similar way
\begin{gather}
\langle\bar{\R}(t^{i}_{j}),\tilde{t}^{k}_{l}\rangle=\langle(m^{-})^{i}_{j}),\tilde{t}^{k}_{l}\rangle=
[((R^{-1})^{t_{1}})^{-1}]^{ki}_{lj}=\langle \R^{-1}(\tilde{t}^{k}_{l}),t^{i}_{j}\rangle,\label{dual2}\\
\langle\bar{\R}(\tilde{t}^{i}_{j}),t^{k}_{l}\rangle=\langle((m^{-})^{-1})^{j}_{i}),t^{k}_{l}\rangle=
R^{kj}_{li}=\langle \R^{-1}(t^{k}_{l}),\tilde{t}^{i}_{j}\rangle,\label{dual3}\\
\langle\bar{\R}(\tilde{t}^{i}_{j}),\tilde{t}^{k}_{l}\rangle=\langle((m^{-})^{-1})^{j}_{i}),\tilde{t}^{k}_{l}\rangle=
(R^{-1})^{lj}_{ki}=\langle \R^{-1}(\tilde{t}^{k}_{l}),\tilde{t}^{i}_{j}\rangle.\label{dual4}
\end{gather}
According to \eqref{dual1}--\eqref{dual4},
we prove that
$\langle \bar{\R}a,b \rangle=\langle\R^{-1}b,a\rangle$ holds for any $a, b\in H_{R_{VV}}$.

Next we will prove the following relations for $\R, \bar{\R}$.
\begin{gather}
\partial^{R}(m^{\pm})^{k}_{l}\ast\R=\R \ast\partial^{L}(m^{\pm})^{k}_{l},\label{dual5}\\
\partial^{R}(m^{\pm})^{k}_{l}\ast\bar{\R}=\bar{\R} \ast\partial^{L}(m^{\pm})^{k}_{l}.\label{dual6}
\end{gather}
With the convolution product and the coalgebra structure,
we obtain
\begin{equation}\label{dual7}
\left.
\begin{array}{l}
(\partial^{R}(m^{\pm})^{k}_{l}\ast\R)(t^{i}_{j})\\=\partial^{R}(m^{\pm})^{k}_{l}(t^{i}_{a})\R(t^{a}_{j})
=(m^{\pm})^{b}_{l}\langle (m^{\pm})^{k}_{b},t^{i}_{a}\rangle\R(t^{a}_{j})\\
=R^{ki}_{ab}(m^{\pm})^{b}_{l}(m^{\pm})^{a}_{j}
=[Rm^{\pm}_{1}m^{\pm}_{2}]^{ki}_{lj}\\
=[m^{\pm}_{2}m^{\pm}_{1}R]^{ki}_{lj}
=(m^{\pm})^{i}_{a}(m^{\pm})^{k}_{b}R^{ba}_{lj}\\
=(m^{\pm})^{i}_{a}\langle (m^{\pm})^{b}_{l},t^{a}_{j}\rangle(m^{\pm})^{k}_{b}
=\R(t^{i}_{a})\partial^{L}(m^{\pm})^{k}_{l}(t^{a}_{j})\\
=(\R \ast\partial^{L}(m^{\pm})^{k}_{l})(t^{i}_{j}).
\end{array}
\right.
\end{equation}
\begin{equation}\label{dual8}
\left.
\begin{array}{l}
(\partial^{R}(m^{+})^{k}_{l}\ast\R)(\tilde{t}^{i}_{j})\\
=\partial^{R}(m^{+})^{k}_{l}(\tilde{t}^{i}_{a})\R(\tilde{t}^{a}_{j})
=(m^{+})^{b}_{l}\langle (m^{+})^{k}_{b},\tilde{t}^{i}_{a}\rangle\R(\tilde{t}^{a}_{j})\\
=((R^{t_{2}})^{-1})^{ki}_{ba}(m^{+})^{b}_{l}[((m^{+})^{-1})^{t}]^{a}_{j}
=[(R^{t_{2}})^{-1}m^{+}_{1}((m^{+})^{-1})^{t}_{2}]^{ki}_{lj}\\
=[((m^{+})^{-1})^{t}_{2}m^{+}_{1}(R^{t_{2}})^{-1}]^{ki}_{lj}
=[((m^{+})^{-1})^{t}]^{i}_{a}(m^{+})^{k}_{b}((R^{t_{2}})^{-1})^{ba}_{lj}\quad (\mbox{by}\  \eqref{eqv1})\\
=[((m^{+})^{-1})^{t}]^{i}_{a}\langle(m^{+})^{b}_{l},\tilde{t}^{a}_{j}\rangle(m^{+})^{k}_{b}
=\R(\tilde{t}^{i}_{a})\partial^{L}(m^{+})^{k}_{l}(\tilde{t}^{a}_{j})\\
=(\R \ast\partial^{L}(m^{\pm})^{k}_{l})(\tilde{t}^{i}_{j}).
\end{array}
\right.
\end{equation}
On the other hand,
we set $R^{-1}=C_{i}\otimes D_{i}$,
then according to $Rm^{+}_{1}m^{-}_{2}=m^{-}_{2}m^{+}_{1}R$,
we have
\begin{equation}\label{eqvm}
\left.
\begin{array}{l}
Rm^{+}_{1}m^{-}_{2}=m^{-}_{2}m^{+}_{1}R\Longleftrightarrow m^{+}_{1}m^{-}_{2}R^{-1}=R^{-1}m^{-}_{2}m^{+}_{1}\\
\Longleftrightarrow
m^{-}_{2}R^{-1}(m^{+})^{-1}_{1}
=(m^{+})^{-1}_{1}R^{-1}m^{-}_{2}\\
\Longleftrightarrow
(I\otimes m^{-})(C_{i}\otimes D_{i})((m^{+})^{-1}\otimes I)=((m^{+})^{-1}\otimes I)(C_{i}\otimes D_{i})(I\otimes m^{-})\\
\Longleftrightarrow
C_{i}(m^{+})^{-1}\otimes m^{-}D_{i}=(m^{+})^{-1}C_{i}\otimes D_{i}m^{-}\\
\Longleftrightarrow
((m^{+})^{-1})^{t}C_{i}^{t}\otimes m^{-}D_{i}=C_{i}^{t}((m^{+})^{-1})^{t}\otimes D_{i}m^{-}\\
\Longleftrightarrow
(I\otimes m^{-})(((m^{+})^{-1})^{t}\otimes I)(C_{i}^{t}\otimes D_{i})
=(C_{i}^{t}\otimes D_{i})(((m^{+})^{-1})^{t}\otimes I)(I\otimes m^{-})\\
\Longleftrightarrow
m^{-}_{2}((m^{+})^{-1})^{t}_{1}(R^{-1})^{t_{1}}=(R^{-1})^{t_{1}}((m^{+})^{-1})^{t}_{1}m^{-}_{2}\\
\Longleftrightarrow
((R^{-1})^{t_{1}})^{-1}m^{-}_{2}((m^{+})^{-1})^{t}_{1}=((m^{+})^{-1})^{t}_{1}m^{-}_{2}((R^{-1})^{t_{1}})^{-1}.
\end{array}
\right.
\end{equation}
\begin{equation}\label{dual9}
\left.
\begin{array}{l}
(\partial^{R}(m^{-})^{k}_{l}\ast\R)(\tilde{t}^{i}_{j})
=\partial^{R}(m^{-})^{k}_{l}(\tilde{t}^{i}_{a})\R(\tilde{t}^{a}_{j})
=(m^{-})^{b}_{l}\langle (m^{-})^{k}_{b},\tilde{t}^{i}_{a}\rangle\R(\tilde{t}^{a}_{j})\\
=((R^{-1})^{t_{1}})^{-1})^{ik}_{ab}(m^{-})^{b}_{l}[((m^{+})^{-1})^{t}]^{a}_{j}
=[((R^{-1})^{t_{1}})^{-1}m^{-}_{2}((m^{+})^{-1})^{t}_{1}]^{ik}_{jl}\\
=[((m^{+})^{-1})^{t}_{1}m^{-}_{2}((R^{-1})^{t_{1}})^{-1}]^{ik}_{jl}\quad
(\mbox{by}\ \eqref{eqvm})\\
=[((m^{+})^{-1})^{t}]^{i}_{a}(m^{-})^{k}_{b}((R^{-1})^{t_{1}})^{-1})^{ab}_{jl}\\
=[((m^{+})^{-1})^{t}]^{i}_{a}\langle(m^{-})^{b}_{l},\tilde{t}^{a}_{j}\rangle(m^{-})^{k}_{b}
=\R(\tilde{t}^{i}_{a})\partial^{L}(m^{-})^{k}_{l}(\tilde{t}^{a}_{j})
=(\R \ast\partial^{L}(m^{-})^{k}_{l})(\tilde{t}^{i}_{j}).
\end{array}
\right.
\end{equation}
By \eqref{dual7}, \eqref{dual8} and \eqref{dual9},
we obtain that
$\partial^{R}(m^{\pm})^{k}_{l}\ast\R=\R \ast\partial^{L}(m^{\pm})^{k}_{l}$,
thus prove the relation \eqref{dual5}.
\eqref{dual6} can also be proved in a similar way.

This completes the proof.\hfill $\Box$

\subsection{The generalized double-bosonization construction Theorem}
After building up the weakly quasitriangular dual pairs of Hopf
algebras suitably  for general $R$-matrices $R_{VV}$'s (especially
for the irregular $R$-matrices), we are in a position to establish
our generalized double-bosonization construction Theorem for any
$R_{VV}$. First of all, as we known, the two braided (co-)vector
algebras involved in the double-bosonization construction were
defined by Majid in the braided categories of left (right)
$A(R)$-comodules, i.e.,
 $V(R^{\prime},R)\in {}^{A(R)}\mathfrak M$, $V^{\vee}(R^{\prime},R_{21}^{-1})\in
 \mathfrak M^{A(R)}$. However, our starting objects are the dual
 pairs of Hopf algebras $(U_q^{\rm ext}(\mathfrak g), H_{R_{VV}})$ for the irregular
 cases. Now by definition, observing that the FRT-bialgebra $A(R)$ is a subbialgebra of $H_{R}$,
 we can view
 $V(R^{\prime},R)\in {}^{H_R}\mathfrak M$, $V^{\vee}(R^{\prime},R_{21}^{-1})\in
 \mathfrak M^{H_R}$.

Actually, in order to fit well into the application to the inductive
construction of the exceptional quantum groups, we have to refine
our framework.  We not only need to work with the pairs of their
central extensions: $( \widetilde{U_{q}^{\textrm{ext}}(\mathfrak
g)}, \widetilde{H_{R_{VV}}})$, where $
\widetilde{U_{q}^{\textrm{ext}}(\mathfrak
g)}=U_{q}^{\textrm{ext}}(\mathfrak g)\otimes k[c,c^{-1}]$, $
\widetilde{H_{R_{VV}}}=H_{R_{VV}}\otimes k[g,g^{-1}]$, but also to
make the normalization of $R_{VV}$ at certain eigenvalue of the
braiding $PR_{VV}$ to obtain a quantum group normalization constant
$\lambda$  and the needed pair $(R, R')$ determined by the minimal
polynomial of $PR_{VV}$. Noticing the relation $R_{VV}=\lambda R$,
and $H_{R_{VV}}\subset \widetilde{H_{R_{VV}}}$, we have $A(R)\cong
A(R_{VV})$ as bialgebras, consequently, we get a bialgebra embedding
$A(R)\hookrightarrow \widetilde{H_{R_{VV}}}$. This implies that our
concerned braided objects $V(R^{\prime},R)\in
{}^{\widetilde{H_{R_{VV}}}}\mathfrak M$,
$V^{\vee}(R^{\prime},R_{21}^{-1})\in
 \mathfrak M^{\widetilde{H_{R_{VV}}}}$. To be more precisely, we
 have

\begin{lemma}\label{pair}
Let $R_{VV}=\lambda R$ and
$(\widetilde{U_{q}^{\textrm{ext}}(\mathfrak g)},
\widetilde{H_{R_{VV}}})$ be the centrally extended weakly
quasitriangular dual pair defined by
\begin{equation}\label{ext}
\Delta(c)=c\otimes c,\quad
\Delta(g)=g\otimes g,\quad
\langle c,g\rangle=\lambda,\quad
\R(g)=c^{-1},\quad
\bar{\R}(g)=c.
\end{equation}
Then $V(R^{\prime},R)$ and $V^{\vee}(R^{\prime},R_{21}^{-1})$ are
braided groups in the categories
$^{\widetilde{H_{R_{VV}}}}\mathfrak{M}$ and $
\mathfrak{M}^{\widetilde{H_{R_{VV}}}}$, respectively, where the
coactions are given by $e^{i} \mapsto g\,t^{i}_{a}\otimes e^{a}$ and
$ f_{i} \mapsto f_{a}\otimes g\,t^{i}_{a}$.
\end{lemma}
\begin{proof}
The duality pairing between
$\widetilde{U_{q}^{\textrm{ext}}(\mathfrak g)}$ and
$\widetilde{H_{R_{VV}}}$ can be induced by $\langle a\otimes
c,b\otimes g\rangle=\langle a,b\rangle\langle c,g\rangle$, for $a\in
U_{q}^{\textrm{ext}}(\mathfrak g), b\in H_{R_{VV}}$, and we have
\begin{gather*}
\langle c,g^{-1}\rangle=\langle c^{-1},g\rangle=\lambda^{-1},\quad
\langle c^{-1},g^{-1}\rangle=\lambda,\\
\langle c,t^{i}_{j}\rangle
=\langle 1\otimes c,t^{i}_{j}\otimes 1\rangle
=\langle1,t^{i}_{j}\rangle\langle c,1\rangle
=\epsilon(t^{i}_{j})\epsilon(c)
=\delta_{ij},\\
\langle c,\tilde{t}^{i}_{j}\rangle
=\langle 1\otimes c,\tilde{t}^{i}_{j}\otimes 1\rangle
=\langle1,\tilde{t}^{i}_{j}\rangle\langle c,1\rangle
=\epsilon(\tilde{t}^{i}_{j})\epsilon(c)
=\delta_{ij},\\
\langle(m^{\pm})^{i}_{j},g\rangle
=\langle(m^{\pm})^{i}_{j}\otimes 1,1\otimes g\rangle
=\langle(m^{\pm})^{i}_{j},1\rangle\langle 1,g\rangle
=\epsilon((m^{\pm})^{i}_{j})
=\delta_{ij}.
\end{gather*}
According to $\langle ab,c\rangle=\langle a,c_{(1)}\rangle\langle b,c_{(2)}\rangle$,
we obtain
\begin{equation}\label{ext1}
\langle ((m^{+})^{-1})^{i}_{j},g\rangle=\delta_{ij}=\langle[(((m^{+})^{-1})^{t})^{-1}]^{i}_{j},g\rangle.
\end{equation}
With these,
we have
\begin{gather}
\langle\bar{\R}(g),t^{i}_{j}\rangle=\langle\R^{-1}(t^{i}_{j}),g \rangle=\delta_{ij},\quad
\langle\bar{\R}(g),\tilde{t}^{i}_{j}\rangle=\langle\R^{-1}(\tilde{t}^{i}_{j}),g \rangle=\delta_{ij},\label{ext2}\\
\langle\bar{\R}(t^{i}_{j}),g\rangle=\langle\R^{-1}(g),t^{i}_{j} \rangle=\delta_{ij},\quad
\langle\bar{\R}(\tilde{t}^{i}_{j}),g\rangle=\langle\R^{-1}(g),\tilde{t}^{i}_{j}\rangle=\delta_{ij}.\label{ext3}
\end{gather}
In view of $\R^{-1}(g)=c, \bar{\R}^{-1}(g)=c^{-1}$, and
\eqref{ext2}, \eqref{ext3}, we obtain that $
\langle\bar{\R}(a),b\rangle=\langle\R^{-1}(b),a\rangle$ holds for
any $a, b\in\widetilde{H_{R_{VV}}}$. Furthermore, by
$(m^{\pm})^{i}_{j}c=c(m^{\pm})^{i}_{j}$, we can prove easily that
\begin{equation*}
\partial^{R}c=\R \ast(\partial^{L}c)\ast\R^{-1},\quad
\partial^{R}c=\bar{\R}\ast(\partial^{L}c)\ast\bar{\R}^{-1}.
\end{equation*}
Hence, \eqref{ext} gives actually a weakly quasitriangular dual pair
between $\widetilde{U_{q}^{\textrm{ext}}(\mathfrak g)}$ and
$\widetilde{H_{R_{VV}}}$.

Finally, we will check that our concerned objects $V(R^{\prime},R)$,
$V^{\vee}(R^{\prime},R_{21}^{-1})$ are indeed the braided ones in
the braided categories $^{\widetilde{H_{R_{VV}}}}\mathfrak{M}$ and $
\mathfrak{M}^{\widetilde{H_{R_{VV}}}}$, respectively. To do so, we
need to verify that the braidings induced by the left\,/\,right
$\widetilde{H_{R_{VV}}}$-comodule structures of $V(R^{\prime},R)$
and $V^{\vee}(R^{\prime},R_{21}^{-1})$, respectively, are consistent
with those of their braided groups structure.

In fact, under the coaction $e^{i} \mapsto g\,t^{i}_{b}\otimes
e^{b}$ and $ e^{j} \mapsto g\,t^{j}_{a}\otimes e^{a}, $ we have
$$
\left.
\begin{array}{ll}
\Psi(e^{i}\otimes e^{j})=e^{a}\otimes e^{i}\lhd\R(gt^{j}_{a})
&=e^{a}\otimes e^{i}\lhd c^{-1}(m^{+})^{j}_{a}
=e^{a}\otimes e^{b}\langle c^{-1}(m^{+})^{j}_{a},gt^{i}_{b}\rangle\\
&=\langle c^{-1},g\rangle\langle(m^{+})^{j}_{a},t^{i}_{b}\rangle e^{a}\otimes e^{b}
=\lambda^{-1}R_{VV}{}^{j}_{a}{}^{i}_{b}e^{a}\otimes e^{b}\\
&=\lambda^{-1}\lambda R^{j}_{a}{}^{i}_{b}e^{a}\otimes e^{b}
=R^{j}_{a}{}^{i}_{b}e^{a}\otimes e^{b}
\end{array}
\right.
$$
for any $e^{i}, e^{j}\in V(R^{\prime},R)$. Using the right
$\widetilde{H_{R_{VV}}}$-comodule structure of
$V^{\vee}(R^{\prime},R_{21}^{-1})$, we can also obtain that $
\Psi(f_{i}\otimes f_{j})=(f_{b}\otimes
f_{a})\mathfrak{r}(g\,t^{a}_{i}\otimes g\,t^{b}_{j}) =f_{b}\otimes
f_{a}R^{ab}_{ij}$. The proof is complete.
\end{proof}

By the above lemma, we obtain dually-paired braided groups
$V(R^{\prime},R)$ and $V^{\vee}(R^{\prime},R_{21}^{-1})$ in the
categories ${}_{\widetilde{U_{q}^{\textrm{ext}}(\mathfrak
g)}}\mathfrak{M}, \
\mathfrak{M}_{\widetilde{U_{q}^{\textrm{ext}}(\mathfrak g)}}$,
respectively. With these, starting from general $R$-matrices
$R_{VV}$'s (especially for the irregular $R$-matrices), we can
arrive at our generalized version of double-bosonization
construction Theorem.
\begin{theorem}\label{cor1}
Let $R_{VV}$ be the $R$-matrix associated to an irreducible $U_{q}(\mathfrak{g})$-module $V$ with a minuscule highest weight.
There exists a normalization constant $\lambda$ such that $\lambda R=R_{VV}$. Then
the enlarged quantum group
$U=U(V^{\vee}(R^{\prime},R_{21}^{-1}),\widetilde{U_{q}^{ext}(\mathfrak
g)},V(R^{\prime},R))$ has the relations
\begin{gather*}
cf_{i}=\lambda f_{i}c,\quad
e^{i}c=\lambda ce^{i},\quad
[c,m^{\pm}]=0,\quad
[e^{i},f_{j}]=\delta_{ij}\frac{(m^{+})^{i}_{j}c^{-1}-c(m^{-})^{i}_{j}}{q_{\ast}-q_{\ast}^{-1}};\\
e^{i}(m^{+})^{j}_{k}=R_{VV}{}^{ji}_{ab}(m^{+})^{a}_{k}e^{b},\quad
(m^{-})^{i}_{j}e^{k}=R_{VV}{}^{ki}_{ab}e^{a}(m^{-})^{b}_{j},\\
(m^{+})^{i}_{j}f_{k}=f_{b}(m^{+})^{i}_{a}R_{VV}{}^{ab}_{jk},\quad
f_{i}(m^{-})^{j}_{k}=(m^{-})^{j}_{b}f_{a}R_{VV}{}^{ab}_{ik},
\end{gather*}
and the coproduct: $$\Delta c=c\otimes c, \quad \Delta
e^{i}=e^{a}\otimes (m^{+})^{i}_{a}c^{-1}+1\otimes e^{i}, \quad
\Delta f_{i}=f_{i}\otimes 1+c(m^{-})^{a}_{i}\otimes f_{a},$$ and the
counit $\epsilon e^{i}=\epsilon f_{i}=0$.
\end{theorem}
\begin{proof}
The weakly quasitriangular dual pair $(\widetilde{U_{q}^{\textrm{ext}}(\mathfrak g)}, \widetilde{H_{R_{VV}}})$
yields the cross relation, by Remark \ref{Cross}.
The coalgebra structure follows from Theorem \ref{ml1} due to Majid.
\end{proof}

\section{Applications: Type-crossing Constructions of $U_q(F_4)$ and $U_q(G_2)$}
This section is devoted to using Theorem 3.4 to give the type-crossing constructions
of $U_q(F_4)$ and $U_q(G_2)$. Among which, the first crucial point
is to choose an appropriate $U_{q}(\mathfrak{g}')$-module
for a Lie subalgebra $\mathfrak g'\subset\mathfrak g$ of corank $1$
to obtain a suitable $R$-matrix $R_{VV}$ (irregular in the
exceptional cases).

\subsection{Type-crossing of $U_{q}(F_{4})$ via $U_q(B_3)$}
Let us begin with $U_q(B_{3})$ and its 8-dimensional spin module $V$,
which is given by the following Figure 1 (see \cite{hk}).

\setlength{\unitlength}{1mm}
\begin{picture}(50,70)
\put(23,0){Figure $1$.~~spin module $V$ for $U_{q}(\mathfrak{so}_{7})$}
\put(40,5){$(-,-,-)$}
\put(35,12){\line(3,-2){8}}
\put(36,8){$\alpha_{3}$}
\put(30,13){$(-,-,+)$}
\put(35,14){\line(0,1){9}}
\put(36,18){$\alpha_{2}$}
\put(25,28){$\alpha_{3}$}
\put(43,28){$\alpha_{1}$}
\put(25,38){$\alpha_{1}$}
\put(43,38){$\alpha_{3}$}
\put(30,23){(--,+,--)}
\put(36,25){\line(3,2){10}}
\put(25,32){\line(3,-2){10}}
\put(20,33){(--,+,+)}
\put(26,35){\line(3,2){10}}
\put(40,33){$(+,-,-)$}
\put(30,43){(+,--,+)}
\put(37,42){\line(3,-2){10}}
\put(35,44){\line(0,1){9}}
\put(36,48){$\alpha_{2}$}
\put(25,58){$\alpha_{3}$}
\put(30,53){(+,+,--)}
\put(25,62){\line(3,-2){10}}
\put(20,63){(+,+,+)}
\end{picture}

The representation space
$V=\text{Span}_\mathbb{C}\{v_{1},\cdots,v_{8}\}$. Let
$(s_{1},s_{2},s_{3})$ denote the basis vector of weight
$\frac{1}{2}(s_{1}\varepsilon_{1}+s_{2}\varepsilon_{2}+s_{3}\varepsilon_{3})$,
$s_{i}=\pm$. Every basis vector $v_{i}$ is denoted by $
v_{1}=(-,-,-), v_{2}=(-,-,+), v_{3}=(-,+,-), v_{4}=(-,+,+),
v_{5}=(+,-,-), v_{6}=(+,-,+), v_{7}=(+,+,-), v_{8}=(+,+,+). $
$E_{i},F_{i}$'s actions can be read off directly from Figure $1$,
then $K_{i},H_{i}$'s actions can be obtained by
$[E_{i},F_{j}]=\delta_{ij}\frac{K_{i}-K_{i}^{-1}}{q_{i}-q_{i}^{-1}}$
(or
$\delta_{ij}\frac{e^{d_{i}hH_{i}}-e^{-d_{i}hH_{i}}}{e^{d_{i}h}-e^{-d_{i}h}})$.
Moreover,
we observe that the spin module $V$ is a minuscule weight module, and $E_{\beta}(v_{m})=v_{n} \Longleftrightarrow
F_{\beta}(v_{n})=v_{m}$, for some $m,\,n$ (see \cite{hk}), as well as $E_{i}^{k},
F_{i}^{k}$ are $0$ for all $k\geq 2$.
Based on these, we can prove the following
\begin{lemma}\label{weight}
$\mu_{m}$ denote the weight of basis vector $ v_{m}$ for any $m$.
If there exist $v_i, v_j$ and $v_k, v_l$ such that
$E_{i_1}^{r_1}E_{i_2}^{r_2}\cdots E_{i_s}^{r_s}(v_i)=v_k,$
$F_{j_1}^{t_1}F_{j_2}^{t_2}\cdots F_{j_s}^{t_s}(v_j)=v_l,$
where
$\{j_1, j_2, \cdots, j_s\}\in\{i_{w(1)}, i_{w(2)}, \cdots, i_{w(s)}|w\in Sym_s\}$,
and $t_k=r_{a}$ when $j_{k}=i_{a}$,
$r_m\in\{0,1\}$ for any $m$,
then we have $(\mu_{i}, \mu_{j})=(\mu_{k}, \mu_{l})$.
\end{lemma}
\begin{proof}
We will prove this Lemma by induction on $r_1+r_2+\cdots+r_s$.

{\bf(1)} When $r_1+r_2+\cdots+r_s=1$,
namely,
the situation of simple root vectors.

According to the spin
representation, suppose that there exist simple root vectors
$E_{\beta}$ and $F_{\beta}$ such that $ E_{\beta}(v_{i})=v_{k},
F_{\beta}(v_{j})=v_{l}, $
which is equivalent to $
E_{\beta}(v_{l})=v_{j}, F_{\beta}(v_{k})=v_{i},$
then $ \mu_{k}=\mu_{i}+\beta,
\mu_{l}=\mu_{j}-\beta. $ Owing to $E_{\beta}^{2}, F_{\beta}^{2}$ are
zero action, we obtain $
E_{\beta}(v_{j})=E_{\beta}(E_{\beta}(v_{l}))=0,
F_{\beta}(v_{i})=F_{\beta}(F_{\beta}(v_{k}))=0. $ According to $
E_{\beta}F_{\beta}-F_{\beta}E_{\beta}=
\frac{K_{\beta}-K_{\beta}^{-1}}{q_{\beta}-q_{\beta}^{-1}}, $ both
sides acting on $v_{i},\, v_{j}$, respectively, we obtain
\begin{gather*}
[E_{\beta},F_{\beta}](v_{i})=
\frac{K_{\beta}-K_{\beta}^{-1}}{q_{\beta}-q_{\beta}^{-1}}(v_{i}),
\Longrightarrow (\mu_{i},\beta)=-\frac{(\beta,\beta)}{2},
\\
[E_{\beta},F_{\beta}](v_{j})=
\frac{K_{\beta}-K_{\beta}^{-1}}{q_{\beta}-q_{\beta}^{-1}}(v_{j}),
\Longrightarrow (\beta,\mu_{j})=\frac{(\beta,\beta)}{2}.
\end{gather*}
Then we have the following equalities
$$
\left.
\begin{array}{rl}
(\mu_{k},\mu_{l})
&=(\mu_{i}+\beta,\mu_{j}-\beta)
=(\mu_{i},\mu_{j})-(\mu_{i},\beta)+(\beta,\mu_{j})-(\beta,\beta)\\
&=(\mu_{i},\mu_{j})-(-\frac{(\beta,\beta)}{2})+\frac{(\beta,\beta)}{2}-(\beta,\beta)
=(\mu_{i},\mu_{j}).
\end{array}
\right.
$$

{\bf(2)} When $r_1+r_2+\cdots+r_s=2$.

If there are basis vectors $v_{k}$ and $v_{l}$ such that
$E_{i}E_{j}(v_{k})=v_{m}$ and $F_{j}F_{i}(v_{l})=v_{n}$(or~$F_{i}F_{j}(v_{l})=v_{n}$), then
$\mu_{m}=\mu_{k}+\alpha_{i}+\alpha_{j},
\mu_{n}=\mu_{l}-\alpha_{i}-\alpha_{j}$.
Let us consider the case of $E_{i}E_{j}(v_{k})=v_{m}$ and $F_{j}F_{i}(v_{l})=v_{n}$.
From these,
we have the
following equivalent relations
\begin{gather}
E_{i}E_{j}(v_{k})=v_{m}
\Longleftrightarrow
E_{j}(v_{k})=v_{x},\quad
E_{i}(v_{x})=v_{m},\label{sym1}\\
F_{j}F_{i}(v_{l})=v_{n}
\Longleftrightarrow
F_{i}(v_{l})=v_{y},\quad
F_{j}(v_{y})=v_{n}\label{sym2}
\end{gather}
for some basis vectors $v_{x}, v_{y}$. 
By \eqref{sym1} and \eqref{sym2}, we obtain the following
equalities for these weights based on the above analysis for simple
root vectors.
\begin{equation}\label{sym3}
\left\{
\begin{array}{l}
(\mu_{k},\alpha_{j})=-\frac{(\alpha_{j},\alpha_{j})}{2},\quad
(\mu_{y},\alpha_{j})=\frac{(\alpha_{j},\alpha_{j})}{2},\\
(\mu_{x},\alpha_{i})=-\frac{(\alpha_{i},\alpha_{i})}{2},\quad
(\mu_{l},\alpha_{i})=\frac{(\alpha_{i},\alpha_{i})}{2}.
\end{array}
\right.
\end{equation}
According to $(\mu_{y},\alpha_{j})=\frac{(\alpha_{j},\alpha_{j})}{2}$
and $(\mu_{x},\alpha_{i})=-\frac{(\alpha_{i},\alpha_{i})}{2}$,
respectively,
we obtain
$$(\mu_{l},\alpha_{j})=(\alpha_{i},\alpha_{j})+\frac{(\alpha_{j},\alpha_{j})}{2},\quad
(\mu_{k},\alpha_{i})=-(\alpha_{i},\alpha_{j})-\frac{(\alpha_{i},\alpha_{i})}{2}.$$
Then we have the following relations by \eqref{sym3},
$$
\left\{
\begin{array}{l}
(\mu_{l},\alpha_{i}+\alpha_{j})=(\alpha_{i},\alpha_{j})+\frac{(\alpha_{i},\alpha_{i})+(\alpha_{j},\alpha_{j})}{2},\\
(\mu_{k},\alpha_{i}+\alpha_{j})=-(\alpha_{i},\alpha_{j})-\frac{(\alpha_{i},\alpha_{i})+(\alpha_{j},\alpha_{j})}{2}.
\end{array}
\right.
$$
With these,
we have
$$
\left.
\begin{array}{rl}
(\mu_{m},\mu_{n})&=(\mu_{k}+\alpha_{i}+\alpha_{j},\mu_{l}-\alpha_{i}-\alpha_{j})\\
&=(\mu_{k},\mu_{l})-(\mu_{k},\alpha_{i}+\alpha_{j})+(\mu_{l},\alpha_{i}+\alpha_{j})
-(\alpha_{i}+\alpha_{j},\alpha_{i}+\alpha_{j})\\
&=(\mu_{k},\mu_{l})+(\alpha_{i},\alpha_{i})+(\alpha_{j},\alpha_{j})+2(\alpha_{i},\alpha_{j})
-(\alpha_{i}+\alpha_{j},\alpha_{i}+\alpha_{j})\\
&=(\mu_{k},\mu_{l}).
\end{array}
\right.
$$
In a similar way,
we also prove $(\mu_{m},\mu_{n})=(\mu_{k},\mu_{l})$ when $E_{i}E_{j}(v_{k})=v_{m}$ and $F_{i}F_{j}(v_{l})=v_{n}$.

{\bf (3)} Suppose that the Proposition is correct for $r_1+r_2+\cdots+r_s=n-1~(n\geq 2)$,
then we will consider the case of $r_1+r_2+\cdots+r_s=n$.

If there are $v_i, v_j, v_k, v_l$ such that
$E_{i_1}^{r_1}E_{i_2}^{r_2}\cdots E_{i_s}^{r_s}(v_i)=v_k,$
$F_{j_1}^{t_1}F_{j_2}^{t_2}\cdots F_{j_s}^{t_s}(v_j)=v_l,$
and $r_1\neq 0,$
$t_1\neq 0$,
then we have the following equivalent relations:
\begin{gather*}
E_{i_1}^{r_1}E_{i_2}^{r_2}\cdots E_{i_s}^{r_s}(v_i)=v_k
\Longleftrightarrow
E_{i_1}^{r_1-1}E_{i_2}^{r_2}\cdots E_{i_s}^{r_s}(v_i)=v_x,\,
E_{i_1}(v_x)=v_k,\\
F_{j_1}^{t_1}F_{j_2}^{t_2}\cdots F_{j_s}^{t_s}(v_j)=v_l
\Longleftrightarrow
F_{j_1}^{t_1-1}F_{j_2}^{t_2}\cdots F_{j_s}^{t_s}(v_j)=v_y,\,
F_{j_1}(v_y)=v_l.
\end{gather*}
Thus we have
$(\mu_{i},\mu_{j})=(\mu_x, \mu_y)$ by supposition,
and $(\mu_x, \mu_y)=(\mu_{k},\mu_{l})$ by {\bf (1)},
then we prove $(\mu_{i},\mu_{j})=(\mu_{k},\mu_{l})$.
According to {\bf (1), (2), (3)},
the proof is complete.
\end{proof}
For simplicity,
we denote $E_{i_1}^{r_1}E_{i_2}^{r_2}\cdots E_{i_s}^{r_s},\,
F_{j_1}^{t_1}F_{j_2}^{t_2}\cdots F_{j_s}^{t_s}$ in the above Lemma \ref{weight} by $E^{\underline{r}}_{\underline{i}}$,
$E^{\underline{t}}_{\underline{j}}$,
respectively.
Corresponding to this spin representation,
we get the upper triangular $64\times 64$ $R$-matrix $R_{VV}$
associated with universal $R$-matrix and the above basis of representation space.
The explicit expressions of root vectors in universal $R$-matrix \eqref{imp2} are obtained by Lusztig's automorphisms $T_i$'s \cite{lus}.
For example,
we fix $w_0=s_3s_2s_3s_2s_1s_2s_3s_2s_1$ a reduced decomposition of the longest element $w_0$ of Weyl group $W$ for $B_3$,
then the corresponding sequence of all positive roots of $\mathfrak so _7$ is
$\alpha_{3}, \alpha_{2}+2\alpha_{3}, \alpha_{2}+\alpha_{3}, \alpha_{2}, \alpha_{1}+2\alpha_{2}+2\alpha_{3},
\alpha_{1}+\alpha_{2}+2\alpha_{3}, \alpha_{1}+\alpha_{2}+\alpha_{3}, \alpha_{1}+\alpha_{2}, \alpha_{1}$.
For composition root $\alpha_{2}+\alpha_{3}$,
we obtain root vectors
$E_{\alpha_{2}+\alpha_{3}}=-E_{3}E_{2}+q^{-1}E_{2}E_{3}$,
$F_{\alpha_{2}+\alpha_{3}}=-F_{2}F_{3}+qF_{3}F_{2}$.
Moreover,
by these root vectors and the equality \eqref{rmatrix},
we obtain
$R_{VV}{}^{53}_{71}={-}q^{-1}(q^{\frac{1}{2}}-q^{-\frac{1}{2}})q^{-\frac{1}{4}}$
and $R_{VV}{}^{17}_{35}={-}q(q^{\frac{1}{2}}-q^{-\frac{1}{2}})q^{-\frac{1}{4}}$,
then we find
$(PR_{VV})^{35}_{71}\neq(PR_{VV})^{71}_{35}$.
That is, the matrix $PR_{VV}$ isn't symmetric.
However,
we have the following
\begin{proposition}
Corresponding to the 8-dimensional spin $U_q(\mathfrak{so}_{7})$-module in Figure 1,
the braiding matrix $PR_{VV}$ is symmetrizable, that is, the braiding is
of diagonal type.
\end{proposition}
\begin{proof}
We know that dual cases, obtained by Lusztig's automorphisms $T_i$'s,
is the standard choice of bases satisfying the condition \eqref{imp1} in Lemma \ref{orbit}.
Naturally,
if we can determine another dual base of $U_q^{+}(\mathfrak{b})$ and $U_q^{-}(\mathfrak{b})$
such that they satisfy \eqref{imp1},
namely,
$\langle E_\beta, F_\beta\rangle'=\langle E_i, F_i\rangle'$ if a positive $\beta$ and a simple root $\alpha_i$ belong to the same $W$-orbit,
then we derive another explicit formula of universal $R$-matrix of $U_q(\mathfrak g)$.
In fact,
this is equivalent to make base transformation for vector spaces $U_q^{+}(\mathfrak{b})$ and $U_q^{-}(\mathfrak{b})$.
Specially,
for $U_q(\mathfrak{so}_{7})$,
we choose the following expressions for composition positive root vectors:
\begin{gather*}
E_{\alpha_1+\alpha_2}=-E_2E_1+q^{-1}E_1E_2,\quad
E_{\alpha_2+\alpha_3}=-E_3E_2+q^{-1}E_2E_3,\\
E_{\alpha_2+2\alpha_3}=E_3^{(2)}E_2-q^{-\frac{1}{2}}E_3E_2E_3+q^{-1}E_2E_3^{(2)},
\quad
\mbox{where}~E_3^{(2)}:=\frac{(E_3)^{2}}{[2]_{q_{3}}!},\\
E_{\alpha_1+\alpha_2+\alpha_3}=-E_3E_2E_1+q^{-1}E_2E_3E_1+q^{-1}E_1E_3E_2-q^{-2}E_1E_2E_3,\\
E_{\alpha_1+\alpha_2+2\alpha_3}
=-E_{3}^{(2)}E_{\alpha_1+\alpha_2}-q^{-1}E_{\alpha_1+\alpha_2}E_{3}^{(2)}+q^{-\frac{1}{2}}E_3E_{\alpha_1+\alpha_2}E_3,\\
E_{\alpha_1+2\alpha_2+2\alpha_3}
=-E_2E_{\alpha_2+2\alpha_3}E_1
+q^{-1}E_2E_1E_{\alpha_2+2\alpha_3}+q^{-1}E_{\alpha_2+2\alpha_3}E_1E_2-q^{-2}E_1E_{\alpha_2+2\alpha_3}E_2.\\
\end{gather*}
For negative root vectors:
\begin{gather*}
F_{\alpha_1+\alpha_2}=\frac{1}{2}qF_2F_1-\frac{1}{2}q^{2}F_1F_2,\quad
F_{\alpha_2+\alpha_3}=(q-\frac{1}{2}q^{\frac{3}{2}})F_3F_2-(q^{2}-\frac{1}{2}q^{\frac{5}{2}})F_2F_3,\\
F_{\alpha_2+2\alpha_3}=F_2F_3^{(2)}-q^{\frac{1}{2}}F_3F_2F_3+qF_3^{(2)}F_2,
\quad
\mbox{where}~F_3^{(2)}:=\frac{(F_3)^{2}}{[2]_{q_{3}}!},\\
F_{\alpha_1+\alpha_2+\alpha_3}=-q^{-2}F_3F_2F_1+q^{-1}F_2F_3F_1+q^{-1}F_1F_3F_2-F_1F_2F_3,\\
F_{\alpha_1+\alpha_2+2\alpha_3}
=-F_{\alpha_1+\alpha_2}F_{3}^{(2)}-q^{-1}F_{3}^{(2)}F_{\alpha_1+\alpha_2}+q^{-\frac{1}{2}}F_3F_{\alpha_1+\alpha_2}F_3,\\
F_{\alpha_1+2\alpha_2+2\alpha_3}
=-F_1F_{\alpha_2+2\alpha_3}F_2
+qF_{\alpha_2+2\alpha_3}F_1F_2+qF_2F_1F_{\alpha_2+2\alpha_3}-q^{2}F_2F_{\alpha_2+2\alpha_3}F_1.\\
\end{gather*}
According to
\begin{gather*}
\langle E_kE_l, F_mF_n\rangle'=\delta_{km}\delta_{ln}\frac{1}{q_m-q_m^{-1}}\frac{1}{q_m-q_m^{-1}}+
\delta_{lm}\delta_{kn}q^{-(\alpha_{n},\alpha_l)}\frac{1}{q_m-q_m^{-1}}\frac{1}{q_m-q_m^{-1}},
\\
\left.
\begin{array}{rl}
\langle E_iE_jE_k, F_mF_nF_p\rangle'&=[\delta_{im}\delta_{kn}\delta_{jp}+\delta_{km}\delta_{in}\delta_{jp}q^{-(\alpha_k,\alpha_n)}]
q^{-(\alpha_k,\alpha_p)}\frac{1}{(q_m-q_m^{-1})(q_n-q_n^{-1})(q_j-q_j^{-1})}\\
&+[\delta_{jm}\delta_{kn}\delta_{ip}+\delta_{km}\delta_{jn}\delta_{ip}q^{-(\alpha_k,\alpha_n)}]
q^{-(\alpha_j+\alpha_k,\alpha_p)}\frac{1}{(q_m-q_m^{-1})(q_n-q_n^{-1})(q_i-q_i^{-1})}\\
&+[\delta_{im}\delta_{jn}\delta_{kp}+\delta_{jm}\delta_{in}\delta_{kp}q^{-(\alpha_j,\alpha_n)}]
\frac{1}{(q_m-q_m^{-1})(q_n-q_n^{-1})(q_k-q_k^{-1})},
\end{array}
\right.
\end{gather*}
it can check that
$\langle E_\beta, F_\beta\rangle'=\langle E_i, F_i\rangle'$ for these new root vectors if a positive $\beta$ and a simple root $\alpha_i$ belong to the same $W$-orbit,
i.e.,
they are satisfy the condition \eqref{imp1}.
Starting from this new dual base,
we can check that the corresponding matrix $PR_{VV}$ is symmetric.
In virtue of the actions of spin modules in Figure 1,
and $E_{k}^{2}, F_{k}^{2}$ is zero as operators,
these new root vectors as operators have the following equalities:
\begin{equation}\label{comp1}
\left.
\begin{array}{c}
E_{\alpha_1+\alpha_2}=-E_2E_1+q^{-1}E_1E_2,\quad
E_{\alpha_2+\alpha_3}=-E_3E_2+q^{-1}E_2E_3,\quad
E_{\alpha_2+2\alpha_3}=-q^{-\frac{1}{2}}E_3E_2E_3,
\\
E_{\alpha_1+\alpha_2+\alpha_3}=-E_3E_2E_1+q^{-1}E_2E_3E_1+q^{-1}E_1E_3E_2-q^{-2}E_1E_2E_3,\\
E_{\alpha_1+\alpha_2+2\alpha_3}=q^{-\frac{1}{2}}E_3E_{\alpha_1+\alpha_2}E_3=
-q^{-\frac{1}{2}}E_3E_2E_1E_3+q^{-\frac{3}{2}}E_3E_1E_2E_3,\\
E_{\alpha_1+2\alpha_2+2\alpha_3}
=-q^{-\frac{3}{2}}E_2E_1E_3E_2E_3-q^{-\frac{3}{2}}E_3E_2E_3E_1E_2;\\
F_{\alpha_1+\alpha_2}=\frac{1}{2}qF_2F_1-\frac{1}{2}q^{2}F_1F_2,\quad
F_{\alpha_2+\alpha_3}=(q-\frac{1}{2}q^{\frac{3}{2}})F_3F_2-(q^{2}-\frac{1}{2}q^{\frac{5}{2}})F_2F_3,\quad
\\
F_{\alpha_2+2\alpha_3}=-q^{\frac{1}{2}}F_3F_2F_3,\quad
F_{\alpha_1+\alpha_2+\alpha_3}=-q^{-2}F_3F_2F_1+q^{-1}F_2F_3F_1+q^{-1}F_1F_3F_2-F_1F_2F_3,\\
F_{\alpha_1+\alpha_2+2\alpha_3}
=q^{-\frac{1}{2}}F_3F_{\alpha_1+\alpha_2}F_3=
\frac{1}{2}q^{\frac{1}{2}}F_3F_2F_1F_3-\frac{1}{2}q^{\frac{3}{2}}F_3F_1F_2F_3,\\
F_{\alpha_1+2\alpha_2+2\alpha_3}
=-q^{\frac{3}{2}}F_2F_1F_3F_2F_3-q^{\frac{3}{2}}F_3F_2F_3F_1F_2.
\end{array}
\right.
\end{equation}
By \eqref{rmatrix},
we know that
$$
\left.
\begin{array}{rl}
&B_{VV}\circ(T_V\otimes T_V)(\mathfrak R)(v_i\otimes v_j)\\
=&B_{VV}\circ (T_V\otimes T_V)\left(\sum_
{
[(\underline{r}, \underline{i}),(\underline{t}, \underline{j})]
\in\Omega^{ij}_{kl}
}
x^{(\underline{r},\underline{i})}_{(\underline{t},\underline{j})}
E^{\underline{r}}_{\underline{i}}\otimes F^{\underline{t}}_{\underline{j}}
+others\right)(v_i\otimes v_j)\\
=&q^{(\mu_k, \mu_l)}\sum_
{
[(\underline{r}, \underline{i}),(\underline{t}, \underline{j})]
\in\Omega^{ij}_{kl}
}
x^{(\underline{r},\underline{i})}_{(\underline{t},\underline{j})}(v_k\otimes v_l)+others,
\end{array}
\right.
$$
where coefficients $x^{(\underline{r},\underline{i})}_{(\underline{t},\underline{j})}\in \mathbb{C}[q,q^{-1}]$,
and the set
$\Omega^{ij}_{kl}$ is defined as
$$\Omega^{ij}_{kl}:=\{[(\underline{r}, \underline{i}),(\underline{t}, \underline{j})]|E^{\underline{r}}_{\underline{i}}(v_i)=E_{i_1}^{r_1}E_{i_2}^{r_2}\cdots E_{i_s}^{r_s}(v_i)=v_k,
F^{\underline{t}}_{\underline{j}}(v_j)=F_{j_1}^{t_1}F_{j_2}^{t_2}\cdots F_{j_s}^{t_s}(v_j)=v_l\}.$$
$E^{\underline{r}}_{\underline{i}}, F^{\underline{t}}_{\underline{j}}$ can be divided into the following three cases:

{\bf(1)} Diagonal entries:

Firstly, the diagonal
entries in the matrix $R_{VV}$ satisfy
$R_{VV}{}^{i}_{i}{}^{j}_{j}=R_{VV}{}^{j}_{j}{}^{i}_{i}=q^{(\mu_{i},\mu_{j})}$,
where $ \mu_{i},\mu_{j} $ denote the weight of $ v_{i}, v_{j}, $
respectively. Next we will consider the non-diagonal entries in the
matrix $R_{VV}$.
Since the operators $E_{i}^{k}$'s and $F_{i}^{k}$'s
are $0$ for all $k\geq 2$, we have the equality
$$B_{VV}\circ(T_{V}\otimes
T_{V})(\mathfrak{R})(v_{i}\otimes v_{j})=B_{VV}\circ(T_{V}\otimes
T_{V})\left(1+\sum\limits_{\alpha}E_{\alpha}\otimes
F_{\alpha}\right)(v_{i}\otimes v_{j}).$$

{\bf(2)} The entries related to simple root vectors:

According to the spin
representation, suppose that there exist simple root vectors
$E_{\beta}$ and $F_{\beta}$ such that $ E_{\beta}(v_{i})=v_{k},
F_{\beta}(v_{j})=v_{l}, $ which is equivalent to $
E_{\beta}(v_{l})=v_{j}, F_{\beta}(v_{k})=v_{i}. $ Namely, the entry
$R_{VV}{}^{ij}_{kl}\neq 0$, also is equivalent to
$R_{VV}{}^{lk}_{ji}\neq 0$.
Moreover,
by Lemma \ref{weight},
the entry
$ R_{VV}{}^{ij}_{kl}=
(q_{\beta}-q_{\beta}^{-1})q^{(\mu_{k},\mu_{l})}
=(q_{\beta}-q_{\beta}^{-1})q^{(\mu_{i},\mu_{j})}$.
Correspondingly,
$ R_{VV}{}^{lk}_{ji}=
(q_{\beta}-q_{\beta}^{-1})q^{(\mu_{i},\mu_{j})}, $ we obtain $
R_{VV}{}^{ij}_{kl}=R_{VV}{}^{lk}_{ji}, $ which is equivalent to $
(PR_{VV})^{ji}_{kl}=(PR_{VV})^{kl}_{ji}. $

{\bf(3)} The entries related to quantum composite root vectors:

The expressions of all composite vectors as operators are listed by \eqref{comp1} as above,
which are $q$-commutators of simple root vectors. Thereby, the actions of
these composite root vectors can be deduced by the actions of simple
root vectors.
Suppose that there are $v_i, v_j, v_k, v_l$ such that
$E_{i_1}^{r_1}E_{i_2}^{r_2}\cdots E_{i_s}^{r_s}(v_i)=v_k$,
$F_{j_1}^{t_1}F_{j_2}^{t_2}\cdots F_{j_s}^{t_s}(v_j)=v_l$,
which is equivalent to
$F_{i_s}^{r_s}\cdots F_{i_2}^{r_2}F_{i_1}^{r_1}(v_k)=v_i$,
$E_{j_s}^{t_s}\cdots E_{j_2}^{t_2}E_{j_1}^{t_1}(v_l)=v_j$,
then $(\mu_i, \mu_j)=(\mu_k, \mu_l)$ by Lemma \ref{weight}.
On the other hand,
according to
$$\mathfrak{R}=x^{(\underline{r},\underline{i})}_{(\underline{t},\underline{j})}E_{i_1}^{r_1}E_{i_2}^{r_2}\cdots E_{i_s}^{r_s}\otimes F_{j_1}^{t_1}F_{j_2}^{t_2}\cdots F_{j_s}^{t_s}
+x^{(\underline{t},\underline{j})}_{(\underline{r},\underline{i})}E_{j_s}^{t_s}\cdots E_{j_2}^{t_2}E_{j_1}^{t_1}\otimes F_{i_s}^{r_s}\cdots F_{i_2}^{r_2}F_{i_1}^{r_1}+others,$$
then the corresponding entries are
$$R_{VV}{}^{ij}_{kl}=x^{(\underline{r},\underline{i})}_{(\underline{t},\underline{j})}q^{(\mu_k, \mu_l)},\quad
R_{VV}{}^{lk}_{ji}=x^{(\underline{t},\underline{j})}_{(\underline{r},\underline{i})}q^{(\mu_i, \mu_j)}=x^{(\underline{t},\underline{j})}_{(\underline{r},\underline{i})}q^{(\mu_k, \mu_l)}.$$
Thence,
we obtain the following equivalent relation
\begin{equation}\label{sym4}
R_{VV}{}^{ij}_{kl}=R_{VV}{}^{lk}_{ji} \Longleftrightarrow x^{(\underline{r},\underline{i})}_{(\underline{t},\underline{j})}
=x^{(\underline{t},\underline{j})}_{(\underline{r},\underline{i})}.
\end{equation}
Furthermore,
associated with the expressions \eqref{comp1} of composition root vectors as operators,
it is easy to check that the equality $x^{(\underline{r},\underline{i})}_{(\underline{t},\underline{j})}
=x^{(\underline{t},\underline{j})}_{(\underline{r},\underline{i})}$ in \eqref{sym4} is always true.
Thereby,
we have the equality
$R_{VV}{}^{ij}_{kl}=R_{VV}{}^{lk}_{ji}$,
namely,
$(PR_{VV})^{ji}_{kl}=(PR_{VV})^{kl}_{ji}$.

Let us demonstrate this by an example.
For example,
starting from the composite root $\alpha_1+\alpha_2$,
we know that the corresponding root vectors are
$E_{\alpha_1+\alpha_2}=-E_2E_1+q^{-1}E_1E_2
$
and
$F_{\alpha_1+\alpha_2}=\frac{1}{2}qF_2F_1-\frac{1}{2}q^{2}F_1F_2$ by \eqref{comp1}.
Then we have
\begin{equation}\label{sym5}
\left.
\begin{array}{rl}
&B_{VV}\circ(T_{V}\otimes
T_{V})(\mathfrak{R})(v_{i}\otimes v_{j})\\
=&B_{VV}\circ(T_{V}\otimes
T_{V})\left((q_{\alpha_1+\alpha_2}-q_{\alpha_1+\alpha_2}^{-1})E_{\alpha_1+\alpha_2}\otimes F_{\alpha_1+\alpha_2}\right)(v_{i}\otimes v_{j})\,+\,\textrm{others}\\
=&B_{VV}\circ(q$-$q^{-1})(T_{V}{\otimes}
T_{V})\left((-E_2E_1+q^{-1}E_1E_2){\otimes}(\frac{1}{2}qF_2F_1-\frac{1}{2}q^{2}F_1F_2)\right)(v_{i}{\otimes} v_{j})$+$\textrm{others}\\
=&B_{VV}\circ(q$-$q^{-1})(T_{V}{\otimes}
T_{V})(\frac{E_{2}E_{1}{\otimes}F_{1}F_{2}}{2q^{-2}}$-$\frac{E_{2}E_{1}{\otimes}
F_{2}F_{1}}{2q}$-$\frac{E_{1}E_{2}{\otimes}
F_{1}F_{2}}{2q}$+$\frac{1}{2}E_{1}E_{2}{\otimes} F_{2}F_{1})(v_{i}{\otimes}
v_{j})$+$\textrm{others}.
\end{array}
\right.
\end{equation}
Now,
firstly,
we analyze explicitly the entries obtained by the operator
$T_{V}\otimes T_{V}(E_{2}E_{1}\otimes F_{1}F_{2})$.
If $T_{V}(E_{2}E_{1})(v_i)=v_k$ and
$T_{V}(F_{1}F_{2})(v_j)=v_l$,
then we obtain equivalently
$T_{V}(F_{1}F_{2})(v_k)=v_i$
and
$T_V(E_2E_1)(v_l)=v_j$.
Thus
$R_{VV}{}^{ij}_{kl}=\frac{1}{2}q^{2}(q-q^{-1})q^{(\mu_{k},\mu_{l})}
=\frac{1}{2}q^{2}(q-q^{-1})q^{(\mu_{i},\mu_{j})}=R_{VV}{}^{lk}_{ji}$,
and
$(PR_{VV})^{ji}_{kl}=(PR_{VV})^{kl}_{ji}$.
Secondly,
suppose that there exist $v_i, v_j$ and $v_k, v_l$ such that
$T_{V}(E_{2}E_{1})(v_i)=v_k,
T_{V}(F_{2}F_{1})(v_j)=v_l,$
by the action of spin module,
which is equivalent to that
$
T_{V}(F_{1}F_{2})(v_k)=v_i,
T_{V}(E_{1}E_{2})(v_l)=v_j.
$
That is,
\begin{equation}\label{sym6}
(T_{V}{\otimes}
T_{V})(E_{2}E_{1}{\otimes} F_{2}F_{1})(v_i{\otimes} v_j){=}v_k{\otimes} v_l
\Leftrightarrow
(T_{V}{\otimes}
T_{V})(E_{1}E_{2}{\otimes} F_{1}F_{2})(v_l{\otimes} v_k){=}v_j{\otimes} v_i.
\end{equation}
Moreover,
by \eqref{sym5},
we know that the coefficients of $E_{2}E_{1}{\otimes} F_{2}F_{1}$ and $E_{1}E_{2}{\otimes} F_{1}F_{2}$ are equal,
and both are $\frac{1}{2}q$ (i.e. satisfy the condition $x^{(\underline{r},\underline{i})}_{(\underline{t},\underline{j})}
=x^{(\underline{t},\underline{j})}_{(\underline{r},\underline{i})}$ in \eqref{sym4}).
So the corresponding entries have the equality
$R_{VV}{}^{ij}_{kl}=\frac{1}{2}q(q-q^{-1})q^{(\mu_{k},\mu_{l})}
=\frac{1}{2}q(q-q^{-1})q^{(\mu_{i},\mu_{j})}=R_{VV}{}^{lk}_{ji}.$
For the situation of operator $\frac{1}{2}E_{1}E_{2}\otimes F_{2}F_{1}$ can be analyzed in a similar way.

The proof is complete.
\end{proof}

According to Lemma \ref{pair}, we need to seek the pair $(R, R')$ to
determine the dually-paired braided groups $V(R',R)$ and
$V^{\vee}(R^{\prime},R_{21}^{-1})$. To this end, we have to figure
out the minimal polynomial of the braiding $PR_{VV}$ in advance.
Notice the size of the matrix $PR_{VV}$ is $64\times 64$ due to
$\dim V=8$, it is not easy to obtain its minimal polynomial
directly. In the course of demonstration of the following
Proposition, we will give an ingenious method to capture the minimal
polynomial by taking advantage of nice features of the
representation involved. For simplicity, write $\hat{R}_{VV}$ for
$PR_{VV}$.
\begin{proposition}
Associated to the spin $U_q(B_3)$-module $V$, the braiding matrix $\hat{R}_{VV}$ obeys the minimal
polynomial equation
$$(\hat{R}_{VV}+q^{-\frac{1}{4}}I)(\hat{R}_{VV}-q^{\frac{3}{4}}I)(\hat{R}_{VV}+q^{-\frac{9}{4}}I)(\hat{R}_{VV}+q^{-\frac{21}{4}}I)=0.$$
\end{proposition}
\begin{proof}
Note that
$V^{\otimes 2}=V_{1}\oplus V_{2}\oplus
V_{3}\oplus V_{4}$ \cite{ful}, where $V_{i}$'s are the irreducible
submodules with highest weights
$\varepsilon_{1}+\varepsilon_{2}+\varepsilon_{3},
\varepsilon_{1}+\varepsilon_{2}, \varepsilon_{1}, 0$, respectively.
This means that $\hat R_{VV}$ has $4$ different eigenvalues (since
the braiding $\hat R_{VV}$ as module homomorphism is of diagonal
type, by Proposition 4.1), denoted by $x_{1},x_{2},x_{3},x_{4}$.
Thus $f(t)=(t-x_{1})(t-x_{2})(t-x_{3})(t-x_{4})$ is the minimal
polynomial of $\hat R_{VV}$, so $\mathcal{M}=f(\hat{R}_{VV})=0$.

Let $\triangle_{1}, \triangle_{2}, \triangle_{3}, \triangle_{4} $
denote the elementary symmetric polynomials:
$x_{1}+x_{2}+x_{3}+x_{4},
x_{1}x_{2}+x_{1}x_{3}+x_{2}x_{3}+(x_{1}+x_{2}+x_{3})x_{4},
x_{1}x_{2}x_{3}+x_{1}x_{2}x_{4}+x_{1}x_{3}x_{4}+x_{2}x_{3}x_{4},
x_{1}x_{2}x_{3}x_{4}$, respectively.

Relative to the order of the fixed basis $v_i\otimes v_j$ $(1\le i,
j\le 8)$, we will consider some special rows in matrix
$\hat{R}_{VV}-x_{i}I$.
Obviously,
there is only one nonzero entry in the row (88),
that is,
$R_{VV}{}^{88}_{88}=q^{\frac{3}{4}}$.
So we have
\begin{equation}\label{f1}
\mathcal{M}^{88}_{88}
=(q^{\frac{3}{4}}-x_1)(q^{\frac{3}{4}}-x_2)(q^{\frac{3}{4}}-x_3)(q^{\frac{3}{4}}-x_4).
\end{equation}
From it,
then we obtain $x_i=q^{\frac{3}{4}}$ for a certain $x_i$.
Next,
we consider the row (12).
Only two nonzero entries occur in row (12),
which locate at columns (12) and (21), namely,
$(\hat{R}_{VV}-x_{i}I)~^{1}_{1}{}^{2}_{2}=-x_{i}, $
$(\hat{R}_{VV}-x_{i}I)~^{1}_{2}{}^{2}_{1}=q^{\frac{1}{4}}$. Now we
turn to seek those nonzero entries appearing in row $(12)$ of matrix
$(\hat{R}_{VV}-x_{i}I)(\hat{R}_{VV}-x_{j}I)$. For any column $(mn)$,
we have
$$
[(\hat{R}_{VV}-x_{i}I)(\hat{R}_{VV}-x_{j}I)]^{1}_{m}{}^{2}_{n}
=(\hat{R}_{VV}-x_{i}I)^{1}_{1}{}^{2}_{2}(\hat{R}_{VV}-x_{j}I)^{1}_{m}{}^{2}_{n}
+(\hat{R}_{VV}-x_{i}I)^{1}_{2}{}^{2}_{1}(\hat{R}_{VV}-x_{j}I)^{2}_{m}{}^{1}_{n}.
$$
From the above equality, we need to further find those nonzero
entries in row $(21)$ of matrix $\hat{R}_{VV}-x_{j}I$, listed as
$(\hat{R}_{VV}-x_{j}I)~^{2}_{1}{}^{1}_{2}=q^{\frac{1}{4}},
(\hat{R}_{VV}-x_{j}I)~^{2}_{2}{}^{1}_{1}=q^{\frac{1}{4}}(q^{\frac{1}{2}}-q^{-\frac{1}{2}})-x_{j}.
$ Hence,
\begin{gather*}
\left.
\begin{array}{rl}
[(\hat{R}_{VV}-x_{i}I)(\hat{R}_{VV}-x_{j}I)]^{1}_{1}{}^{2}_{2}
&=(\hat{R}_{VV}-x_{i}I)^{1}_{1}{}^{2}_{2}(\hat{R}_{VV}-x_{j}I)^{1}_{1}{}^{2}_{2}
+(\hat{R}_{VV}-x_{i}I)^{1}_{2}{}^{2}_{1}(\hat{R}_{VV}-x_{j}I)^{2}_{1}{}^{1}_{2}\\
&=x_1x_2+q^{\frac{1}{2}},
\end{array}
\right.
\\
\left.
\begin{array}{rl}
[(\hat{R}_{VV}-x_{i}I)(\hat{R}_{VV}-x_{j}I)]^{1}_{2}{}^{2}_{1}
&=(\hat{R}_{VV}-x_{i}I)^{1}_{1}{}^{2}_{2}(\hat{R}_{VV}-x_{j}I)^{1}_{2}{}^{2}_{1}
+(\hat{R}_{VV}-x_{i}I)^{1}_{2}{}^{2}_{1}(\hat{R}_{VV}-x_{j}I)^{2}_{2}{}^{1}_{1}\\
&=-q^{\frac{1}{4}}(x_1+x_2)+(q-1).
\end{array}
\right.
\end{gather*}
Namely, those nonzero entries still locate at columns (12) and (21)
in row $(12)$ of matrix
$(\hat{R}_{VV}-x_{i}I)(\hat{R}_{VV}-x_{j}I)$. According to the same
analysis, the nonzero entries in row $(12)$ of
$(\hat{R}_{VV}-x_{1}I)(\hat{R}_{VV}-x_{2}I)(\hat{R}_{VV}-x_{3}I)$
still lie at columns $(12)$ and $(21)$, given by
\begin{gather*}
\left.
\begin{array}{rl}
&[(\hat{R}_{VV}-x_{1}I)(\hat{R}_{VV}-x_{2}I)(\hat{R}_{VV}-x_{3}I)]^{1}_{1}{}^{2}_{2}\\
=&[(\hat{R}_{VV}-x_{1}I)(\hat{R}_{VV}-x_{2}I)]^{1}_{1}{}^{2}_{2}(\hat{R}_{VV}-x_{3}I)^{1}_{1}{}^{2}_{2}
+[(\hat{R}_{VV}-x_{1}I)(\hat{R}_{VV}-x_{2}I)]^{1}_{2}{}^{2}_{1}(\hat{R}_{VV}-x_{3}I)^{2}_{1}{}^{1}_{2}\\
=&-x_1x_2x_3-q^{\frac{1}{2}}(x_1+x_2+x_3)+q^{\frac{1}{4}}(q-1),
\end{array}
\right.
\\
\left.
\begin{array}{rl}
&[(\hat{R}_{VV}-x_{1}I)(\hat{R}_{VV}-x_{2}I)(\hat{R}_{VV}-x_{3}I)]^{1}_{2}{}^{2}_{1}\\
=&[(\hat{R}_{VV}-x_{1}I)(\hat{R}_{VV}-x_{2}I)]^{1}_{1}{}^{2}_{2}(\hat{R}_{VV}-x_{3}I)^{1}_{2}{}^{2}_{1}
+[(\hat{R}_{VV}-x_{1}I)(\hat{R}_{VV}-x_{2}I)]^{1}_{2}{}^{2}_{1}(\hat{R}_{VV}-x_{3}I)^{2}_{2}{}^{1}_{1}\\
=&q^{\frac{1}{4}}(x_1x_2+x_1x_3+x_2x_3)-(q-1)(x_1+x_2+x_3)+q^{\frac{1}{4}}(q-q^{-1})(q^{\frac{1}{2}}-q^{-\frac{1}{2}})+q^{\frac{3}{4}}.
\end{array}
\right.
\end{gather*}
Thereby, by the same reasoning, we get the entries lying at row
$(12)$ and columns (12), (21) of matrix $\mathcal
M=(\hat{R}_{VV}-x_{1}I)(\hat{R}_{VV}-x_{2}I)(\hat{R}_{VV}-x_{3}I)(\hat{R}_{VV}-x_{4}I)$
as follows:
\begin{gather}
\left.
\begin{array}{rl}
\mathcal M^{12}_{12}
=&[(\hat{R}_{VV}-x_{1}I)(\hat{R}_{VV}-x_{2}I)(\hat{R}_{VV}-x_{3}I)]^{1}_{1}{}^{2}_{2}(\hat{R}_{VV}-x_{4}I)^{1}_{1}{}^{2}_{2}\\
&+[(\hat{R}_{VV}-x_{1}I)(\hat{R}_{VV}-x_{2}I)(\hat{R}_{VV}-x_{3}I)]^{1}_{2}{}^{2}_{1}(\hat{R}_{VV}-x_{4}I)^{2}_{1}{}^{1}_{2},\\
=&\triangle_{4}+q^{\frac{1}{2}}\triangle_{2}-q^{\frac{1}{4}}(q-1)\triangle_{1}+(q-1)^{2}+q,
\end{array}
\right.\label{f2}
\\
\left.
\begin{array}{rl}
\mathcal M^{12}_{21}
=&[(\hat{R}_{VV}-x_{1}I)(\hat{R}_{VV}-x_{2}I)(\hat{R}_{VV}-x_{3}I)]^{1}_{1}{}^{2}_{2}(\hat{R}_{VV}-x_{4}I)^{1}_{2}{}^{2}_{1}\\
&+[(\hat{R}_{VV}-x_{1}I)(\hat{R}_{VV}-x_{2}I)(\hat{R}_{VV}-x_{3}I)]^{1}_{2}{}^{2}_{1}(\hat{R}_{VV}-x_{4}I)^{2}_{2}{}^{1}_{1},\\
=&{-}q^{\frac{1}{4}}\triangle_{3}{+}(q{-}1)\triangle_{2}
{-}[q^{\frac{1}{4}}(q{-}1)(q^{\frac{1}{2}}{-}q^{{-}\frac{1}{2}}){+}
q^{\frac{3}{4}}]\triangle_{1}{+}(q-1)(q^{\frac{3}{2}}+q^{-\frac{1}{2}}).
\end{array}
\right.
\label{f3}
\end{gather}
Similarly,
Associated with nonzero entries in rows (58), (67), (76), (85):
\begin{gather*}
(\hat{R}_{VV}-x_{i})^{58}_{58}=-x_{i},\quad
(\hat{R}_{VV}-x_{i})^{58}_{85}=q^{-\frac{1}{4}}=(\hat{R}_{VV}-x_{i})^{85}_{58},\quad
(\hat{R}_{VV}-x_{i})^{67}_{67}=-x_i,\\
(\hat{R}_{VV}-x_{i})^{67}_{76}=q^{-\frac{1}{4}}=(\hat{R}_{VV}-x_{i})^{76}_{67},\quad
(\hat{R}_{VV}-x_{i})^{67}_{85}=q^{-\frac{1}{4}}(q^{\frac{1}{2}}-q^{-\frac{1}{2}})=(\hat{R}_{VV}-x_{i})^{85}_{67},\\
(\hat{R}_{VV}-x_{i})^{76}_{76}=q^{-\frac{1}{4}}(q-q^{-1})-x_i,\quad
(\hat{R}_{VV}-x_{i})^{76}_{85}=-q^{-\frac{5}{4}}(q^{\frac{1}{2}}-q^{-\frac{1}{2}})=(\hat{R}_{VV}-x_{i})^{85}_{76},\\
(\hat{R}_{VV}-x_{i})^{85}_{85}=q^{-\frac{1}{4}}(q-1)(1+q^{-2})-x_i,
\end{gather*}
we obtain
\begin{equation}\label{f4}
\mathcal{M}^{58}_{58}=\triangle_4{+}q^{{-}\frac{1}{2}}\triangle_2{-}q^{{-}\frac{3}{4}}(q{-}1)(1{+}q^{{-}2})\triangle_1{+}q^{{-}1}
{+}(1-q^{-1})^{2}(1+q^{-2})(1+q+q^{-1}).
\end{equation}
Then according to \eqref{f1}--\eqref{f4},
and $\mathcal{M}=0$,
we obtain the following equations:
$$
\left\{
\begin{array}{l}
\triangle_{4}-q^{\frac{3}{4}}\triangle_{3}+q^{\frac{3}{2}}\triangle_{2}-q^{\frac{9}{4}}\triangle_{1}+q^{3}=0,\\
\triangle_{4}+q^{\frac{1}{2}}\triangle_{2}-q^{\frac{1}{4}}(q-1)\triangle_{1}+(q-1)^{2}+q=0,\\
{-}q^{\frac{1}{4}}\triangle_{3}{+}(q{-}1)\triangle_{2}
{-}[q^{\frac{1}{4}}(q{-}1)(q^{\frac{1}{2}}{-}q^{{-}\frac{1}{2}}){+}
q^{\frac{3}{4}}]\triangle_{1}{+}(q-1)(q^{\frac{3}{2}}+q^{-\frac{1}{2}})=0,\\
\triangle_4+q^{-\frac{1}{2}}\triangle_2{-}q^{-\frac{3}{4}}(q-1)(1+q^{-2})\triangle_1{+}q^{{-}1}
{+}(1-q^{-1})^{2}(1+q^{-2})(1+q+q^{-1})=0.\\
\end{array}
\right.
$$
Solving the system of equations, we get $ \left\{
\begin{array}{l}
\triangle_{1}=q^{\frac{3}{4}}-q^{-\frac{1}{4}}-q^{-\frac{9}{4}}-q^{-\frac{21}{4}},\\
\triangle_{2}=q^{-\frac{5}{2}}+q^{-\frac{11}{2}}+q^{-\frac{15}{2}}-q^{\frac{1}{2}}-q^{-\frac{3}{2}}-q^{-\frac{9}{2}},\\
\triangle_{3}=q^{-\frac{7}{4}}+q^{-\frac{19}{4}}+q^{-\frac{27}{4}}-q^{-\frac{31}{4}},\\
\triangle_{4}=-q^{-7},
\end{array}
\right.
$

and then
$
x_{1}=q^{\frac{3}{4}},
x_{2}=-q^{-\frac{1}{4}},
x_{3}=-q^{-\frac{9}{4}},
x_{4}=-q^{-\frac{21}{4}}.
$
So the proof is complete.
\end{proof}

Having the minimal polynomial equation of the braiding $\hat R_{VV}$
in hand, it is necessary to consider at which eigenvalue of $\hat
R_{VV}$ we should make its normalization in order to construct the
quantum group of higher-one rank we expect? Corresponding to the
$n$-dimensional representation $T_{V}$ with basis $v_{i}$ we choose,
we know that the new simple root vectors $E_\alpha$, $F_\alpha$ and
the group-like element $K_\alpha$ are always identified respectively
as $e^{n}, f_{n}, (m^{+})^{n}_{n}c^{-1}$ from those
double-bosonization constructions in \cite{HH}. Denote by $\mu_{n}$
the weight of $v_{n}$, we obtain
$R_{VV}{}^{nn}_{nn}=q^{(\mu_{n},\mu_{n})}$. In order to construct
the larger quantum group we want, the most important cross relation
is
$e^{n}((m^{+})^{n}_{n}c^{-1})=R^{nn}_{nn}(((m^{+})^{n}_{n}c^{-1}))e^{n}$
because it gives us the length of the new additional simple root
$\alpha$, namely,
$e^{n}((m^{+})^{n}_{n}c^{-1})=q^{(\alpha,\alpha)}(((m^{+})^{n}_{n}c^{-1}))e^{n}$.
From these, we know at which one of eigenvalues we should take its
normalization in virtue of $R_{VV}{}^{nn}_{nn}$ and
$q^{(\alpha,\alpha)}$ such that the value of $R^{nn}_{nn}$ satisfy
our requirement.
Thus we obtain the following
\begin{proposition}\label{constant}
Starting from an appropriate $U_{q}(\mathfrak{g}')$-module with highest weight $\mu$
for a Lie subalgebra $\mathfrak g'\subset\mathfrak g$ of corank $1$,
if we can construct inductively quantum group $U_q(\mathfrak g)$ from $U_{q}(\mathfrak{g}')$ by double-bosonization procedure,
then the quantum group normalization constant is
$\lambda=\frac{q^{(\alpha,\alpha)}}{q^{(\mu,\mu)}},$
where $\alpha$ is the additional simple root for $\mathfrak g$.
\end{proposition}
By Proposition \ref{constant},
we choose the normalization constant $\lambda=q^{-\frac{1}{4}}$ for the construction of $U_q(F_4)$ from this spin representation.
Then setting
\begin{equation}\label{prime}
R=q^{\frac{1}{4}}R_{VV},\qquad
R^{\prime}=R(\hat{R})^{2}+(q^{-2}-q+q^{-5})R\hat{R}+(q^{-7}-q^{-4}-q^{-1})R-(q^{-6}-1)P,
\end{equation}
we obtain $(PR+I)(PR^{\prime}-I)=0$.
The entries in $m^{\pm}$-matrices are given by the following Lemma.
\begin{lemma}\label{lemmf4}
Corresponding to the $8$-dimensional spin representation of $U_{h}(B_{3})$,
the entries in the matrices $m^{\pm}$ we need are listed as follows.
\begin{gather*}
(m^{+})^{3}_{5}=-(q-q^{-1})E_{1}K_{1}^{-\frac{1}{2}}K_{3}^{\frac{1}{2}}, \quad
(m^{+})^{5}_{5}=K_{1}^{-\frac{1}{2}}K_{3}^{\frac{1}{2}},
\\
(m^{+})^{6}_{7}=-(q-q^{-1})E_{2}K_{1}^{-\frac{1}{2}}K_{2}^{-1}K_{3}^{-\frac{1}{2}},\quad
(m^{+})^{7}_{7}=K_{1}^{-\frac{1}{2}}K_{2}^{-1}K_{3}^{-\frac{1}{2}},
\\
(m^{+})^{7}_{8}=-(q^{\frac{1}{2}}-q^{-\frac{1}{2}})E_{3}K_{1}^{-\frac{1}{2}}K_{2}^{-1}K_{3}^{-\frac{3}{2}},\quad
(m^{+})^{8}_{8}=K_{1}^{-\frac{1}{2}}K_{2}^{-1}K_{3}^{-\frac{3}{2}}.
\\
(m^{-})^{5}_{3}=q(q-q^{-1})K_{1}^{\frac{1}{2}}K_{3}^{-\frac{1}{2}}F_{1},\quad
(m^{-})^{5}_{5}=K_{1}^{\frac{1}{2}}K_{3}^{-\frac{1}{2}},
\\
(m^{-})^{7}_{6}=q(q-q^{-1})K_{1}^{\frac{1}{2}}K_{2}K_{3}^{\frac{1}{2}}F_{2},\quad
(m^{-})^{7}_{7}=K_{1}^{\frac{1}{2}}K_{2}K_{3}^{\frac{1}{2}},
\\
(m^{-})^{8}_{7}=q(q^{\frac{1}{2}}-q^{-\frac{1}{2}})K_{1}^{\frac{1}{2}}K_{2}K_{3}^{\frac{3}{2}}F_{3},\quad
(m^{-})^{8}_{8}=K_{1}^{\frac{1}{2}}K_{2}K_{3}^{\frac{3}{2}}.
\end{gather*}
\end{lemma}
With these data in hand, we have the following

\begin{theorem}
Identify elements
$e^{8},
f_{8},
(m^{+})^{8}_{8}c^{-1}$
with the additional simple root vectors
$E_{4}, F_{4}$ and the group-like $K_{4}$,
then the quantum group
$U(V^{\vee}(R^{\prime}, R_{21}^{-1}),\widetilde{U_{q}^{ext}(B_{3})},V(R^{\prime},R))$
is exactly the
$U_{q}(F_{4})$ with $K_{i}^{\frac{1}{2}}$ adjoined, $1\leq i\leq3$.
\end{theorem}
\begin{proof}
According to the cross relations in Theorem \ref{cor1},
we obtain
$$e^{8}(m^{+})^{8}_{8}c^{-1}
=\lambda R^{88}_{88}(m^{+})^{8}_{8}e^{8}c^{-1}
=R^{88}_{88}(m^{+})^{8}_{8}c^{-1}e^{8}
=q^{\frac{1}{4}}R_{VV}{}^{88}_{88}(m^{+})^{8}_{8}c^{-1}e^{8}
=q(m^{+})^{8}_{8}c^{-1}e^{8}.
$$
Then with the above identification,
we have
$E_{4}K_{4}=qK_{4}E_{4}$.
The relations between $E_{4}$ and $K_{i}$ can be obtained by the following relations
\begin{gather*}
(m^{+})^{8}_{8}K_{3}=(m^{+})^{7}_{7},\quad
e^{8}(m^{+})^{5}_{5}=\lambda R^{58}_{58}(m^{+})^{5}_{5}e^{8}=q^{-\frac{1}{4}}(m^{+})^{5}_{5}e^{8},\\
e^{8}(m^{+})^{7}_{7}=\lambda R^{78}_{78}(m^{+})^{7}_{7}e^{8}=q^{\frac{1}{4}}(m^{+})^{7}_{7}e^{8},\quad
e^{8}(m^{+})^{8}_{8}=\lambda R^{88}_{88}(m^{+})^{8}_{8}e^{8}=q^{\frac{3}{4}}(m^{+})^{8}_{8}e^{8}.
\end{gather*}
From the above relations,
we obtain that
$
e^{8}K_{1}=K_{1}e^{8},\,
e^{8}K_{2}=K_{2}e^{8}
$
and
$
e^{8}K_{3}=q^{-\frac{1}{2}}K_{3}e^{8}.
$
Namely,
$$
E_{4}K_{1}=K_{1}E_{4},\quad
E_{4}K_{2}=K_{2}E_{4},\quad
E_{4}K_{3}=q^{-\frac{1}{2}}K_{3}E_{4}.
$$
The relations between $K_{4}=K_{1}^{-\frac{1}{2}}K_{2}^{-1}K_{3}^{-\frac{3}{2}}c^{-1}$ and $E_{i}$ are given by

$E_{1}K_{4}=E_{1}K_{1}^{-\frac{1}{2}}K_{2}^{-1}K_{3}^{-\frac{3}{2}}c^{-1}
=q^{-1}qK_{1}^{-\frac{1}{2}}K_{2}^{-1}K_{3}^{-\frac{3}{2}}c^{-1}E_{1}=K_{4}E_{1},$

$
E_{2}K_{4}=E_{2}K_{1}^{-\frac{1}{2}}K_{2}^{-1}K_{3}^{-\frac{3}{2}}c^{-1}
=q^{\frac{1}{2}}q^{-2}q^{\frac{2}{3}}K_{1}^{-\frac{1}{2}}K_{2}^{-1}K_{3}^{-\frac{3}{2}}c^{-1}E_{2}=K_{4}E_{2},
$

$
E_{3}K_{4}=E_{3}K_{1}^{-\frac{1}{2}}K_{2}^{-1}K_{3}^{-\frac{3}{2}}c^{-1}
=qq^{-\frac{3}{2}}K_{1}^{-\frac{1}{2}}K_{2}^{-1}K_{3}^{-\frac{3}{2}}c^{-1}E_{3}=q^{-\frac{1}{2}}K_{4}E_{3}.
$

We will check the $q$-Serre relations related to $E_{i}$,
which also can be obtained by the cross relations in Theorem \ref{cor1}.
In fact,
by Lemma \ref{lemmf4} and Theorem \ref{cor1},
we have
$$
(m^{+})^{3}_{5}=-(q-q^{-1})E_{1}(m^{+})^{5}_{5},\quad
e^{8}(m^{+})^{3}_{5}=\lambda R^{38}_{38}(m^{+})^{3}_{5}e^{8}=q^{-\frac{1}{4}}(m^{+})^{3}_{5}e^{8},
$$
These yield
$
e^{8}E_{1}=E_{1}e^{8}
$,
namely,
$
E_{4}E_{1}=E_{1}E_{4}.
$
On the other hand,
starting from the following cross relations
(followed by Lemma \ref{lemmf4} and Theorem \ref{cor1}, respectively)
$$
(m^{+})^{6}_{7}=-(q-q^{-1})E_{2}(m^{+})^{7}_{7},\quad
e^{8}(m^{+})^{6}_{7}=\lambda R^{68}_{68}(m^{+})^{6}_{7}e^{8}=q^{\frac{1}{4}}(m^{+})^{6}_{7}e^{8},
$$
we obtain
$
e^{8}E_{2}=E_{2}e^{8},
$
namely,
$
E_{4}E_{2}=E_{2}E_{4}.
$
In order to check the relation between $E_{4}~(e^{8})$ and $E_{3}$,
we consider the cross relations from \ref{lemmf4} and Theorem \ref{cor1},
\begin{gather*}
(m^{+})^{7}_{8}=-(q^{\frac{1}{2}}-q^{-\frac{1}{2}})E_{3}(m^{+})^{8}_{8},\quad
e^{7}(m^{+})^{8}_{8}=\lambda R^{87}_{87}(m^{+})^{8}_{8}e^{7}=q^{\frac{1}{4}}(m^{+})^{8}_{8}e^{7},\\
e^{8}(m^{+})^{7}_{8}=q^{\frac{1}{4}}(m^{+})^{7}_{8}e^{8}+(q^{\frac{1}{2}}-q^{-\frac{1}{2}})q^{\frac{1}{4}}(m^{+})^{8}_{8}e^{7},
\end{gather*}
Thus we obtain that
$
e^{7}=q^{-\frac{1}{2}}E_{3}e^{8}-e^{8}E_{3}.
$
Further more,
we need the relation $e^{7}E_{3}=q^{\frac{1}{2}}E_{3}e^{7}$,
 which can be obtained by $e^{7}(m^{+})^{7}_{8}=\lambda R^{77}_{77}(m^{+})^{7}_{8}e^{7}=q^{\frac{3}{4}}(m^{+})^{7}_{8}e^{7}$ (by Theorem \ref{cor1}).
Combining with the above equality $e^{7}=q^{-\frac{1}{2}}E_{3}e^{8}-e^{8}E_{3},$
we obtain $e^{8}(E_{3})^{2}-(q^{\frac{1}{2}}+q^{-\frac{1}{2}})E_{3}e^{8}E_{3}+(E_{3})^{2}e^{8}=0,$
namely,
$$
E_{4}(E_{3})^{2}-(q^{\frac{1}{2}}+q^{-\frac{1}{2}})E_{3}E_{4}E_{3}+(E_{3})^{2}E_{4}=0.
$$
On the other hand,
in order to check another $q$-Serre relation,
we need to figure out the relation between $e^{7}$ and $e^{8}$,
that is,
$e^{8}e^{7}=R'{}^{78}_{ab}e^{a}e^{b}$.
Since the nonzero entries in rows (78) and (87) of $R$ are
$
R^{78}_{78}=q^{\frac{1}{2}},
$
$
R^{78}_{87}=q-1,
$
$
R^{87}_{87}=q^{\frac{1}{2}},
$
then we have
\begin{gather*}
(R\hat{R}^2)^{78}_{78}=R^{78}_{78}R^{87}_{87}R^{78}_{78}+R^{78}_{87}R^{78}_{87}R^{78}_{78}=q^{\frac{3}{2}}+q^{\frac{1}{2}}(q-1)^{2},\\
(R\hat{R}^2)^{78}_{87}=R^{78}_{78}R^{87}_{87}R^{78}_{87}+R^{78}_{87}R^{78}_{78}R^{87}_{87}+R^{78}_{87}R^{78}_{87}R^{78}_{87}=q^{3}-q^2+q-1,\\
(R\hat{R})^{78}_{78}=R^{78}_{87}R^{78}_{78}=q^{\frac{1}{2}}(q-1),\quad
(R\hat{R})^{78}_{87}=R^{78}_{78}R^{87}_{87}+R^{78}_{87}R^{78}_{87}=q^2-q+1.
\end{gather*}
From the formula \eqref{prime} of $R, R'$ we get,
nonzero entries in rows (78) of $R^{\prime}$ are
\begin{gather*}
R'{}^{78}_{78}=q^{\frac{3}{2}}+q^{\frac{1}{2}}(q-1)^{2}+(q^{-2}-q+q^{-5})q^{\frac{1}{2}}(q-1)+(q^{-7}-q^{-4}-q^{-1})q^{\frac{1}{2}},
\\
R'{}^{78}_{87}=q^{3}-q^2+q-1+(q^{-2}-q+q^{-5})(q^2-q+1)+(q^{-7}{-}q^{-4}{-}q^{-1})(q-1)-(q^{-6}{-}1).
\end{gather*}
So
$e^{8}e^{7}=R'{}^{78}_{78}e^{7}e^{8}+R'{}^{78}_{87}e^{8}e^{7}$,
then we obtain
$e^{8}e^{7}=q^{\frac{1}{2}}e^{7}e^{8}.
$
Combining with the above equality $e^{7}=q^{-\frac{1}{2}}E_{3}e^{8}-e^{8}E_{3}$ again,
we obtain $(e^{8})^{2}E_{3}-(q^{\frac{1}{2}}+q^{-\frac{1}{2}})e^{8}E_{3}e^{8}+E_{3}(e^{8})^{2}=0$,
namely,
$$
(E_{4})^{2}E_{3}-(q^{\frac{1}{2}}+q^{-\frac{1}{2}})E_{4}E_{3}E_{4}+E_{3}(E_{4})^{2}=0.
$$

Thence,
we obtain that the Cartan matrix of the new quantum group is
$
\left(
\begin{array}{cccc}
2&-1&0&0\\
-1&2&-2&0\\
0&-1&2&-1\\
0&0&-1&2
\end{array}
\right).
$

This completes the proof.
\end{proof}

\subsection{Type-crossing of $U_{q}(G_{2})$ via $U_q(A_1)$}
From the classical cases, we know that $U_{q}(\mathfrak{sl}_{3})$
was constructed by choosing braided groups corresponding to the
vector representation of $U_{q}(\mathfrak{sl}_{2})$. This inspires
us to consider the spin $\frac{3}{2}$ representation
$T_{\textrm{sp}}$ of $U_{q}(\mathfrak{sl}_{2})$ when constructing
$U_{q}(G_{2})$,
 here $T_{\textrm{sp}}$ is taken from the $4$-dimensional homogeneous subspace of degree $3$ of braided group $\mathcal{O}(\mathbb{C}_{q}^{2})$ generated by $x, y$  with $yx=qxy$ in the category of $U_{q}(\mathfrak{sl}_{2})$-modules.
In fact, this braided group $\mathcal A_q(2)=\mathbb{C}_q\langle x,
y\rangle$ is a left $U_{q}(\mathfrak{sl}_{2})$-module algebra
(\cite{H}). Explicitly, starting from the $2$-dimensional vector
representation of $U_{q}(\mathfrak{sl}_{2})$, we have the standard
$4\times 4$ $R$-matrix below
$$R=
\left(
\begin{array}{cccc}
q & ~0 & ~0 & ~0\\
0 & ~1 & ~q-q^{-1} & ~0\\
0 & ~0& ~1 & ~0\\
0 & ~0 & ~0 & ~q
\end{array}
\right).\qquad T= \left(
\begin{array}{cccc}
a&b\\
c&d
\end{array}
\right).
$$
For such $R$, the FRT-bialgebra $A(R)$ has a quotient Hopf algebra
generated by the entries $a, b, c, d$ in the above matrix $T$,
denoted by $\mathcal{O}_{q}(SL(2))$. Dually, there is a right
$\mathcal{O}_{q}(SL(2))$-comodule algebra structure on $\mathcal
A_q(2)$, namely
\begin{equation}\label{vector}
\rho(
y,\ x
)
=
(y\otimes a+x\otimes c, y\otimes b+x\otimes d)
=(y, \ x)\otimes T.
\end{equation}
On the $4$-dimensional homogeneous subspace of degree $3$ consisting
of $x^{i-1}y^{3-(i-1)},i=1,2,3,4$, denoted by $\mathcal
A_q(2)^{(3)}$, there exists a right
$\mathcal{O}_{q}(SL(2))$-subcomodule induced by \eqref{vector}, its
corepresentation matrix is still denoted by $T$, given by
\begin{equation}\label{spin}
T= \left(
\begin{array}{cccc}
t^{1}_{1}&t^{1}_{2}&t^{1}_{3}&t^{1}_{4}\\
t^{2}_{1}&t^{2}_{2}&t^{2}_{3}&t^{2}_{4}\\
t^{3}_{1}&t^{3}_{2}&t^{3}_{3}&t^{3}_{4}\\
t^{4}_{1}&t^{4}_{2}&t^{4}_{3}&t^{4}_{4}
\end{array}
\right) = \left(
\begin{array}{cccc}
a^3 & a^2b & ab^2 & b^3\\
{[3]_q} ca^2 & a^{2}d+(q^{-2}{+}1)cab & q^{-2}c b^2+(q^{-2}{+}1)adb & {[3]_q} d b^2\\
{[3]_q} c^2a & q^{-2}c^2b +(q^{-2}{+}1)acd & ad^2+(q^{-2}{+}1)cdb & {[3]_q} d^2b\\
c^3 & c^2d & cd^2 & d^3
\end{array}
\right).
\end{equation}

Since the dual pair of $(U_q^{\textrm{ext}}(\mathfrak sl_2),\mathcal{O}_{q}(SL(2)))$,
then this 4-dimensional right $\mathcal{O}_{q}(SL(2))$-comodule induces a left $U_q(\mathfrak sl_2)$-module,
called as spin $\frac{3}{2}$
representation $T_{\textrm{sp}}$,
give by
\begin{equation}\label{spinrep}
K_{1}\left(
\begin{array}{c}
v_{1}\\
v_{2}\\
v_{3}\\
v_{4}
\end{array}
\right) =\left(
\begin{array}{c}
q^{-3}v_{1}\\
q^{-1}v_{2}\\
qv_{3}\\
q^{3}v_{4};
\end{array}
\right), \quad E_{1} \left(
\begin{array}{c}
v_{1}\\
v_{2}\\
v_{3}\\
v_{4}
\end{array}
\right) = \left(
\begin{array}{c}
q^{-\frac{3}{2}}[3]^{\frac{1}{2}}v_{2}\\
q^{-\frac{1}{2}}[2]v_{3}\\
q^{\frac{1}{2}}[3]^{\frac{1}{2}}v_{4}\\
0
\end{array}
\right), \quad F_{1} \left(
\begin{array}{c}
v_{1}\\
v_{2}\\
v_{3}\\
v_{4}
\end{array}
\right) = \left(
\begin{array}{c}
0\\
q^{\frac{3}{2}}[3]^{\frac{1}{2}}v_{1}\\
q^{\frac{1}{2}}[2]v_{2}\\
q^{-\frac{1}{2}}[3]^{\frac{1}{2}}v_{3}
\end{array}
\right).
\end{equation}
Where $v_i$ denote the basis of
representation space of $T_{\textrm{sp}}$,
with
corresponding weights
$-\frac{3}{2}\alpha_{1},$
$-\frac{1}{2}\alpha_{1},$
$\frac{1}{2}\alpha_{1},$
$\frac{3}{2}\alpha_{1}$.
Moreover,
we observe that $E_{1}^{i}, F_{1}^{i}$ are zero actions for all
$i\geq 4$, so we needn't to consider every summand in the expression
of universal $\mathcal{R}$-matrix. 
Then $$
B_{VV}\circ(T_{V}\otimes T_{V})(\mathcal{R})(v_{i}\otimes v_{j})
=B_{VV}\circ(T_{V}\otimes
T_{V})(\sum\limits_{i=0}^{3}c_{i}E_{1}^{i}\otimes
F_{1}^{i})(v_{i}\otimes v_{j}),$$
here,
$c_{n}=\frac{(1-q^{-2})^{n}q^{\frac{n(n+1)}{2}}}{[n]_{q}!}.$ With these, we obtain that the
following $16\times 16$ $R$-matrix $R_{VV}$ corresponding to the
spin representation \eqref{spinrep}.
$$
\left(
\begin{array}{cccccccccccccccc}
q^{\frac{9}{2}}&0&0&0&0&0&0&0&0&0&0&0&0&0&0&0\\
0&q^{\frac{3}{2}}&0&0&\frac{c_{1}[3]}{q^{-\frac{3}{2}}}&0&0&0&0&0&0&0&0&0&0&0\\
0&0&q^{-\frac{3}{2}}&0&0&\frac{c_{1}[3]^{\frac{1}{2}}}{[2]^{-1}q^{\frac{1}{2}}}&0&0&\frac{c_{2}[3]}{[2]^{-2}q^{\frac{3}{2}}}&0&0&0&0&0&0&0\\
0&0&0&q^{-\frac{9}{2}}&0&0&\frac{c_{1}[3]}{q^{\frac{5}{2}}}&0&0&\frac{c_{2}[3]}{[2]^{-2}q^{\frac{5}{2}}}&0&0&\frac{c_{3}[3]^{2}}{[2]^{-2}q^{\frac{9}{2}}}&0&0&0\\
0&0&0&0&q^{\frac{3}{2}}&0&0&0&0&0&0&0&0&0&0&0\\
0&0&0&0&0&q^{\frac{1}{2}}&0&0&\frac{c_{1}[3]^{\frac{1}{2}}}{[2]^{-1}q^{\frac{1}{2}}}&0&0&0&0&0&0&0\\
0&0&0&0&0&0&q^{-\frac{1}{2}}&0&0&\frac{c_{1}[2]^{2}}{q^{\frac{1}{2}}}&0&0&\frac{c_{2}[3]}{[2]^{-2}q^{\frac{5}{2}}}&0&0&0\\
0&0&0&0&0&0&0&q^{-\frac{3}{2}}&0&0&\frac{c_{1}[3]^{\frac{1}{2}}}{[2]^{-1}q^{\frac{1}{2}}}&0&0&\frac{c_{2}[3]}{[2]^{-2}q^{\frac{3}{2}}}&0&0\\
0&0&0&0&0&0&0&0&q^{-\frac{3}{2}}&0&0&0&0&0&0&0\\
0&0&0&0&0&0&0&0&0&q^{-\frac{1}{2}}&0&0&\frac{c_{1}[3]}{q^{\frac{5}{2}}}&0&0&0\\
0&0&0&0&0&0&0&0&0&0&q^{\frac{1}{2}}&0&0&\frac{c_{1}[3]^{\frac{1}{2}}}{[2]^{-1}q^{\frac{1}{2}}}&0&0\\
0&0&0&0&0&0&0&0&0&0&0&q^{\frac{3}{2}}&0&0&\frac{c_{1}[3]}{q^{-\frac{3}{2}}}&0\\
0&0&0&0&0&0&0&0&0&0&0&0&q^{-\frac{9}{2}}&0&0&0\\
0&0&0&0&0&0&0&0&0&0&0&0&0&q^{-\frac{3}{2}}&0&0\\
0&0&0&0&0&0&0&0&0&0&0&0&0&0&q^{\frac{3}{2}}&0\\
0&0&0&0&0&0&0&0&0&0&0&0&0&0&0&q^{\frac{9}{2}}
\end{array}
\right)
$$

Obviously, the matrix $PR_{VV}$ is symmetric.
For this $R$-matrix $R_{VV}$,
the existence of $R,R^{\prime}$ is guaranteed by
Remark \ref{rem1}.
We can obtain that the quantum group normalization constant is
$\lambda=\frac{q^{3}}{q^{(\frac{3}{2}\alpha,\frac{3}{2}\alpha)}}=\frac{q^{3}}{q^{\frac{9}{2}}}=q^{-\frac{3}{2}}$ by Proposition \ref{constant}.
For consistency, we denote the generators of
braided groups $V^{\vee}(R^{\prime},R_{21}^{-1})$ induced by the
representation $T_{\textrm{sp}}$ by $f_{i}$, and identify $f_{i}$ as
$x^{i-1}y^{3-(i-1)},i=1,2,3,4$,
then the algebra structure of $V^{\vee}(R^{\prime},R_{21}^{-1})$ can be obtained directly by $yx=qxy$,
not necessary deduced from the matrix $R'$,
which is different from the case of construction for $U_q(F_4)$.
On the other hand,
in order to obtain the
relations in the new quantum group by the generalized
double-bosonization construction theorem, we need to know the matrix
$m^{\pm}$ consisting of FRT-generators, which are listed in the
following lemma.
\begin{lemma}\label{lemf}
Corresponding to the spin $\frac{3}{2}$ representation, the matrices
$m^{\pm}$ are given by
$$
m^{+}= \left(
\begin{array}{cccc}
K_{1}^{\frac{3}{2}} &
c_{1}[3]^{\frac{1}{2}}q^{\frac{3}{2}}K_{1}^{\frac{1}{2}}E_{1}  &
-c_{2}[3]^{\frac{1}{2}}K_{1}^{-\frac{1}{2}}E_{1}^{2}
& c_{3}[3][2]q^{-\frac{9}{2}}K_{1}^{-\frac{3}{2}}E_{1}^{3}\\
o & K_{1}^{\frac{1}{2}}  &c_{1}[2]q^{\frac{1}{2}}K_{1}^{-\frac{1}{2}}E_{1} & -c_{2}[3]^{\frac{1}{2}}[2]q^{-2}K_{1}^{-\frac{3}{2}}E_{1}^{2}\\
0 & 0 & K_{1}^{-\frac{1}{2}} & c_{1}[3]^{\frac{1}{2}}q^{-\frac{1}{2}}K_{1}^{-\frac{3}{2}}E_{1}\\
0 & 0 & 0 &  K_{1}^{-\frac{3}{2}}
\end{array}
\right)
$$
$$
m^{-}= \left(
\begin{array}{cccc}
K_{1}^{-\frac{3}{2}}&0&0&0\\
-c_{1}[3]^{\frac{1}{2}}q^{-\frac{9}{2}}K_{1}^{-\frac{1}{2}}F_{1}&K_{1}^{-\frac{1}{2}}&0&0\\
c_{2}[3]^{\frac{1}{2}}[2]q^{-8}K_{1}^{\frac{1}{2}}F_{1}^{2}&-c_{1}[2]q^{-\frac{5}{2}}K_{1}^{\frac{1}{2}}F_{1}&K_{1}^{\frac{1}{2}}&0\\
-c_{3}[3][2]q^{-\frac{9}{2}}K_{1}^{\frac{3}{2}}F_{1}^{3}&c_{2}[3]^{\frac{1}{2}}K_{1}^{\frac{1}{2}}F_{1}^{2}
&-c_{1}[3]^{\frac{1}{2}}q^{\frac{1}{2}}K_{1}^{\frac{3}{2}}F_{1}&K_{1}^{\frac{3}{2}}
\end{array}
\right)
$$
\end{lemma}
%
 With these, we have
the following
\begin{theorem}\label{g2}
Corresponding to the spin $\frac{3}{2}$ representation, the
normalization constant $\lambda=q^{-\frac{3}{2}}$, and
$R=q^{\frac{3}{2}}R_{VV}$. Identifying $e^{4}, f_{4},
m^{+}{}^{4}_{4}c^{-1}$ with the new additional simple root vectors
$E_{2}, F_{2}$ and the group-like element $K_{2}$, respectively,
then the resulting new quantum group
$U(V^{\vee}(R^{\prime},R_{21}^{-1}),\widetilde{U_{q}^{ext}(\mathfrak{sl}_{2})},V(R^{\prime},R))$
is exactly the $U_{q}(G_{2})$ with $K_{1}^{\frac{1}{2}}$
adjoined.
\end{theorem}
\begin{proof}
We just describe the cross relations and $q$-Serre relations in
positive part, the corresponding relations in negative part can be
obtained in a similar way. Under the identification of Theorem
\ref{g2},
$$
[E_{2},F_{2}]=\frac{K_{2}-K_{2}^{-1}}{q^{3}-q^{-3}},
\Delta(E_{2})=E_{2}\otimes K_{2}+1\otimes E_{2}, \mbox{and}~
\Delta(F_{2})=F_{2}\otimes 1+K_{2}^{-1}\otimes F_{2}
$$
can by obtained directly by Theorem \ref{cor1}. Moreover,
$$
\left.
\begin{array}{ll}
e^{4}(m^{+})^{4}_{4}c^{-1}&=\lambda
R^{4}_{a}{}^{4}_{b}(m^{+})^{a}_{4}e^{b}c^{-1}=
\lambda R^{4}_{4}{}^{4}_{4}(m^{+})^{4}_{4}e^{4}c^{-1}=\lambda \frac{1}{\lambda}R^{4}_{4}{}^{4}_{4}(m^{+})^{4}_{4}c^{-1}e^{4}\\
&=q^{\frac{3}{2}}R_{VV}{}^{4}_{4}{}^{4}_{4}(m^{+})^{4}_{4}c^{-1}e^{4}
=q^{\frac{3}{2}}q^{\frac{6}{2}}(m^{+})^{4}_{4}c^{-1}e^{4}=q^{6}(m^{+})^{4}_{4}c^{-1}e^{4}.
\end{array}
\right.
$$
Namely, we obtain $E_{2}K_{2}=q^{6}K_{2}E_{2}.$ In view of the
equality $(m^{+})^{2}_{2}K_{1}=(m^{+})^{1}_{1}$ obtained by Lemma
\ref{lemf}, and the cross relation $ e^{4}(m^{+})^{i}_{i}=\lambda
R^{i}_{i}{}^{4}_{4}(m^{+})^{i}_{i}e^{4} $, we obtain that
$e^{4}K_{1}=q^{-3}K_{1}e^{4},$ so $ E_{2}K_{1}=q^{-3}K_{1}E_{2}. $
On the other hand, by $K_{2}=K_{1}^{-\frac{3}{2}}c^{-1}$, we obtain
that
$$E_{1}K_{2}=E_{1}K_{1}^{-\frac{3}{2}}c^{-1}=q^{-3}K_{1}^{-\frac{3}{2}}c^{-1}E_{1}=q^{-3}K_{2}E_{1}.$$
Since $F_{1}$ is included in the entry $(m^{-})^{2}_{1}$ of $m^{-}$,
the relation between $E_{2}$ and $F_{1}$ can be obtained by the
following cross relation
$$(m^{-})^{2}_{1}e^{4}=\lambda R^{4}_{a}{}^{2}_{b}e^{a}(m^{-})^{b}_{1}=q^{-\frac{3}{2}}e^{4}(m^{-})^{2}_{1}.$$
Then by the expression of $(m^{-})^{2}_{1}$, we obtain
$$K_{1}^{-\frac{1}{2}}F_{1}e^{4}=q^{-\frac{3}{2}}e^{4}K_{1}^{-\frac{1}{2}}F_{1}=q^{-\frac{3}{2}}q^{\frac{3}{2}}K_{1}^{-\frac{1}{2}}e^{4}F_{1}=K_{1}^{-\frac{1}{2}}e^{4}F_{1},$$
so $F_{1}e^{4}=e^{4}F_{1}$, namely, $[E_{2},F_{1}]=0$.

Let us check the $q$-Serre relations of $E_{i},i=1,2$. Firstly, by
Theorem \ref{cor1}. we have $ e^{4}(m^{+})^{1}_{2}
=q^{-\frac{9}{2}}(m^{+})^{1}_{2}e^{4}+c_{1}[3]q^{-\frac{5}{2}}(m^{+})^{2}_{2}e^{3}$
and $ e^{4}K_{1}=q^{-3}K_{1}e^{4}$. So we obtain
\begin{equation}\label{g1}
e^{4}E_{1}-q^{-3}E_{1}e^{4}=[3]^{\frac{1}{2}}q^{-\frac{5}{2}}e^{3}.
\end{equation}
Next we need to know the relation of $e^{3}, E_{1}$ and $e^{4}$.
$f_{4}f_{3}=x^{3}x^{2}y=q^{-3}x^{2}yx^{3}=q^{-3}f_{3}f_{4}$ by
$yx=qxy$, so $R^{\prime}{}^{3}_{3}{}^{4}_{4}=q^{3}$, then
$e^{4}e^{3}=R^{\prime}{}^{3}_{3}{}^{4}_{4}e^{3}e^{4}=q^{3}e^{3}e^{4}$.
Combining with \eqref{g1}, we have
$(e^{4})^{2}E_{1}-(q^{3}+q^{-3})e^{4}E_{1}e^{4}+E_{1}(e^{4})^{2}=0$,
namely,
$$(E_{2})^{2}E_{1}-
\left[
\begin{array}{c}
2\\
1
\end{array}
\right] _{q^{3}} E_{2}E_{1}E_{2}+E_{1}(E_{2})^{2}=0.$$ On the other
hand, in order to obtain the relation of $e^{3}$ and $E_{1}$,
starting from the following cross relations in the new quantum group
\begin{gather*}
e^{3}(m^{+})^{1}_{2}=c_{1}[3]^{\frac{1}{2}}[2]q^{-\frac{1}{2}}m^{+2}_{2}e^{2}+q^{-\frac{3}{2}}(m^{+})^{1}_{2}e^{3},\\
e^{2}(m^{+})^{1}_{2}=c_{1}[3]q^{\frac{3}{2}}(m^{+})^{2}_{2}e^{1}+q^{\frac{3}{2}}(m^{+})^{1}_{2}e^{2},\quad
e^{1}(m^{+})^{1}_{2}=q^{\frac{9}{2}}(m^{+})^{1}_{2}e^{1},
\end{gather*}
we obtain $ E_{1}e^{3}-qe^{3}E_{1}=q^{-\frac{3}{2}}[2]e^{2},\,
E_{1}e^{2}-q^{-1}e^{2}E_{1}=[3]^{\frac{1}{2}}q^{-\frac{5}{2}}e^{1},$
and $ e^{1}E_{1}=q^{3}E_{1}e^{1}. $ Combining with \eqref{g1} again,
we obtain $
(E_{1})^{4}e^{4}-(q^{-3}+q^{-1}+q+q^{3})(E_{1})^{3}e^{4}E_{1}
+(q^{2}+1+q^{4}+1+q^{-4}+q^{-2})(E_{1})^{2}e^{4}(E_{1})^{2}-(q+q^{3}+q^{-1}+q^{-3})E_{1}e^{4}(E_{1})^{3}+e^{4}(E_{1})^{4}=0,
$ namely,
$$
(E_{1})^{4}E_{2}- \left[
\begin{array}{c}
4\\
1
\end{array}
\right] _{q}(E_{1})^{3}E_{2}E_{1}+ \left[
\begin{array}{c}
4\\
2
\end{array}
\right] _{q}(E_{1})^{2}E_{2}(E_{1})^{2}- \left[
\begin{array}{c}
4\\
3
\end{array}
\right] _{q}E_{1}E_{2}(E_{1})^{3}+E_{2}(E_{1})^{4}=0.
$$
So the Cartan matrix of the resulting quantum group is $ \left(
\begin{array}{cc}
2&-3\\
-1&2
\end{array}
\right), $ the proof is complete.
\end{proof}


\vskip1cm

\end{document}